\documentclass[11pt]{amsart}
\usepackage{amsfonts,amsmath,amsthm,amssymb,hyperref,fancyhdr,float,caption,mathtools}
\usepackage{wasysym}
\usepackage{hyperref,fancyhdr,tikz,mathtools}
\usepackage{enumerate}
\usepackage{graphicx}
\usepackage{latexsym}
\usepackage{amsfonts}
\usepackage[utf8]{inputenc}

\newtheorem{theorem}{Theorem}[section]
\newtheorem{prop}[theorem]{Proposition}

\newtheorem{lemma}[theorem]{Lemma}

\newtheorem{problem}[theorem]{Problem}

\newtheorem{cor}[theorem]{Corollary}

\newcommand{\R}{{\mathbb R}}
\newcommand{\C}{{\mathbb C}}
\newcommand{\Z}{{\mathbb Z}}
\newcommand{\Q}{{\mathbb Q}}

\newcommand{\T}{{\mathbb T}}

\newcommand{\la}{\lambda} 
\newcommand{\al}{\alpha}

\newcommand{\om}{\omega}

\newcommand{\abs}[1]{\left|#1\right|}

\newenvironment{remark}{\refstepcounter{theorem}\par\medskip\noindent{\em
Remark~\thetheorem.}}{\unskip\nobreak\hfill\hbox{ $\oslash$}\par\bigskip}

\newenvironment{question}{\refstepcounter{theorem}\par\medskip\noindent{\em 
Question~\thetheorem.}}{\unskip\normalfont \unskip\nobreak\hfill\hbox{\bigskip}}

\newenvironment{example}{\refstepcounter{theorem}\par\medskip\noindent{\em
Example~\thetheorem.}}{\unskip\nobreak\hfill\hbox{ $\oslash$}\par\bigskip}

\newenvironment{definition}{\refstepcounter{theorem}\par\medskip\noindent{\em
Definition~\thetheorem.}}

\newcommand{\got}[1]{\mathfrak{#1}}

\newcommand{\imm}{\hookrightarrow}
\newcommand{\adjustedarrow}[1]{\xhookrightarrow{\raisebox{-1.5pt}[3pt][0pt]{\ensuremath{\scriptstyle{#1}}}\,}}

\newcommand{\immrk}{\adjustedarrow{\R^k}}

\newcommand{\immG}{\adjustedarrow{G}}

\newcommand{\sympplain}{\mathrm{Symp}}
\newcommand{\sympG}{\sympplain^{2n,G}}
\newcommand{\sympT}{\sympplain^{2n,\T^n}_{\mathrm{T}}}

\newcommand{\vol}{\mathrm{vol}}

\newcommand{\cB}[2]{c_\mathrm{B}^{#1,#2}}
\newcommand{\cBmk}{\cB{m}{k}}

\newcommand{\toricpack}{\mathcal{T}}

\usepackage{geometry} 
\newlength{\leftside}
\newlength{\rightside}
\title{
Hamiltonian and symplectic symmetries: an introduction}

\author{\'Alvaro Pelayo}

\begin{document}

\maketitle

\begin{center}
\emph{In memory of Professor J.J.~Duistermaat (1942--2010)}
\end{center}

\begin{abstract}
Classical 
 mechanical systems are modeled by a symplectic manifold $(M,\omega)$, and their symmetries,
 encoded in the action of a Lie group $G$ on $M$ by diffeomorphisms that preserves $\omega$. These actions, which are called ``symplectic",  
 have been studied in the past forty years, following  the works 
of Atiyah, Delzant, Duistermaat, Guillemin, Heckman, Kostant, Souriau, and Sternberg in the 1970s and 1980s
on symplectic actions of compact abelian Lie groups that are, in addition, of ``Hamiltonian" type, i.e. they also
satisfy Hamilton's equations. Since then a number of connections with combinatorics, finite dimensional integrable 
Hamiltonian systems, more general symplectic actions, and topology, have flourished. 
In this paper we review classical and recent results on Hamiltonian and non Hamiltonian symplectic group actions roughly 
starting from the results of these authors. The paper also serves as a quick introduction to the basics
of symplectic geometry.
\end{abstract}

\section{Introduction}

Symplectic geometry is a geometry concerned with the study of a notion of signed area, rather
than length or distance.  It can be, as we will see, less intuitive than Euclidean or metric geometry and it is taking
mathematicians many years to understand some of its intricacies (which is still work in progress).

The word ``symplectic" goes back
to Hermann Weyl's (1885-1955) book~\cite{Wy} on Classical Groups (1946). It derives from a Greek word 
meaning ``complex".  Since the word ``complex" had already a precise meaning in mathematics, and was
already used at the time of Weyl, he  took the Latin roots of ``complex" (which means ``plaited together") and replaced them by the Greek roots ``symplectic".

The origins of symplectic geometry are in classical mechanics, where the phase space of a mechanical
system is modeled by a ``{symplectic manifold}" $(M,\omega)$,  that is, a smooth manifold $M$ endowed with a non\--degenerate closed $2$\--form $\omega \in \Omega^2(M)$, called a ``{symplectic form}".  At each point $x \in M$,
 $\omega_x \colon {\rm T}_xM \times {\rm T}_x M \to \mathbb{R}$ is an antisymmetric bilinear form on ${\rm T}_xM$, and given $u,v \in {\rm T}_xM$ the real number
 $\omega_x(u,v)$  is called the ``{symplectic area}" spanned by $u$ and $v$.  Intuitively, $\omega$ gives
 a way to measure area along $2$\--dimensional sections of $M$, which itself can be of an arbitrarily large dimension.
 
The most typical example of a symplectic manifold  is a cotangent bundle, the phase space of mechanics, 
which comes endowed with a canonical
symplectic form. Initially it was the study of mechanical systems which motivated many of the developments in
symplectic geometry.  

Joseph\--Louis Lagrange (1736\--1813) gave the first example of a symplectic manifold in 1808,
in his study of the motion of the planets under the influence of their mutual gravitational 
interaction~\cite{La1,La2}. An explicit description of Lagrange's construction and his derivation 
of what are known today  as Hamilton's equations is given by  Weinstein in~\cite[Section 2]{Wesurvey}. 

The origins of the current view point in symplectic geometry may be traced back to Carl Gustav Jacob Jacobi (1804\--1851) and then William Rowan Hamilton's (1805\--1865) deep formulation of Lagrangian mechanics, around 1835. Hamilton was expanding on and reformulating  ideas 
of Galileo Galilei (1564\--1642), Christiaan Huygens (1629\--1695), Leonhard Euler (1707\--1883), Lagrange, 
and Isaac Newton (1642\--1727) 
about the structure and behavior of orbits of planetary systems.  

At the time of Newton and Huygens 
the point of view in classical mechanics was geometric. Later Lagrange, Jacobi, and Hamilton approached 
the subject from an analytic view point. Through their influence the more geometric view point fell out
of fashion. Further historical details and references are given by Weinstein in~\cite{Wesurvey}. Several treatments about mechanical systems
in the 1960s and 1970s, notably including~\cite{Ar74, Ar1978, am0, am, st12},  had an influence in the development 
of ideas in symplectic geometry.

The modern  view point in symplectic geometry starts with the important contributions of a number of authors in the early 1970s (some slightly before or slightly after) including the works of Ralph Abraham, Vladimir Arnold, 
Johannes J. Duistermaat,  Victor Guillemin,  Bertram Kostant, Paulette Libermann, George Mackey,  Jerrold Marsden,  Clark Robinson, Jean-Marie Souriau, Shlomo Sternberg, and Alan Weinstein. Even at these early stages, many other authors contributed to aspects of the subject so the list of 
developments is extensive and we do not make an attempt to cover it here.

 Symplectic geometry went through a series of developments in the period 1970\--1985 where connections with other
areas flourished, including: (i) geometric, microlocal and semiclassical analysis, as in the works of Duistermaat, Heckman, and  H\"ormander~\cite{DuHe, DuHo,Du4}; Duistermaat played a leading role
 in establishing relations between the microlocal and symplectic communities in particular through his
 article on oscillatory integrals and Lagrange immersions~\cite{Du4}; (ii)
 completely integrable systems, of which Duistermaat's article on global action\--angle coordinates~\cite{Du1980} may be considered to mark the beginning of the global theory of completely integrable systems; (iii) Poisson geometry, as in Weinstein's foundational article~\cite{WePoisson}; (iv) Lie theory and geometric quantization, as in Kostant and Souriau's geometric quantization~\cite{kostant1970, souriau1966} (in early 1960s  the quantum view point had already reached significant relevance in mathematics, 
see Mackey's mathematical foundations of quantum mechanics~\cite{Ma}), on which the works
by Segal~\cite{Se} and Kirillov~\cite{Ki} had an influence; and (v) symplectic and Hamiltonian group actions, as pioneered by Atiyah~\cite{atiyah},
Guillemin\--Sternberg~\cite{gs},  Kostant~\cite{kostant1966}, and Souriau~\cite{souriau1970}. 
It is precisely symplectic and Hamiltonian group actions that we are interested in this paper, 
 and we will give abundant references later.

An influential precursor in the study of global aspects in symplectic geometry, the study of which is often referred to as ``symplectic topology",  is Arnold's 
conjecture~\cite[Appendix 9]{Ar1978} (a particular case appeared in~\cite{Ar76};  
see  Zehnder's article~\cite{Ze1986} for an expository account). 
 Arnold's conjecture is  a higher dimensional analogue of the classical fixed point theorem of 
Henri Poincar\'e (1854\--1912) and George Birkhoff (1884-1944) which says that any area\--preserving periodic twist  of a closed annulus  has at least two geometrically distinct fixed points.  This fixed point result can be traced to the work of  Poincar\'e in celestial mechanics~\cite{Po93}, where he showed that the study of the dynamics of certain cases of the restricted $3$\--Body Problem 
may be reduced to investigating area\--preserving maps, and led him to this result, which he stated in~\cite{Po12} 
 in 1912. The complete proof was given  by Birkhoff~\cite{Bi1} in 1925.  Arnold realized
that the higher dimensional version of the result of Poincar\'e and Birkhoff should concern ``{symplectic maps}", 
that is, maps preserving a symplectic form, and not volume\--preserving
maps, and formulated his conjecture. Arnold's conjecture has been responsible for many of the developments in
symplectic geometry (as well as in other subjects like Hamiltonian dynamics and topology).

 In 1985 Gromov~\cite{Go85} introduced pseudoholomorphic curve techniques into symplectic geometry
and constructed the first so called ``{symplectic capacity}", a notion of monotonic symplectic invariant pioneered
by Ekeland and Hofer~\cite{EkHo1989, Hofer1990, Hofer1990b} and developed by Hofer and his collaborators,
as well as many others, from the angle of 
dynamical systems and Hamiltonian dynamics.

There have been many major developments since the early 1980s, and on many different fronts of the symplectic geometry and topology, and covering them (even very superficially)
would be beyond the scope of this paper. In this article we study only on the topic of   
symplectic and Hamiltonian group actions, item (v) above, starting roughly with the work of Atiyah and Guillemin\--Sternberg.

 While the phase space of a mechanical
system is mathematically modeled by a symplectic manifold, its 
symmetries are described by symplectic group actions. The study of such symmetries or actions
 fits into a large body of work by the name of ``{equivariant symplectic geometry}", which
 includes tools of high current interest also in algebraic geometry, such as equivariant cohomology
 on which we will (very) briefly touch.

Mathematically speaking, equivariant symplectic geometry is concerned with the study of smooth actions of Lie groups $G$ on symplectic manifolds $M$, by means of diffeomorphisms $\varphi \in {\rm Diff}(M)$
which pull\--back the symplectic form $\omega$ to itself:
 $\varphi^*\omega=\omega.$ A map $\varphi$ satisfying this condition is called a
``{symplectomorphism}"  following Souriau, or a ``{canonical transformation}". Actions
satisfying this natural condition are called ``{symplectic}".  As a first example of a symplectic action 
 consider $S^2 \times (\R/\Z)^2$  with the product form (of any areas forms on
$S^2$ and $(\R/\Z)^2$). The action of the $2$\--torus $(\R/\Z)^2$ by translations 
on the right factor is symplectic.

 In this paper we treat primarily the case when
$G$ is a compact, connected, abelian Lie group, that is, a torus: $T \simeq (S^1)^k, \,\,k \geq 1.$ 
Let $\mathfrak{t}$ be the Lie algebra of $T$, and let $\mathfrak{t}^*$ be its dual Lie algebra. We think
of $\mathfrak{t}$ as the tangent space at the identity $1 \in T$.

Equivalently, a $T$\--action is symplectic if
 ${\rm L}_{X_M}\omega=0$ for every $X \in \mathfrak{t}$, where ${\rm L}$ is
the Lie derivative and $X_M$ is the vector field generated by the $T$\--action from $X \in \mathfrak{t}$
via the exponential map. In view of the homotopy
formula for the Lie derivative, this is equivalent to 
\begin{eqnarray} \label{closed}
{\rm d}(\omega(X_M,\cdot))=0
\end{eqnarray}
 for every $X \in \mathfrak{t}$.

A fundamental subclass of symplectic
 actions admit what is called a ``{momentum map}" $\mu \colon M \to \mathfrak{t}^*$,
 which is a $\mathfrak{t}^*$\--valued smooth
 function on $M$ which encodes information about $M$ itself, the symplectic form, and the $T$\--action, and
 is characterized by the condition that for all $X \in \mathfrak{t}$:
 \begin{equation}
-{\rm d} \langle \mu, \, X \rangle = \omega(X_M,\cdot). \label{x}
\end{equation}
Such very special symplectic actions  are called ``{Hamiltonian}" (the momentum map was
introduced generally for any Lie group action by Kostant \cite{kostant1966} and Souriau \cite{souriau1966}). 
A simple example would
be to take $M=S^2 \subset \R^3$ and $T=S^1$ acting by rotations about the $z=0$ axis in $\R^3$.
In this case $\mathfrak{t}^*\simeq \R^* \simeq \R$ and $\mu \colon (\theta,h) \mapsto h$. 
 
The fundamental observation here is that the right hand side
of equation (\ref{x}) is always a closed $1$\--form by (\ref{closed}) and being Hamiltonian may be rephrased as 
the requirement that this form is moreover exact. Therefore, the obstruction for a symplectic 
action to being Hamiltonian lies in the first cohomology group ${\rm H}^1(M;\mathbb{R})$.

In particular,
any symplectic action on a simply connected manifold is Hamiltonian. Notice that is an extremely
stringent condition, for instance by (\ref{x}) it forces the action to have fixed points on a compact manifold (because
$\mu$ always has critical points, and these correspond to the fixed points of the action). 

Many symplectic
actions of interest in complex algebraic geometry and K\"ahler geometry are symplectic but not
Hamiltonian; one such case is the action of the $2$\--torus on the Kodaira variety, which appears in  Kodaira's description~\cite[Theorem~19]{kodaira} of the  compact complex analytic surfaces that 
have a  holomorphic $(2,\, 0)$\--form that is nowhere 
vanishing, described later in this paper (Example~\ref{ktexample}).  Other symplectic actions that do not admit a momentum map include examples of interest  in classical differential geometry (eg.~multiplicty free spaces), and topology (eg.~nilmanifolds over nilpotent Lie groups).  

 Recent work of Susan Tolman (Theorem~\ref{MT}) 
indicates that even for low dimensional Lie
groups $G$ most symplectic actions are not Hamiltonian. The advantage of having the existence of a momentum map $\mu \colon M \to \mathfrak{t}^*$ for
a symplectic action  has led to a rich 
general theory, part of which is described in this article. One can often find out
information about $(M,\omega)$ and the $T$\--action through the study of $\mu$. For instance, if $\dim M =2 \dim T$,
Delzant proved~\cite{De} that the image 
$\mu(M) \subset \mathfrak{t}^*$ completely characterizes $(M,\omega)$ and the $T$\--action, up
to symplectic and $T$\--equivariant transformations.
We will see a proof of this result later in the paper.

 Hamiltonian actions
have been extensively studied since the 1970s following  the seminal works 
of Atiyah~\cite{atiyah}, Delzant~\cite{De}, Duistermaat\--Heckman~\cite{DuHe}, 
Guillemin\--Sternberg~\cite{gs}, and Kostant~\cite{kostant1966}, and have been
a motivation to study more general symplectic actions.  

The majority of proofs/results about general
symplectic actions use in an essential way the Hamiltonian theory, 
but also include other ingredients. The
fact that  there is not necessarily a momentum map $\mu$ means that Morse
theory for $\mu$ and Duistermaat\--Heckman theory, often used in the Hamiltonian case,
must be replaced by alternative techniques.

Research on symplectic actions is still at its infancy and there are many unsolved
problems. A question of high interest has been whether there are symplectic non\--Hamiltonian $S^1$\--actions
with some, but only finitely many fixed points on compact connected manifolds. The aforementioned result by 
 Tolman provides an example~\cite{To15} of such an action with thirty two fixed points.  
In \cite{GoPeSa2015} the authors give a general lower bound for the number of fixed
points of any symplectic $S^1$\--action, under a mild assumption.

  Approximately the first half of the paper concerns
 the period from 1970 to approximately 2002, where  the emphasis is on symplectic Hamiltonian actions,
 its applications, and its implications,  including the (subsequent) interactions with completely integrable systems.
 The second half of the paper concerns symplectic actions which are not necessarily Hamiltonian, with a focus
 on the developments that took place in the approximate period from 2002 to 2015.
 
 Of course this separation is somewhat artificial, because Hamiltonian actions play a fundamental
 role in the study of other types of symplectic actions.

The paper  gives a  succinct introduction to the basics of symplectic
 geometry, followed by an introduction to symplectic and Hamiltonian actions, and
 it is is written for a general audience of mathematicians.    It  is not a survey, which would require, due to the volume of works,
 a much longer paper.  
 We will cover a few representative proofs with the goal of giving readers a flavor of the subject. The background assumed is  knowledge of
 geometry and topology for instance as covered in second year graduate courses.

 The author is supported by  NSF CAREER Grant DMS-1518420. He is very thankful to
Tudor Ratiu and Alan Weinstein for discussions and for pointing out  several useful references.
 \smallskip

   \emph{Outline of Topics.} Section~\ref{symm} give 
 a succinct introduction to symplectic geometry for
 a general mathematical audience.

 Sections~\ref{ham:sec} and \ref{ham2:sec} introduce the basics of  symplectic
 and Hamiltonian Lie group actions.
 
Section~\ref{examples2} contains examples of  Hamiltonian and symplectic non\--Hamiltonian torus actions. 
 
 Section~\ref{hck} includes classification results on symplectic Hamiltonian actions of Lie groups.
 In most cases, the Lie group is compact, connected, and abelian; but a certain case of
 noncompact groups which is pertinent to completely integrable systems is also included.
  
 The material from Section~\ref{ffg} onwards is probably less well known to 
 nonexperts; it focuses on developments on symplectic group actions, not necessarily Hamiltonian, 
 in the past fifteen years, with an emphasis on classification results in terms of
 symplectic invariants.
  
\section{Symplectic manifolds} \label{symm}

Symplectic geometry is  concerned with 
the study a notion of signed \emph{area}, rather than length,  distance, or volume.  In this sense it is a
peculiar type of geometry, which displays certain non intuitive features, as we see in this section. 

Symplectic geometry displays a degree of flexibility and rigidity at the same time which makes it a
rich subject, the study of which is of interest well beyond its original connection to classical mechanics.

\subsection{Basic properties} \label{s2}

For the basics of symplectic geometry, we recommend the textbooks \cite{HoZe1994, AC, MS}. This section gives a 
quick overview of the subject, and develops the fundamental notions which we need for the following
sections. Unless otherwise specified all manifolds are ${\rm C}^{\infty}$\--smooth and have no boundary.

\begin{definition}
A \emph {symplectic manifold} is a pair $(M,\, \omega)$ consisting of
a smooth ${\rm C}^{\infty}$\--manifold $M$, and a smooth $2$\--form $\omega$ on $M$ which is:
(1) \emph{closed}, i.e. ${\rm d}\omega=0$;
(2)  \emph{non\--degenerate}, i.e. for each $x \in M$ it holds that if $u \in {\rm T}_xM$ is such that 
$\omega_x(u,v)=0$ for all $v \in {\rm T}_xM$
at $x$, then necessarily $u=0$.  The  form $\omega$ is called a \emph{symplectic form}. 
\end{definition}

\medskip
Conditions (1) and (2) can be geometrically understood as we will see.

\begin{example}
In dimension $2$ a symplectic form and
an area form are the same object.  Accordingly, the simplest example of a symplectic
manifold is given by a surface endowed with an area form.  
 A typical non\--compact example is the Euclidean space $\mathbb{R}^{2n}$ with 
coordinates $(x_1,y_1,\ldots,x_n,y_n)$ equipped with the symplectic form 
$\sum_{i=1}^n{\rm d}x_i \wedge {\rm d}y_i.$
Any open subset $U$ of $\mathbb{R}^{2n}$ endowed with the symplectic form
given by this same formula is also a symplectic manifold.
\end{example}

Let $X$ be a smooth $n$\--dimensional manifold and let
$(V,\, x_1, \ldots, x_n)$ be a smooth chart for $X$. To this chart we can associate a cotangent chart 
$({\rm T}^*V,\,x_1,\, \ldots,\,x_n,\,\xi_1,\, \ldots,\,\xi_n)$ on which we can define, in coordinates,
the smooth $2$\--form given by the formula
$
\omega_0:=\sum_{i=1}^n {\rm d} x_i \wedge {\rm d} \xi_i .
$

\begin{prop} \label{pqr}
The expression for $\omega_0$ is coordinate independent, that is, it defines a canonical form
on the cotangent bundle ${\rm T}^*X$, which is moreover symplectic and exact. 
\end{prop}

\begin{proof}
It follows by
observing that 
$\omega_0=-{\rm d}\alpha
$ 
where
$
\alpha:=\sum_{i=1}^n \xi_i \,{\rm d}x_i, 
$
which one can  check that is intrinsically defined. 
\end{proof}

A  way to construct symplectic
manifolds is by taking products, and endowing them with the product
symplectic form. As such, $S^2 \times \mathbb{R}^{2n}$ or ${\rm T}^*X \times {\rm T}^* Y$
are naturally symplectic manifolds, where $X$ and $Y$ are manifolds of any finite dimension.

There is a geometric interpretation of the closedness of the symplectic form in terms of area of
a surface, as follows. Define the \emph{symplectic area} of a surface $S$ (with or without boundary) 
inside of a symplectic
manifold $(M,\omega)$ 
 to be the integral 
 \begin{eqnarray} \label{sf}
 {\rm symplectic \,\,area\,\, of\,\,} S:=\varint_S \,\, \omega \in \mathbb{R}.
 \end{eqnarray}
By Stokes' theorem,
the closedness condition ${\rm d}\omega=0$ implies:

\begin{prop} \label{xc2}
Every point $x \in M$ has an open neighborhood such that if the surface $S$ is contained in $U$ then
the symplectic area of $S$ does not change when 
deforming  $S$ inside of $U$ while keeping the boundary $\partial S$ of $S$ fixed under the deformation. If $\omega$
is exact then $U=M$.
\end{prop}

\begin{proof}
If ${\rm d}\omega=0$ then locally near every point $x \in M$, $\omega={\rm d}\sigma$
for some smooth $1$\--form $\sigma \in \Omega^1(M)$. Hence
$$
\varint_S \omega=\varint_S {\rm d}\sigma=\varint_{\partial S}\sigma,
$$
and the result follows.
\end{proof}

In view of
Proposition~\ref{pqr} and Proposition~\ref{xc2} we obtain the following.

\begin{cor}
The symplectic area of any surface $S$ in a cotangent bundle $({\rm T}^*X,\omega_0)$
depends only on the boundary $\partial S$. Moreover, if $\partial S=\varnothing$ then
the symplectic area of $S$ is zero.
\end{cor}

 The fact that $\omega$ is non\--degenerate gives:

\begin{prop}
Let $(M,\omega)$ be a symplectic manifold.
There is an isomorphism between
the tangent and the cotangent bundles 
$
  {\rm T}M \to {\rm T}^*M$ by means of the mapping
$\mathcal{X}	\mapsto \omega(\mathcal{X}, \,\, \cdot ).$
\end{prop}

In other words, the symplectic form allows us to give 
a natural correspondence between one\--forms and vector fields. 

One can ask  some basic questions about symplectic manifolds. For instance, one can wonder:

\begin{question}
Does the $3$\--dimensional sphere $S^3$ admit a symplectic form? 
\end{question}

\medskip

 The answer is ``no", because symplectic
manifolds are even\--dimensional; for otherwise the non\--degeneracy condition is violated,
which follows from elementary linear algebra.  Similarly:

\begin{question}
Does the Klein bottle admit a symplectic form?
\end{question}

\medskip

The answer to this question is again ``no" since symplectic manifolds must be orientable; indeed, 
since the symplectic form $\omega$ is non\--degenerate,  the wedge $\omega^n=\omega \wedge \ldots (n\,\, \textup{{\tiny times}})\,\, \wedge \omega,$ where $2n$ is the
dimension of the manifold, is a volume form giving an orientation to $M$.  

\medskip

To summarize:

\begin{prop}
Symplectic manifolds are even\--dimensional and orientable.
\end{prop}

To continue the discussion with spheres:

\begin{question}
Does the $4$\--dimensional sphere 
$S^4$ admit a symplectic form?
\end{question}

\medskip
 The answer  is given by the following.

\begin{prop}
Let $(M,\omega)$ be a compact symplectic manifold of dimension $2n$. Then its even\--dimensional
cohomology groups are non trivial, that is, ${\rm H}^{2k}(M,\mathbb{R})\neq 0$ for $1 \leq k \le n$.
\end{prop}

\begin{proof}
The cohomology class $[\omega^k]$ is nontrivial. This is an exercise which follows from Stokes' theorem, using
${\rm d}\omega=0$. 
\end{proof}

One immediate consequence of these observations is the following.

\begin{prop} \label{rv}
The $2$\--dimensional
sphere $S^2$ is the only sphere $S^n,n \geq 1$, which may be endowed with a symplectic form. 
\end{prop}

However, the question whether an arbitrary manifold it admits or not a symplectic
form is in general very difficult. 

For instance if $N$ is a closed oriented $3$\--manifold, in Friedl\--Vidussi \cite{FrVi} 
and Kutluhan\--Taubes~\cite{KuTa}, the authors study when a closed $4$\--manifold of the form $S^1 \times N$ 
admits a symplectic form, which turns out to imply that $N$ must fiber over the circle $S^1$ (for details
refer to the aforementioned papers).

\subsection{Symplectomorphisms} \label{can}

The natural maps between symplectic manifolds are the diffeomorphisms which preserve the symplectic
structure, they are called  \emph{canonical transformations}, \emph{symplectic diffeomorphisms}, or 
following Souriau~\cite{souriau1970}, \emph{symplectomorphisms}.

\begin{definition} \label{ed}
A \emph{symplectomorphism} $\varphi \colon (M_1,\omega_1) \to (M_2,\omega_2)$
between symplectic manifolds is a diffeomorphism $\varphi \colon M_1 \to M_2$ 
which satisfies $\varphi^*\omega_2=\omega_1.$ In this
case we say that \emph{$(M_1,\omega_1)$ and $(M_2,\omega_2)$ are symplectomorphic}.
\end{definition}

\medskip
Recall that the expression $\varphi^*\omega_2=\omega_1$ in Definition~\ref{ed} means that 
$(\omega_2)_x({\rm d}_x\varphi(u),{\rm d}_x\varphi(v))=(\omega_1)_x(u,v)$ for every point 
$x \in M_1$ and for every pair of tangent vectors $u,v \in {\rm T}_xM_1$.  That is, 
the symplectic area spanned by $u,v$ coincides with the symplectic area spanned by the
images ${\rm d}_x\varphi(v),{\rm d}_x\varphi(w)$, for every $x \in M$, and for 
every $u,v \in {\rm T}_xM$.

 \begin{remark}
 Roughly speaking one can view symplectomorphisms as diffeomorphisms
preserving the area enclosed by loops, or rather, the sum of the areas enclosed by their projections onto 
a collection of $2$\--dimensional planes. For instance, if 
$(M,\omega)=(\mathbb{R}^6,\, 
{\rm d}x_1 \wedge {\rm d}y_1+{\rm d}x_2 \wedge {\rm d}y_2+{\rm d}x_3 \wedge {\rm d}y_3),$ with coordinates 
$(x_1,y_1,x_2,y_2,x_3,y_3)$, then you would want to preserve the area (counted by $\omega$ with sign depending
on the orientation of the region inside) of the projection of any loop in $\mathbb{R}^6$ onto the
$(x_1,y_1)$, $(x_2,y_2)$ and $(x_3,y_3)$ planes. I learned how to think about symplectomorphisms
in this way from Helmut Hofer. 
\end{remark}

The \emph{symplectic volume} (or \emph{Liouville volume}) of a symplectic manifold of 
dimension $2n$ is
 \begin{eqnarray} \label{volm}
 {\rm vol}(M,\omega):=\frac{1}{n!}\varint_M\omega^n.
\end{eqnarray}
Of course, since symplectomorphisms preserve $\omega$, they preserve the symplectic
volume but the converse is false (we discuss this in Section~\ref{si}).

 Since the late twentieth century it is known that 
symplectic manifolds have no local invariants
except the dimension. This is a result due to Jean-Gaston Darboux (1842\--1917).

\begin{theorem}[Darboux~\cite{darboux}, 1882] \label{da}

Let $(M,\omega)$ be a symplectic manifold. 
Near  each point  $ p\in M$ one can find  coordinates
 $(x_{1},y_{1},\ldots,x_n,y_n)$ in which the symplectic form
 $\omega$ has the expression 
$
\omega=\sum_{i=1}^n{\rm d}x_i \wedge {\rm d}y_i. 
$
That is, any two symplectic manifolds $(M_1,\omega_1)$ and $(M_2,\omega_2)$
of the same dimension  are locally symplectomorphic near any choice of points
$p_1 \in M_1$ and $p_2 \in M_2$.
\end{theorem}

Theorem~\ref{da} gives a striking difference with Riemannian geometry, where
the curvature is a local invariant.  

\begin{remark}
It is important to note, however, that  there are local aspects of symplectic actions which
have only been understood recently (certainly much later than Darboux's result), but in theses cases we are not concerned with the local normal form of the symplectic form itself, but instead with that a geometric
object (a form, a vector field, an action, etc.) defined on the manifold. For instance in the case
of symplectic group actions which we will discuss later the local normal
form of symplectic Hamiltonian actions is due to 
Guillemin\--Marle\--Sternberg~\cite{gsst, marle}, and  in the general case of symplectic actions to Ortega\--Ratiu~\cite{ortegaratiu} and Benoist~\cite{benoist}, in a neighborhood of an orbit.
\end{remark}

In 1981 Alan Weinstein referred to symplectic geometry in \cite{Wesurvey} as 
``the more flexible geometry of canonical (in particular, area
preserving) transformations instead of the rigid geometry of Euclid; accordingly,
the conclusions of the geometrical arguments are often qualitative
rather than quantitative."  

Manifestations of ``rigidity" in symplectic 
geometry were discovered in the early days of modern symplectic geometry 
by Mikhael Gromov, Yakov Eliashberg, and others.

\begin{theorem}[Eliashberg\--Gromov~\cite{E1,E2,G1}] The group
of symplectomorphisms of a compact symplectic manifold is ${\rm C}^{0}$\--closed in the
group of symplectomorphisms.
\end{theorem}

\begin{remark}
The group of symplectomorphisms of a manifold often has a complicated (but extremely interesting) structure, and many
basic questions about it remain open, see Leonid Polterovich's book~\cite{Po01}.  
Interest in the behavior of symplectic matrices may be found in the early days of symplectic geometry, in important work of Clark Robinson~\cite{Ro1, Ro2}. See also Arnold~\cite{Ar66} for work in a related direction.
\end{remark}

\subsection{Moser stability}

In 1965 J\"urgen Moser  published an influential article~\cite{Mo65} in which he showed that:

\begin{theorem}[Moser \cite{Mo65}] \label{mm}
 If $\omega$ and $\tau$ are volume forms on a compact connected oriented smooth manifold without boundary $M$ such that  $$\varint_M \omega=\varint_M \tau$$
then there exists a diffeomorphism $\psi \colon M \to M$ such that $\psi^*\tau=\omega$.
\end{theorem}

\begin{remark}
It follows that the total symplectic area of $S^2$ given by (\ref{sf}) completely
determines the symplectic form on $S^2$ (and $S^2$ is the only symplectic sphere
according to Propositon~\ref{rv}).
\end{remark}

An extension to fiber bundles of this result appears in~\cite{KeLe}. 

The proof of Moser's theorem uses in an essential way the compactness of $M$, but the method of proof, known as Moser's method and described below, may be generalized.  The following is the generalization (1979) to noncompact manifolds.

\begin{theorem}[Greene\--Shiohama \cite{GrSh79}]
 If $M$ is a noncompact connected smooth manifold and  $\omega$ and $\tau$ are 
 volume forms on $M$ such that 
 $$\varint_M \omega = \varint_M \tau \leq \infty$$
then there is a volume preserving diffeomorphism 
$\varphi \colon M \to M$
such that $\varphi^* \tau = \omega$ provided that every end $\epsilon$ of $M$ it holds that $\epsilon$ has finite volume with respect to $\omega$ if and only if has finite volume with respect to $\tau$.
\end{theorem}

\begin{remark}
If the assumption on the ends does not hold then they prove that their result does not necessarily hold~\cite[Example in page 406]{GrSh79}.
\end{remark}

This result was recently extended to fiber bundles with noncompact fibers in~\cite{PeTa2016}. The simplest
incarnation of this result is the case of trivial fiber bundles, which is equivalent to considering,  instead
of two volume forms, two \emph{smooth families} of volume forms $\omega_t, \tau_t$, indexed by some compact manifold without boundary which plays the role of parameter space $B$.  The Greene\--Shiohama theorem 
produces for each point $t$ a volume preserving diffeomorphism $\varphi_t$ with the required properties, but there is no information given about how the $\varphi_t$ change when $t$ changes in $B$.

In the same article where Moser proved Theorem~\ref{mm}, he also proved the following stability result
for symplectic forms.

\begin{theorem}[Moser~\cite{Mo65}, 1965]
Suppose we have a smooth family of symplectic forms $\{\omega_t\}_{t \in [0,1]}$ on a compact smooth manifold $M$ with exact derivative 
$
  \frac{d\omega_t}{dt}=d\sigma_t
$
(or $[\omega_t]$ constant in $t$) then there exists a smooth family $\{\varphi_t\}_{t \in [0,1]}$ of diffeomorphisms of $M$ such that $\varphi^*_t \omega_t=\omega_0$.
\end{theorem}

Moser's article introduced a method, known as ``Moser's method" to prove this stability result. In addition to
Moser's article there are now many expositions of the result, see for instance~\cite{MS}.

\subsection{The fixed point theorem of Poincar\'e and Birkhoff}

In his work in celestial mechanics~\cite{Po93}
Poincar\'e showed the study of the dynamics of certain cases of the restricted $3$\--Body Problem 
may be reduced to investigating area\--preserving maps. 
He concluded that there is no reasonable way to solve the problem {explicitly} in
the sense of finding formulae for the trajectories.   

Instead of aiming at finding the trajectories, in dynamical systems one aims at 
describing their analytical and topological behavior.  
Of a particular interest are the constant ones, i.e., the fixed points.   

The development of  the modern field of dynamical systems was markedly influenced  
by Poincar\'e's work in mechanics, which led him to state (1912)  the Poincar\'e\--Birkhoff 
Theorem \cite{Po12, Bi1}. It was proved in full by Birkhoff in 1925.  

We will formulate the
result (equivalently) for the strip covering the annulus.

\begin{definition}
Let $\mathcal{S}=\mathbb{R}\times [-1,1]$. A diffeomorphism
$
F \colon \mathcal{S} \to \mathcal{S}$,
$F(q,p)=(Q(q,p),P(q,p)),$
 is an \emph{area\--preserving periodic twist} if:
 (1) {\emph{area preservation}}: it preserves area; 
(2) {\emph{boundary invariance}}: it preserves $\ell_{\pm}:=\mathbb{R}\times \{\pm 1\} $, i.e. 
 $P(q,\pm 1):=\pm 1;$
 (3) {\emph{boundary twisting}}:  
$F$ is orientation preserving and  
$\pm Q(q,\pm 1)>\pm q$ for all $q$;
(4) {\emph{periodicity}}: $F(q+1,p)=(1,0)+F(q,p)$ for all $p,q$. 
\end{definition}

\medskip
The following is the famous result of Poincar\'e and Birkhoff on area\--preserving
twist maps.

\begin{theorem}[Poincar\'e\--Birkhoff  \cite{Po12, Bi1}] \label{cpbt}
An area\--preserving periodic twist 
$
F \colon \mathcal{S} \to \mathcal{S}
$
has at least two geometrically distinct fixed points.
\end{theorem}

Arnold formulated
the higher dimensional analogue of this result,  the
Arnold Conjecture \cite{Ar1978} (see also \cite{Au13}, \cite{Ho12},  \cite{HoZe1994}, \cite{Ze1986}).  
The conjecture says that a ``Hamiltonian map" on a compact symplectic
manifold  possesses at least as many fixed points as a function on the manifold 
 has critical points, we refer to Zehnder~\cite{Ze1986} for details. A. Weinstein~\cite{We86} 
 observed that Arnold's conjecture holds on compact manifolds 
when the Hamiltonian map belongs to the flow of a sufficiently small Hamiltonian vector field.
The first breakthrough on the 
conjecture was by Charles Conley and Eduard Zehnder~\cite{CoZe1983}, who
proved it for the $2n$\--torus (a proof
using generating functions was later given
by Chaperon \cite{Ch1984}).  According to
their theorem, any smooth symplectic map $F:\mathbb{T}^{2d}\to \mathbb{T}^{2d}$ that is isotopic to the identity has at least $2d+1$ many fixed points.  
 The second breakthrough
was by Floer~\cite{Fl1988,Fl1989,Fl1989b,Fl1991}.  

There have been many generalizations of this result, see for instance~\cite{F1, PeRe}.

\subsection{Lagrangian submanifolds}

In the years 1970\--1975 Alan Weinstein proved a series of foundational theorems about what Maslov
called \emph{Lagrangian
submanifolds} which were influential in the development of symplectic geometry.

\begin{definition}
A submanifold $C$ of a symplectic manifold $(M,\omega)$ is \emph{isotropic} if
the symplectic form vanishes along $C$, that is, $i^*\omega=0$ where $i \colon C \to M$
is the inclusion mapping.
\end{definition}

\medskip

It is an exercise to verify that if $C$ is isotropic then $2\dim  C\leq \dim  M$.

For instance, in the symplectic plane $(\mathbb{R}^2,{\rm d}x\wedge {\rm d}y)$ any line
$x={\rm constant}$ or $y={\rm constant}$ is an isotropic submanifold. More generally:

\begin{example}
For any
 constants $c_1, \ldots,c_n$, $
X_{c_1,\ldots,c_n}:=\{(x_1,y_1,\ldots,x_n,y_n) \,|\, x_i=c_i \,\, \, \forall i=1,\ldots,n\}
$
is an isotropic submanifold of $(\mathbb{R}^{2n},\sum_{i=1}^n {\rm d}x_i \wedge {\rm d}y_i)$
of dimension $n$. 
\end{example}

The following is due to Maslov~\cite{Mas}.

\begin{definition} \label{maslov}
An isotropic submanifold $C$ of a symplectic manifold $(M,\omega)$ is \emph{Lagrangian} if $2 \dim  C=\dim  M$.
\end{definition}

\medskip

Cotangent bundles  provide a source of Lagrangian submanifolds.

\begin{lemma}
The image of a section $s \colon X \to {\rm T}^*X$ of the contangent 
bundle ${\rm T}^*X$ is Lagrangian if and only if ${\rm d}s=0$.
\end{lemma}

The following is one of the foundational results of symplectic geometry, it is known
as the Lagrangian neighborhood theorem.

\begin{theorem}[Weinstein~\cite{We71}] \label{kl}
Let $M$ be a smooth $2n$\--dimensional  manifold and let $\omega_0$ and $\omega_1$
be symplectic forms on $M$. Let $X$ be a compact $n$\--dimensional submanifold. 
Suppose that $X$ is a Lagrangian submanifold of both $(M,\omega_0)$ and $(M,\omega_1)$.
Then there are neighborhoods $V_0$ and $V_1$ of $X$, and a symplectomorphism
$\varphi \colon (V_0,\omega_0) \to (V_1,\omega_1)$ such that $i_1=\varphi \circ i_0$ where $i_0 \colon X \to V_0$
and $i_1 \colon X \to V_1$ are the inclusion maps.
\end{theorem}

Using Theorem~\ref{kl}, Weinstein proved the following result.

\begin{theorem}[Weinstein~\cite{We71}] \label{kl2} Let $(M,\omega)$ be
a symplectic manifold. Let $X$ be a closed Lagrangian submanifold.
Let $\omega_0$ be the standard cotangent symplectic form on ${\rm T}^*X$.
Let $i_0 \colon X \to {\rm T}^*X$ be the Lagrangian embedding given by the zero section
and $i \colon X \to {\rm T}^*X$ be the Lagrangian embedding given by the
inclusion. Then there are
 neighborhoods $V_0$ of $X$ in ${\rm T}^*X$ and $V$
of $X$ in $M$, and
 a symplectomorphism $\varphi \colon (V_0,\omega_0) \to (V,\omega)$
such that $i=\varphi \circ i_0$.
\end{theorem}

As we will see later in this paper, isotropic and Lagrangian submanifolds
play a central role in the theory of symplectic group actions (as well
as in  other parts of symplectic geometry, for instance the study of 
intersections of Lagrangian submanifolds, see Arnold~\cite{Ar65},
Chaperon~\cite{Ch82}, and Hofer~\cite{Ho85}).

 \subsection{Monotonic symplectic invariants} \label{si}
Let ${\rm B}^{2n}(R)$ be the open ball of radius $R>0$.  If $U$ and $V$ are open
subsets of $\mathbb{R}^{2n}$,  with the standard symplectic form
$\omega_0=\sum_{i=1}^{n} {\rm d}x_i \wedge {\rm d}y_i.$

For  $n\geq 1$ and $r>0$ let 
$\mathrm{B}^{2n}(r) \subset \C^n$ be the $2n$\--dimensional open symplectic ball 
of radius $r$ and let 
$
 \mathrm{Z}^{2n}(r) = \{\,(z_i)_{i=1}^n\in\C^n \mid \abs{z_1} < r\,\}
$
be the $2n$\--dimensional open symplectic cylinder of radius $r$.
Both inherit a
symplectic structure from their embedding as a subset 
of $\C^n$ with symplectic form  $\om_0 = \frac{\mathrm{i}}{2}\sum_{j=1}^{n}{\rm d}z_j\wedge {\rm d}\bar{z}_j$.

A \emph{symplectic embedding} $f \colon U
\to V$ is a smooth embedding such that $f^*\omega_0=\omega_0.$

 Similarly
one defines symplectic embeddings $f \colon (M_1,\omega_1) \to
(M_2,\omega_2)$ between general symplectic manifolds.

\medskip

If there is a symplectic embedding $f \colon U \to V$ then 
 ${\rm vol}(U)$ is at most  equal to ${\rm vol}(V)$, that
is, the volume provides an elementary embedding obstruction.
  
\begin{theorem}[Gromov~\cite{Go85}]  \label{ns}
There is no symplectic embedding of ${\rm B}^{2n}(1)$ into ${\rm Z}^{2n}(r)$ for $r<1$. 
\end{theorem}

This result shows a rigidity property that symplectic transformations exhibit, in contrast with
their volume\--preserving counterparts.  It shows that in addition to the volume there
are other obstructions which are more subtle and come from the
symplectic form, they are called \emph{symplectic capacities}.

Denote by $\mathcal{E}\ell\ell$ the
category of ellipsoids in $\mathbb{R}^{2n}$ with symplectic embeddings
induced by global symplectomorphisms of $\mathbb{R}^{2n}$ as
morphisms, and by ${\rm Symp}^{2n}$ the category of symplectic
manifolds of dimension $2n$, with symplectic embeddings as morphisms.
A \emph{symplectic category} is a subcategory $\mathcal{C}$ of ${\rm
  Symp}^{2n}$ containing  $\mathcal{E}\ell\ell$ 
  such that $(M,\omega) \in \mathcal{C}$ implies that
$(M,\lambda\omega) \in \mathcal{C}$ for all $\lambda>0$.

 Let $d$ be a fixed integer such that $1 \leq d \leq n$. Following  Ekeland\--Hofer and Hofer~\cite{EkHo1989,Hofer1990}
 we make the following definition of monotonic symplectic invariant.

\begin{definition} \label{capa}
A \emph{symplectic $d$\--capacity} on a symplectic category
$\mathcal{C}$ is a functor $c$ from $\mathcal{C}$ to the
category $([0,\infty],\leq)$ satisfying: i) \emph{Monotonicity}: $c(M_1,\omega_1)\leq c(M_2,\omega_2)$ if
  there is a morphism from $(M_1,\omega_1)$ to $(M_2,\omega_2)$;  ii) \emph{Conformality}: 
  $c(M,\lambda\omega)=\lambda
  c(M,\omega)$ for all $\lambda>0$.
iii) \emph{Non\--triviality}: $c({\rm B}^{2n}(1))>0$,    $c({\rm B}^{2d}(1) \times \mathbb{R}^{2(n-d)})<\infty$, and
$c({\rm B}^{2(d-1)}(1) \times \mathbb{R}^{2(n-d+1)})=\infty.$ 
If $d=1$, a symplectic $d$\--capacity is called a \emph{symplectic capacity}. 
\end{definition}

\medskip

Symplectic capacities
are a fundamental class of invariants of symplectic manifolds. 
 They were introduced in Ekeland and Hofer's 
work~\cite{EkHo1989,Hofer1990}. 

The first symplectic capacity was the Gromov radius, constructed
by Gromov in~\cite{Go85}: it is the radius of the largest ball (in the
sense of taking a supremum) that can be symplectically embedded in a manifold:
$$
(M,\om) \mapsto\mathrm{sup} \{\,R>0 \mid \mathrm{B}^{2n}(R) \imm M \,\},$$ 
where $\imm$ denotes a symplectic embedding. The fact that the Gromov radius is a symplectic capacity is a deep result, it follows from Theorem~\ref{ns}. 

The symplectic volume (\ref{volm}) is a symplectic $n$\--capacity.

Today many constructions of symplectic capacities are known, see~\cite{CHLS07}; therein
one can find for instance two of the best known capacities, the Hofer-Zehnder capacities and 
the Ekeland-Hofer capacities, but there are many more.

The following results~\cite{Gu08, PeVN15} clarify the existence of continous symplectic $d$\--capacities.
In what follows a capacity satisfies the \emph{exhaustion property} if 
value of the capacity on any open set equals the supremum of the values
on its compact subsets.

\begin{theorem}[Guth~\cite{Gu08}]
Let $n\geq 3$.  If $1<d<n$,
  symplectic $d$\--capacities satisfying the exhaustion property do not exist on any subcategory of the
  category of symplectic $2n$\--manifolds.
\end{theorem}

That is, other than the volume, the monotonic invariants of symplectic geometry only measure 
$2$\--dimensional information. 

The assumption on the previous theorem was  removed in~\cite{PeVN15} where
the authors prove that  symplectic $d$\--capacities do not exist on any subcategory of the
  category of symplectic $2n$\--manifolds.

There are  invariants of symplectic manifolds which do not fit Definition~\ref{capa},
see for instance~\cite{MS}.

An equivariant theory of symplectic capacities appears
in~\cite{FiPaPe15}, we will discuss it later.

\section{Symplectic and Hamiltonian actions} \label{ham:sec}

\subsection{Lie groups and Lie algebras}

Lie groups are smooth manifolds that are simultaneously groups, and as such
they can also act on smooth manifolds and describe their symmetries. 

They are named after
Sophus Lie (1842-1899), one of the most influential figures in 
differential geometry to whom many of the modern notions can be traced back, including 
that of a transformation (Lie) group, and some particular instances of the notion of 
a momentum map (already mentioned in the introduction, but which we will formally define shortly).  

Recall that a \emph{Lie group} is a pair $(G,\, \star)$ where 
$G$ is smooth manifold and
$\star$ is an internal group operation $\star \colon G \times G \to G$ which is smooth
and such that $G\to G$, $g \mapsto g^{-1}$ is also smooth.

\begin{example}
The most important example of Lie group for the purpose of this paper is the circle, which
may be viewed in two isomorphic ways, either as a quotient of $\R$ by its integral lattice $\Z$, or as a subset of the complex
numbers $
(\mathbb{R}/\mathbb{Z},+)  \simeq (S^1:=\{z \in \C \,\,\, \vert \,\,\, |z|=1\} ,
\cdot)
$
\end{example}

\begin{example}
A \emph{torus} is a compact, connected, abelian Lie group, and one can prove that
such group is isomorphic to a  product of circles, that is,
$
((\mathbb{R}/\mathbb{Z})^k,+)  \simeq ((S^1)^k:=\{z \in \C \,\,\, \vert \,\,\, |z|=1\}^k ,(\cdot,\ldots,\cdot)).
$
The integer $k$ is the \emph{dimension} of the torus. Other well known Lie groups are  the general linear group 
${\rm GL}(n,\mathbb{R})$ and the  
orthogonal group ${\rm O}(n)$ endowed with matrix multiplication.
 \end{example}

A subset $H$ of  $G$ is a  \emph{Lie subgroup} if it is a subgroup of $G$, a Lie group, and the 
inclusion $H \hookrightarrow G$ is an immersion.

\begin{theorem}[Cartan~\cite{EC}] \label{sub}
 A closed subgroup of a Lie group  is a Lie subgroup.
\end{theorem}

Let $G/H$ be the space of right cosets endowed with the quotient topology. One can prove that
there exists a unique smooth structure on $G/H$ for which the quotient map $G \to G/H$ is smooth. 
The $G$\--action on $G$ descends to an action on $G/H$: $g\in G$ acts on $G/H$ by sending 
$g'\, H$ to $(g\, g')\, H$.  In this case $G/H$ is called a \emph{homogeneous space}.

The  transformations between Lie groups are the Lie group homomorphisms. 
Let $(G,\cdot), (G',\star)$ be Lie groups. A \emph{Lie group homomorphism} is a smooth map $f: G \to G'$
such that $f(a \cdot b)= f(a) \star f(b)$ for all $a,b \in G$.
 A \emph{Lie group isomorphism} is a diffeomorphism $f: G \to G'$
such that $f(a \cdot b)= f(a) \star f(b)$ for all $a,b \in G$.

A \emph{Lie algebra} is a vector space $V$, together with a bilinear map 
$[\cdot,\cdot] \colon V \times V \to V,$ called a \emph{Lie bracket}, which satisfies:
$[\zeta,\eta]=-[\eta,\zeta]$ (\emph{antisymmetry}) and $[\zeta,[\eta,\rho]]+[\eta,[\rho,\zeta]]+
[\rho,[\zeta,\eta]]=0$ for all $\zeta,\eta,\rho \in V$ (\emph{Jacobi identity}). 

For instance,
the space of $n$\--dimensional matrices with real coefficients endowed with
the commutator of matrices as bracket, is a Lie algebra. 

For the properties of Lie algebras 
and for how from a given a Lie group one can define its associated Lie algebra,  
see for instance Duistermaat\--Kolk~\cite{DuKo}.   

For this paper, if
 $T$ is an $n$\--dimensional torus, a compact, connected, abelian
 Lie group and $1$ is the identity of $T$, 
 the \emph{Lie algebra of $T$} is the additive vector space 
$\mathfrak{t}:={\rm T}_1T$ endowed with the trivial bracket.

\subsection{Lie group actions} \label{sm}

Let  $(G,\, \star)$ be a  Lie group, and let $M$ be a smooth manifold.
A \emph{smooth $G$\--action} on $M$ is a smooth map
$
G \times M \to M,
$
denoted by
$
(g,\, x)\mapsto g \cdot x,
$
such that $e \cdot x=x$ and
  $g\cdot (h\cdot x)=(g\star h)\cdot x$, for all $g,h \in G$ and for all $x \in M$.

For instance, we have the following smooth actions.
 The map $S^1 \times \mathbb{C}^n \to \mathbb{C}^n$ on $\mathbb{C}^n$ given by
$
(\theta,\, (z_1,\,\,z_2,\,\ldots,\,z_n)) \mapsto (\theta\, z_1,\,\,z_2,\,\ldots,\,z_n)
$
is a smooth $S^1$\--action on $\mathbb{C}^n$.
Also, any Lie group $G$ acts on itself by \emph{left multiplication} 
 ${\rm L}(h) \colon g \mapsto g\,h$
 and analogously \emph{right multiplication}, and also by 
 the \emph{adjoint action}
$
{\rm Ad}(h) \colon g \mapsto hgh^{-1}.
$
If $G$ is abelian, then ${\rm Ad}(h)$ is the identity map for every $h \in G$.

 We say that
the $G$\--action is \emph{effective} if every element in $T$
moves at least one point in $M$, or equivalently
$
\cap_{x \in M} \, G_x=\{e\},
$ 
where
$
G_x:=\{t \in G \,\, | \,\, t \cdot x=x\}
$
is the \emph{stabilizer subgroup of the $G$\--action at $x$}.  The action is \emph{free} if
$G_x=\{e\}$ for every $x \in M$.  The action is \emph{semi\--free} if for every $x \in M$
either $G_x=G$ or $G_x=\{e\}$. 
The action is {\em proper} if for any compact 
subset $K$ of $M$ the set of all $(g,\, m)\in G\times M$ 
such that $(m,\, g\cdot m)\in K$ is compact in $G\times M$.

 \begin{remark}
There are obstructions to the existence of
effective smooth $G$\--actions on compact and non-compact manifolds,
even in the case that the $G$\--action is only required
to be smooth. For instance, in~\cite[Corollary in page 242]{Yau77} it is
proved that if $N$ is an $n$\--dimensional manifold on which
a compact connected Lie group $G$ acts effectively and there
are $\sigma_1,\ldots,\sigma_n \in {\rm H}^1(M,\mathbb{Q})$
such that $\sigma_1\cup \ldots \cup \sigma_n \neq 0$ then
$G$ is a torus and the $G$\--action is locally free. In~\cite{Yau77}
Yau also proves several other results giving restrictions on
$G$, $M$, and the fixed point set $M^G$. If the $G$\--action
is moreover assumed to be symplectic or K\"ahler, there are
even more non-trivial constraints. 
\end{remark}

The set
$G \cdot x:=\{t \cdot x \, |\, t \in G\}$ is 
the \emph{$G$\--orbit} that goes through the point $x$.

\begin{prop}
The stabilizer $G_x$ is a Lie subgroup of $G$.
\end{prop}

\begin{proof}
By Theorem~\ref{sub} it suffices to show if that $G_x$ is closed in $G$. 
Let $A \colon G \times M \to M$ denote the $G$\--action. For each
$x \in G$ let $i_x \colon G \to G\times M$ be the mapping $i_x(g):=(g,x)$. Then
$G_x=(i_x)^{-1}(A^{-1}(x))$, and hence $G_x$ is closed since $\{x\}$
is closed and $i_x$ and $\varphi$ are continuous mappings.
\end{proof} 

The following can be easily checked.

\begin{prop}
The action of $G$  is proper if $G$ is 
compact. 
\end{prop}

For each closed subgroup $H$ of $T$ which 
can occur as a stabilizer subgroup, 
the {\em orbit type} $M^H$ is defined as the set of all $x\in M$ 
such that $T_x$ is conjugate to $H$, but because $T$ is commutative 
this condition is equivalent to the equation $T_x=H$. 

Each connected component $C$ of $M^H$ is a smooth 
$T$\--invariant submanifold of $M$.  
The connected components of the orbit types in $M$ 
form a finite partition of $M$, which actually is a Whitney stratification.  
This is called the {\em orbit type stratification} of $M$.

There is a unique open orbit type, called the {\em principal orbit type}, 
which is the orbit type of a subgroup $H$ which is contained 
in every stabilizer subgroup $T_x$, $x\in M$. 

Because 
the effectiveness of the action means that the intersection 
of all the $T_x$, $x\in M$ is equal to the identity element, 
this means that the principal orbit type consists of the 
points $x$ where $T_{x}=\{ 1\}$, that is where the action is free. 
If the action is free at $x$, then  
the linear mapping $X\mapsto X_M(x)$ from $\got{t}$ to ${\rm T}_x M$ 
is injective.

\begin{theorem}[Theorem~2.8.5 in~Duistermaat\--Kolk~\cite{DuKo}]
The principal orbit type $M_{\rm{reg}}$ is 
a dense open subset of 
$M$, and connected if $G$ is connected.
\end{theorem}

The following notions will be essential later on.

\begin{definition}
The points $x\in M$ at which the $T$\--action 
is free are  called the {\em regular points} of $M$, 
and the principal orbit type, the set of all regular points 
in $M$, is denoted by $M_{\rm reg}$. 
The {\em principal orbits} 
are the orbits in $M_{{\rm reg}}$. 
 \end{definition}

Next we define equivariant maps, following on Definition~\ref{ed}.
Suppose that $G$ acts smoothly on $(M_1,\omega)$ and $(M_2,\omega_2)$. 

\begin{definition}
A \emph{$G$\--equivariant 
diffeomorphism} (resp. \emph{$G$\--equivariant embedding})
$\varphi \colon (M_1,\omega_1) \to (M,\omega_2)$
 is a symplectomorphism (resp. embedding) $\varphi \colon M_1 \to M_2$ such that
 $
 \varphi(g \cdot x )=g \star \varphi(x) \,\,\,\,\, \forall g \in G,\,\,\, \forall x \in M_1,
 $
where $\cdot$ denotes the $G$\--action on $M_1$ and $\star$ denotes the $G$\--action on $M_2$. In this case we say that $(M_1,\omega_1)$ and
$(M_2,\omega_2)$ are $G$\--\emph{equivariantly diffeomorphic} (resp. that $(M_1,\omega_1)$ is 
\emph{symplectically and $G$\--equivariantly embedded} in $(M_2,\omega_2)$).
\end{definition}

\medskip

It is useful to work with a notion of equivariant map up to reparametrizations of the acting
group.  If $G_1$ and $G_2$ are isomorphic Lie groups (possibly equal) acting symplectically on $(M_1,\omega_1)$ and
$(M_2,\omega_2)$ respectively, $\varphi \colon (M_1,\omega_1) \to (M,\omega_2)$ is
an \emph{equivariant diffeomorphism} (resp. \emph{equivariant embedding}) if it is a diffeomorphism
(resp. an embedding) for which there exists an isomorphism $f \colon G_1 \to G_2$
such that $\varphi(g \cdot x)=f(g) \star \varphi(x)\,\,\, \forall g \in G_1,\,\,\, \forall x \in M_1$,
where $\cdot$ denotes the $G_1$\--action on $M_1$ and $\star$ denotes the $G_2$\--action on $M_2$.
In this case we say that $(M_1,\omega_1)$ and
$(M_2,\omega_2)$ are \emph{equivariantly diffeomorphic} (resp. that $(M_1,\omega_1)$ is 
\emph{symplectically
and equivariantly embedded} in $(M_2,\omega_2)$). This more general notion is particularly
important when working on equivariant symplectic packing problems~\cite{Pe06,Pe07, PeSc08, FiPe15, FiPaPe15}, 
because if $f$ is
only allowed to be the identity these problems are too rigid to be of interest.

\begin{prop} \label{312}
For  $x\in M$, the stabilizer $G_x$ of a proper $G$\--action is compact and 
$A_x:g\mapsto g\cdot x:G\to M$ 
induces a smooth $G$\--equivariant embedding 
$
\alpha _x:g\, G_x\mapsto g\cdot x:G/G_x\to M 
$ with closed image equal to $G\cdot x$. 
\end{prop}

In this paper we will be concerned with actions on symplectic manifolds which preserve the
symplectic form (defined in Section~\ref{hd}).
We will later describe a result of Benoist
and Ortega\--Ratiu which gives a symplectic normal form for proper actions in the
neighborhood of any $G$\--orbit $G \cdot x$ of a symplectic manifold $(M,\omega)$, 
$x \in M$ (this is Theorem~\ref{Gthm}).

Let $T$ be an $n$\--dimensional torus with identity $1$ and 
let  $\mathfrak{t}:={\rm T}_1T$ be its Lie algebra. 
 Let $X \in \mathfrak{t}$.  There exists a unique \emph{homomorphism} $\alpha_X\colon
\R \to T$ with $\alpha_X(0)=1,\alpha_X'(0)=X$.  Define the so called
\emph{exponential mapping} ${\rm exp} \colon \mathfrak{t} \to T$ by
\begin{eqnarray} \label{emap}
{\rm exp}(X):=\alpha_X(1)
\end{eqnarray}
The exponential mapping ${\rm exp}:\mathfrak{t}\to T$ 
is a surjective homomorphism from the additive Lie group 
$(\mathfrak{t},+)$ onto $T$. Furthermore, $\mathfrak{t}_{\Z}:={\rm ker}({\rm exp})$ is a 
discrete subgroup of $(\mathfrak{t},\, +)$ and ${\rm exp}$ induces an isomorphism from 
$\mathfrak{t}/\mathfrak{t}_{\Z}$ onto $T$, which we also denote by ${\rm exp}$.

Because $\mathfrak{t}/\mathfrak{t}_{\Z}$ is compact, $\mathfrak{t}_{\Z}$ has a $\Z$\--basis 
which at the same time is an $\R$\--basis of $\mathfrak{t}$, and 
each $\Z$\--basis of $\mathfrak{t}_{\Z}$ is an $\R$\--basis of $\mathfrak{t}$.  

Using coordinates with respect to an ordered $\Z$\--basis of 
$\mathfrak{t}_{\Z}$, we obtain a linear isomorphism from $\mathfrak{t}$ 
onto $\R ^n$ which maps $\mathfrak{t}_{\Z}$ onto $\Z ^n$, and therefore 
induces an isomorphism from $T$ onto $\R ^n/\Z ^n$.  
The set $\mathfrak{t}_{\Z}$ is called the {\em integral lattice} 
in $\mathfrak{t}$.

 However, because we do not have a preferred 
$\Z$\--basis of $\mathfrak{t}_{\Z}$, we do not write $T=\R ^n/\Z ^n$.

  Using (\ref{emap}) one can
generate  vector fields on a smooth manifold from a given 
action.

\begin{definition}
For each $X \in \mathfrak{t}$, the vector field \emph{infinitesimal action} $X_M$ of $X$ on $M$
is defined by
\begin{eqnarray} \label{xm2}
X_M(x) :=\textcolor{black}{\text{tangent vector to}}
\underbrace{ t \mapsto
  \overbrace{{\rm exp}(tX)}^{\textup{\textcolor{black}{curve in}}\,\,
    T}\cdot x}_{\textup{\textcolor{black}{curve in}}\,\, M\,\,
  \textcolor{black}{\textup{through}}\,\,x}
\textcolor{black}{\text{at}}\,\, t=0,
\end{eqnarray}
i.e. $X_M(x)={\rm d}/{{\rm d} t}|_{t=0}  \,{\rm exp}(tX) \cdot x$.
\end{definition}

\medskip

In the sequel, ${\mathcal X}^{\infty}(M)$ denotes 
the Lie algebra of all smooth 
vector fields on $M$, provided with the 
\emph{Lie brackets} $[\mathcal{X},\, \mathcal{Y}]$ of $\mathcal{X}, \mathcal{Y} \in {\mathcal X}^{\infty}(M)$: 
$
[\mathcal{X},\, \mathcal{Y}]\, f:=\mathcal{X}\, (\mathcal{Y}\, f)-\mathcal{Y}\, (\mathcal{X}\, f),\,\, \forall f\in{\rm C}^{\infty}(M).
$ 
The Lie brackets vanish when the flows of  the vector fields 
$\mathcal{X}$ and $\mathcal{Y}$ commute. Therefore 
\begin{eqnarray} \label{xmym}
[X_M,Y_M]=0
\end{eqnarray} for all
$X,Y \in \mathfrak{t}$.

\subsection{Symplectic and Hamiltonian actions} \label{hd}

Let $(M, \,\omega)$ be a $2n$\--dimensional symplectic manifold. Let $G$ be a Lie group.

\begin{definition}
The action $ \phi \colon G \times M \to M$
is \emph{symplectic} if $G$ acts by symplectomorphisms, i.e. for each $t \in G$ the
diffeomorphism $\varphi_t \colon M \to M$ given  by
$
\varphi_t(x):= t \cdot x
$
is such that $(\varphi_t)^*\omega=\omega.$  The triple $(M,\omega,\phi)$ is a 
called a \emph{symplectic $G$\--manifold}.
\end{definition}

\medskip

Let ${\rm L}_{\mathcal{X}}$ denote the Lie derivative with respect to the 
vector field $\mathcal{X}$, and ${\rm i}_{\mathcal{X}}\omega$ the inner product 
of $\omega$ with $\mathcal{X}$, obtained by inserting 
$\mathcal{X}$ in the first slot of $\omega$.   The fact that the $T$\--action is symplectic says that
for every $X \in \mathfrak{t}$
\begin{equation}
{\rm d}({\rm i}_{X_M}\omega )={\rm L}_{X_M}\omega=0,
\label{LXsigma}
\end{equation} 
so the $1$\--form 
\begin{eqnarray} \label{xc}
{\rm i}_{X_M}\omega=\omega(X_M,\cdot) 
\end{eqnarray}
is closed. The first identity in (\ref{LXsigma}) 
follows from the homotopy 
identity 
\begin{eqnarray} \label{Lv}
{\rm L}_v ={\rm d}\circ{\rm i}_v +{\rm i}_v \circ {\rm d}
\end{eqnarray}
combined with ${\rm d} \omega =0$.   The case when (\ref{xc})
is moreover an exact form, for each $X \in \mathfrak{t}$, has
been throughly studied in the literature.

Indeed, there is a special type of symplectic actions which appear often in classical mechanics, 
and which enjoy a number of very interesting properties: they are called {\em Hamiltonian} actions,
named after  William Hamilton (1805-1865).

\medskip

Let $T$ be an $n$\--dimensional torus, a compact, connected, abelian
 Lie group, with Lie algebra $\mathfrak{t}$. Let $\mathfrak{t}^*$ be the dual of $\mathfrak{t}$.

\begin{definition} \label{xm}
A symplectic action $T \times M \to M$ is \emph{Hamiltonian} if there
is a smooth map $\mu \colon M \to \mathfrak{t}^*$ such that 
Hamilton's equation
\begin{eqnarray} \label{xts}
-{\rm d} \langle \mu (\cdot), \, X \rangle =  \textup{i}_{X_M}\omega:=\omega(X_M,\cdot),  \, \, \, \, \, \forall X \in \mathfrak{t},
\end{eqnarray}
holds, where
$X_M$ is the 
vector field infinitesimal action of $X$ on $M$, 
and the left hand\--side of equation (\ref{xts})  is the differential of the
real valued function $\langle \mu(\cdot) ,\, X\rangle$ obtained by evaluating elements
of $\mathfrak{t}^*$ on $\mathfrak{t}$. 
\end{definition}

There is a natural notion of symplectic and  Hamiltonian vector field. 
Given a smooth function $f \colon M \to \mathbb{R}$,
let $\mathcal{H}_f$ be the vector field defined by Hamilton's equation
$\omega({\mathcal{H}_f}, \cdot)=-{\rm d} f.$

\begin{definition}
We say that a smooth vector field $\mathcal{Y}$ on a symplectic manifold
$(M,\,\omega)$ is \emph{symplectic} if its flow preserves the
symplectic form $\omega$. We say that $\mathcal{Y}$ is \emph{Hamiltonian} if
there exists a smooth function $f \colon M \to \mathbb{R}$ such
that $\mathcal{Y}=\mathcal{H}_f$.
\end{definition}

\medskip

 It follows that  a $T$\--action on $(M,\, \omega)$ is symplectic if and only if all the
vector fields that it generates through (\ref{xm2}) are symplectic. A symplectic $T$\--action is
Hamiltonian if all the vector fields $X_M$ that it generates are
\emph{Hamiltonian}, i.e.  for each $X \in \mathfrak{t}$ 
there exists a smooth solution $\mu_X \colon M \to \mathbb{R}$, called a \emph{Hamiltonian}
or \emph{energy function}, to
$
-{\rm d} \mu_X =\omega(X_M,\, \cdot).
$

\begin{prop}
A any symplectic $T$\--action on a simply connected manifold $(M,\omega)$ is Hamiltonian.
\end{prop}
\begin{proof}
The obstruction for $\omega(X_m,\cdot)$
to being exact lies in the first cohomology group
of the manifold ${\rm H}^1(M,\mathbb{R})=0$. If the manifold is simply connected then $\pi_1(M)=0$,
and hence ${\rm H}^1(M,\mathbb{R})=0$.
\end{proof}

The natural transformations between symplectic manifolds $(M_1,\omega_1)$ and $(M_2,\omega_2)$
endowed with symplectic $T$\--actions are the $T$\--equivariant diffeomorphisms
which preserve the symplectic form, they are called \emph{$T$\--equivariant symplectomorphisms}.

\begin{remark}
 Kostant \cite{kostant1966} and Souriau \cite{souriau1966} gave the general notion
 of momentum map (we refer to Marsden\--Ratiu \cite[Pages 369, 370]{marsdenratiu} for the
history). The momentum map may be defined  generally for   a Hamiltonian action of a Lie group.  It
was a key tool in Kostant~\cite{kostant1970}
and Souriau discussed it at length in~\cite{souriau1970}. In this paper we only deal with
the momentum map for a Hamiltonian action of a torus. 
\end{remark}

\subsection{Conditions for a symplectic action to be Hamiltonian} \label{versus}

A Hamiltonian $S^1$\--action on a compact 
symplectic manifold $(M,\, \omega)$ has at least 
$\frac{1}{2} \dim M + 1$
fixed points. This follows from the fact that, if the fixed set is discrete,
then the  momentum map $\mu \colon M \to \R$  is a perfect Morse function
whose critical set is the fixed set.
Therefore,  the number of fixed points is equal to the rank of $\sum_{i=1}^{1/2 \dim M} {\rm H}_i(M;\R)$. Finally,
this sum is at least
$\frac{1}{2}\dim M +1$ 
because 
$ [1],\, [\omega],\, 
\big[\omega^2 \big], \ldots, \big[ \omega^{\frac{1}{2} \dim M}\big]$
are distinct cohomology
classes. 

We are not aware of general criteria to detect when a symplectic action is Hamiltonian,
other than in a few specific situations.  In fact, one striking question is:

\begin{question} \label{qh}
\emph{Are there non\--Hamiltonian symplectic $S^1$\--actions on compact connected
symplectic manifolds with non\--empty discrete fixed point set?} 
\end{question}

\medskip

In recent years there has been a flurry of
activity related to this question, see for instance  Godinho~\cite{Go1, Go2}, Jang~\cite{dj0, dj},
 Pelayo\--Tolman~ \cite{PeTo}, and Tolman\--Weitsman~\cite{ToWe}. Recently Tolman constructed
 an  example~\cite{To15} answering Question~\ref{qh} in the positive.

 \begin{theorem}[Tolman~\cite{To15}] \label{MT}
 There exists  a symplectic non\--Hamiltonian $S^1$\--action on a compact connected manifold
 with exactly $32$ fixed points.
 \end{theorem}
 
  Tolman and Weitsman proved the answer to the question is no for 
 semifree symplectic actions (they used equivariant cohomological methods, briefly 
 covered here in Section~\ref{abbv}).

 \begin{theorem}[Tolman\--Weitsman \cite{ToWe}]
 Let $M$ be a compact, connected symplectic manifold,
equipped with a semifree, symplectic $S^1$\--action with isolated
fixed points. Then if there is at least one fixed point, the circle action is 
 Hamiltonian.
 \end{theorem}
 
 In the K\"ahler case the answer to the question is a classical theorem
 of Frankel (which started much of the activity on the question).

 \begin{theorem}[Frankel \cite{Frankel1959}] \label{fr}
 Let $(M,\omega)$ be a compact connected K\"ahler
manifold admitting an $S^1$\--action  preserving the K\"ahler
structure. If the $S^1$-action has fixed points, then the
action is Hamiltonian.
\end{theorem}

 Ono \cite{Ono1984} proved the analogue of Theorem~\ref{fr} for compact Lefschetz manifolds and 
 McDuff \cite[Proposition 2]{MD} proved the a symplectic version of Frankel's
 theorem (later generalized by Kim~\cite{kim} to arbitrary dimensions).
 
 \begin{theorem}[McDuff~\cite{MD}]
A symplectic  $S^1$\--action on a compact connected symplectic $4$\--manifold with fixed points must be
 Hamiltonian. 
 \end{theorem}
 
  On the other hand, McDuff \cite[Proposition 1]{MD}  constructed a compact 
connected symplectic $6$-manifold with a non\--Hamiltonian symplectic  $S^1$\--action which has fixed point 
set equal to a union of tori.

 Less is known for higher dimensional Lie groups; the following corresponds to  \cite[Theorem 3.13]{giacobbe}.
  
 \begin{theorem}[Giacobbe \cite{giacobbe}] \label{gc}
  A symplectic
action of a $n$\--torus on a compact connected symplectic $2n$\--manifold with fixed points must be Hamiltonian.
\end{theorem}

Theorem~\ref{gc} appears as~\cite[Corollary 3.9]{DuPe}.
 If $n=2 $ this is deduced from the classification 
of symplectic $4$\--manifolds with symplectic 2-torus actions in \cite[Theorem 8.2.1]{Pe} (Theorem~\ref{main:thm}
covered later in this paper) in view of \cite[Theorem 1.1]{DuPecois}. 

There are also related results by Ginzburg describing the obstruction to
the existence of a momentum map for a symplectic action, see \cite{Gi}, where he 
showed that a symplectic action can be decomposed as cohomologically free action, and a Hamiltonian action. 

In the present paper we will focus on symplectic actions of tori of dimension
$k$ on manifolds of dimension $2n$ where $k \geq n$. For these
there exist recent complete classifications of certain classes of symplectic actions,
which are described in the following sections, in which a complete answer to the following more
general question can be given in terms of the vanishing of certain invariants. 

\begin{question}
\emph{When is a symplectic torus action on a compact connected symplectic manifold
Hamiltonian? Describe precisely the obstruction to being Hamiltonian.}
\end{question}

\medskip

  When in addition to being symplectic the action is
Hamiltonian, then necessarily  $n=k$, but there are many  non\--Hamiltonian symplectic
actions when $n=k$, and also when $n\geq k+1$.

\subsection{Monotonic symplectic $G$\--invariants}

Let $G$ be a Lie group. We denote the collection of all 
$2n$\--dimensional symplectic $G$ manifolds by $\sympG$.
The set $\sympG$ is a category with morphisms given
by $G$\--equivariant symplectic embeddings. We call a subcategory $\mathcal{C}_G$
of $\sympG$ a \emph{symplectic $G$\--category} if $(M,\om, \phi)\in\mathcal{C}_G$ implies
$(M,\la \om, \phi)\in\mathcal{C}_G$ for any $\la\in\R\setminus\{0\}$.

\begin{definition}\label{def_eqvcap}
Let $\mathcal{C}_G$ be a symplectic $G$\--category.
 A \emph{generalized symplectic
 $G$\--capacity} on $\mathcal{C}_G$ is a map $c\colon \mathcal{C}_G\to[0,\infty]$ satisfying:
i) \emph{Monotonicity}:  if 
   $(M,\om,\phi),(M',\om',\phi')\in\mathcal{C}_G$ and there exists a 
   $G$\--equivariant symplectic embedding $M\immG M'$ then $c(M,\om,\phi)\leq c(M',\om',\phi')$; ii) \emph{Conformality}: if $\la\in\R\setminus\{0\}$ and 
   $(M,\om,\phi)\in\mathcal{C}_G$ then $c(M,\la\om, \phi) = \abs{\la} c(M,\om,\phi)$.
 \end{definition} 
  
  \medskip

\begin{definition}
  For $(N,\om_N,\phi_N)\in\mathcal{C}_G$ we say that $c$ satisfies \emph{$N$\--non-triviality} if $0<c(N)<\infty$.
\end{definition}

\begin{definition}
 We say that $c$ is \emph{tamed} by $(N,\om_N,\phi_N)\in\sympG$ if
 there exists
 some $a\in(0,\infty)$ such that the following two properties hold:\\
 (1) if $M\in\mathcal{C}_G$ and there exists a $G$\--equivariant symplectic
   embedding $M\immG N$ then $c(M)\leq a$;\\
 (2) if $P\in\mathcal{C}_G$ and there exists a $G$\--equivariant symplectic
   embedding $N\immG P$ then $a\leq c(P)$.
 \end{definition}
  
\medskip

For any integer $1\leq d\leq n $ the standard action of the $d$\--dimensional torus $\T^d$ on $\C^n$ is given by 
 $
 \phi_{\C^n}\big((\al_i)_{i=1}^d, (z_i)_{i=1}^n\big) = (\al_1 z_1, \ldots, \al_d z_d, z_{d+1}, \ldots, z_n)$. 
 It induces actions of $\T^d=\T^k\times\T^{d-k}$
on $\mathrm{B}^{2n}(1)$ and $\mathrm{Z}^{2n}(1)$, which in turn induce
the  actions of
$\T^k\times\R^{d-k}$ on $\mathrm{B}^{2n}(1)$ and $\mathrm{Z}^{2n}(1)$
for  $k\leq d$. The action of an
element of $\T^k\times\R^{d-k}$ is the
action of its image under the quotient map
$\T^k\times\R^{d-k} \to \T^d$.
We endow $\mathrm{B}^{2n}(1)$ and $\mathrm{Z}^{2n}(1)$
with these actions.

 \begin{definition}\label{def_RTcapacity}
 A generalized symplectic $(\T^k\times\R^{d-k})$\--capacity is a \emph{symplectic 
 $(\T^k\times\R^{d-k})$\--capacity} if it is
 tamed by $\mathrm{B}^{2n}(1)$ and $\mathrm{Z}^{2n}(1)$.
 \end{definition}

\medskip

Given integers $0\leq k\leq m\leq n$ we define
the \emph{$(m,k)$\--equivariant Gromov radius} 
$\cBmk \colon \sympplain^{2n, \R^k}  \to [0,\infty]$,
 $(M,\om,\phi) \mapsto\mathrm{sup} \{\,r>0 \mid \mathrm{B}^{2m}(r) \immrk M \,\}$, 
where $\immrk$ denotes a symplectic $\R^k$\--embedding
and $\mathrm{B}^{2m}(r)\subset\C^m$ has the $\R^k$\--action
given by rotation of the first $k$ coordinates.

\begin{prop}[\cite{FiPaPe15}]
 If $k\geq 1$, 
 $\cBmk$
 is a symplectic $\R^k$\--capacity.
\end{prop}

Let $\vol(E)$ denote the symplectic volume of a subset $E$ of a symplectic
manifold and $\sympT$ 
the category of $2n$\--dimensional symplectic toric manifolds.
A \emph{toric ball packing} $P$ of $M$ is given by a disjoint collection
of symplecticly and $\T^n$\--equivariantly embedded balls.
As an application of symplectic $G$\--capacities to Hamiltonian $\mathbb{T}^n$\--actions we define
the \emph{toric packing capacity}
$\toricpack\colon  \sympT \to [0, \infty]$,
 $$                  (M, \om, \phi) \mapsto \left( {\sup\{\,\vol(P) \mid P \textrm{ is a toric ball packing of }M\,\}}/{\vol(\mathrm{B}^{2n}(1))}\right)^{\frac{1}{2n}},$$

\begin{prop}[\cite{FiPaPe15}] 
$\toricpack\colon\sympT\to[0,\infty]$ is a
 symplectic $\T^n$\--capacity.
\end{prop}

In general, symplectic $G$\--capacities provide a general setting to define
monotonic invariants of integrable systems. There should be many such
invariants but so far few are known beyond the toric case, and the
semitoric case also discussed in~\cite{FiPaPe15}.

\section{Properties of Hamiltonian actions} \label{ham2:sec}

\subsection{Marsden\--Weinstein symplectic reduction}

Even though one cannot in general take quotients of symplectic manifolds by group
actions and get again a symplectic manifold, for Hamiltonian actions we have the following
notion of ``symplectic quotient".
 
\begin{theorem}[Marsden\--Weinstein~\cite{WeMa74}, Meyer~\cite{Me73}] \label{MaWe}
 Let $(M,\omega)$ be a symplectic manifold and let $G$ be a compact Lie group with Lie
 algebra $\mathfrak{g}$ acting
 on $(M,\omega)$ in a Hamiltonian fashion with momentum map $\mu \colon M \to \mathfrak{g}^*$.
 Let $i \colon \mu^{-1}(\lambda) \to M$ be the inclusion map and suppose that $G$ acts freely
 on $\mu^{-1}(\lambda)$. Then the orbit space $M_{\rm red,\lambda}:=\mu^{-1}(\lambda)/G$ is a smooth manifold, the
 projection $\pi \colon \mu^{-1}(\lambda) \to M_{\rm red,\lambda}$ is a principal $G$\--bundle, and there
 is a symplectic form $\omega_{\rm red, \lambda}$ on $M_{\rm red,\lambda}$ such that 
 $\pi^*\omega_{\rm red,\lambda}=i^*\omega$. 
  \end{theorem}
 
 The ``symplectic quotient" $(M_{\rm red},\omega_{\rm red})$ is called the \emph{Marsden\--Weinstein 
 reduction of $(M,\omega)$ for the $G$\--action at $\lambda$}.  Symplectic reduction has
 numerous applications in mechanics and geometry, see for instance~\cite{marsdenratiu}. Here
 we will give an one in the proof of the upcoming result Theorem~\ref{delzant}.

\subsection{Atiyah\--Guillemin\--Sternberg convexity}

It follows from equation (\ref{xts}) that Hamiltonian $T$\--actions on compact connected manifolds have fixed points
because zeros of the vector field $X_M$ correspond to critical points of $ \langle \mu (\cdot), \, X \rangle $ and $ \langle \mu (\cdot), \, X \rangle $ always has critical
points in a compact manifold. 
The Atiyah-Guillemin-Sternberg Convexity Theorem (1982, \cite{atiyah,gs}) says that 
$\mu(M)$ is  a convex polytope.

\begin{theorem}[Atiyah \cite{atiyah}, Guillemin\--Sternberg
  \cite{gs}] \label{gs} If an $m$\--dimensional torus $T$ acts on a
  compact, connected $2n$\--dimensional symplectic manifold $(M,\,
  \omega)$ in a Hamiltonian fashion, the image $\mu(M)$ under the
  momentum map $ \mu \colon M \to \mathfrak{t}^*$ is 
   the convex hull of  the image under $\mu$ of the fixed point set of the $T$\--action.
   In particular, $\mu(M)$ is a convex polytope in $\mathfrak{t}^*$.
\end{theorem}

The fixed point set in this theorem is given by a collection of symplectic submanifolds of $M$.

The polytope $\mu(M)$ is called the \emph{momentum polytope of $M$}. One precedent of this 
result appears in Kostant's article~\cite{kostant}. 

 Other convexity theorems were proven later by Birtea\--Ortega\--Ratiu~\cite{biorra09}, Kirwan \cite{Ki} (in the case of compact,
non\--abelian group actions), Benoist \cite{benoist}, and  Giacobbe \cite{giacobbe}, to name a few. 
Convexity in the case of Poisson
actions has been studied by Alekseev, Flaschka\--Ratiu, Ortega\--Ratiu
and Weinstein \cite{Alek, FlaRat, ORbook, Weinstein} among others.

Given a point $x \in \mu(M)$, its preimage $\mu^{-1}(x)$ is connected (this is known as Atiyah's
connectivity Theorem). Moreover, it is diffeomorphic to a torus of dimension $\ell$, where $\ell$ is the
dimension of the lowest dimensional face $F$ of $\mu(M)$  such that $x \in F$.

 The symplectic form $\omega$ must vanish
along each $\mu^{-1}(x)$, that is, $\mu^{-1}(x)$ is isotropic.

\subsection{Duistermaat\--Heckman theorems} \label{hans:sec}

Roughly at the same time as Atiyah, Guillemin\--Sternberg proved the convexity theorem,
Duistermaat and Heckman proved an influential result which we will describe next.
Let $(M,\omega)$ be a $2n$\--dimensional symplectic manifold.   Suppose that $T$ is
a torus acting on a symplectic manifold $(M,\omega)$ in a symplectic Hamiltonian fashion with momentum map $\mu \colon M \to \mathfrak{t}^*$.
Assume moreover that $\mu$ is \emph{proper}, that is, for every compact $K \subseteq \mathfrak{t}^*$, the preimage
$\mu^{-1}(K)$ is compact.

\begin{definition}
The \emph{Liouville measure} of a Borel subset
$U$ of $M$ is
$
m_{\omega}(U):=\varint_{U} \frac{\omega^n}{n!}.
$
The \emph{Duistermaat\--Heckman measure} $m_{\rm DH}$ on $\mathfrak{t}^*$
is the push\--forward measure $\mu_{*}m_{\omega}$ of $m_{\omega}$ by $\mu \colon M \to \mathfrak{t}^*$.
\end{definition}

\medskip

Let $\lambda$ be the Lebesgue measure in $\mathfrak{t}^* \simeq \R^m$.

\begin{theorem}[Duistermaat\--Heckman~\cite{DuHe, DuHe2}]
There is a function $f \colon \mathfrak{t}^* \to \R$ such that $f$ is a polynomial of degree 
at most $n-m$ on each component of regular values of $\mu$, and
$
m_{\omega}(U)=\varint_U f \, {\rm d}\lambda.
$
\end{theorem}

\begin{definition}
The function $f$ is called called the \emph{Duistermaat-Heckman polynomial}.
\end{definition}

\medskip

In the case  of the symplectic\--toric manifold $(S^2,\omega,S^1)$, the Liouville
measure is $m_{\omega}([a,b])=2\pi (b-a)$ and the Duistermaat\--Heckman
polynomial is the characteristic function $2\pi \, \chi_{[a,b]}$.

If $T$ acts freely on $\mu^{-1}(0)$, then it acts freely on fibers $\mu^{-1}(t)$ for
which $t \in \mathfrak{t}^*$ is closed to $0$. Consider the Marsden\--Weinstein reduced space 
$M_t=\mu^{-1}(t)/T$ (see~\cite[Section~4.3]{am})  with the reduced symplectic form $\omega_t$ 
as in the proof  of Theorem~\ref{delzant}.

\begin{theorem}[Duistermaat\--Heckman~\cite{DuHe, DuHe2}] \label{dh2}
The cohomology class $[\omega_t]$ varies linearly in t.
\end{theorem}

Theorem~\ref{dh2} does not hold for non-proper momentum
maps, see~\cite[Remark 4.5]{NON}.

\subsection{Atiyah\--Bott\--Berline\--Vergne localization} \label{abbv}

A useful tool in the study of properties of symplectic $S^1$\--actions has been
\emph{equivariant cohomology}, because it encodes well the fixed point set information.
Within symplectic geometry, equivariant cohomology is an active
area, and in this paper we do not touch on it but only give the very basic definition,
a foundational result, and an application to symplectic $S^1$\--actions. Although
equivariant cohomology may be defined generally we concentrate on the case
of $S^1$\--equivariant cohomology.

\begin{definition}
Let $S^1$ act on a smooth manifold $M$. The \emph{equivariant cohomology of $M$}
is
$
{\rm H}^*_{S^1}(M) :={\rm H}^*(M \times_{S^1} S^{\infty}).
$
\end{definition}

\begin{example}
If $x$ is a point then
${\rm H}^*_{S^1}(x,\Z) = {\rm H}^*(\mathbb{C}P^\infty,\, \Z) = \Z[t]$.
\end{example}

If $V$ is an equivariant vector bundle over $M$, 
then the \emph{equivariant Euler  class of } $V$ is the Euler class
of the vector bundle $V \times_{S^1} S^\infty$ over $M \times_{S^1} S^\infty$.
The \emph{equivariant Chern classes} of equivariant complex vector bundles
are defined analogously.

If $M$ is oriented and compact then 
the projection map $\pi \colon M \times_{S^1} S^\infty \to \mathbb{C}P^\infty$
induces a natural push\--forward map, denoted by  $\varint_M$:
$\pi_* \colon {\rm H}_{S^1}^{i}(M,\Z)
\to {\rm H}^{i - \dim M} (\mathbb{C}P^\infty,\, \Z).$
In particular $\pi_*(\alpha) = 0$ for all $\alpha \in {\rm H}^i_{S^1}(M;\Z)$ 
when $i < \dim M$.  For a component $F$ of the fixed point set $M^{S^1}$ we denote by
${\rm e}_{S^1}({\rm N}_F)$ the equivariant Euler class of the normal
bundle to $F$.

\begin{theorem}[Atiyah\--Bott~\cite{AB}, Berline\--Vergne~\cite{BV}] \label{ABBV}
Fix $\alpha \in {\rm H}_{S^1}^*(M,\Q)$. As elements of $\Q(t)$,
$$\varint_M \alpha = \sum_{F \subset M^{S^1}} \varint_F \frac{ \alpha|_F}{{\rm e}_{S^1}({\rm N}_F)},$$
where the sum is over all fixed components $F$.
\end{theorem}

Let $S^1$ act symplectically on a symplectic manifold  $(M,\, \omega)$, and 
let $J \colon {\rm T}M \to {\rm T}M$ be a compatible almost
complex structure.
If $p \in M^{S^1}$ is an isolated fixed point, then there are
well-defined non-zero integer weights 
$\xi_1,\,\ldots,\,\xi_n$ in the isotropy representation ${\rm T}_p M$ (repeated with multiplicity). 
Indeed, there exists an identification of  ${\rm T}_pM$ with $\mathbb{C}^n$,
where the $S^1$ action on $\mathbb{C}^n$ is
given by  $\lambda \cdot (z_1,\ldots,z_n)=
(\lambda^{\xi_1}z_1,\ldots,\lambda^{\xi_n}z_n);$ the
integers $\xi_1,\ldots, \xi_n$ are determined, up to permutation,
by the $S^1$\--action and the symplectic form.
The restriction of
the  $i^{th}$\--equivariant Chern class  $p$ is given by
$
{\rm c}_i(M)|_p = \sigma_i(\xi_1,\ldots,\xi_n) \, t^i
$
where $\sigma_i$ is the $i$'th elementary symmetric polynomial
and $t$ is the generator of ${\rm H}^2_{S^1}(p,\Z)$.
For example, ${\rm c}_1(M)|_p = \sum_{i=1}^n \xi_i t$
and  the equivariant Euler class 
of the tangent bundle at $p$
is given by
$
{\rm e}_{S^1}({\rm N}_p) ={\rm c}_n(M)|_p =  \left( \prod_{j=1}^n \xi_j \right) \, t^n.
$
Hence,
$$\varint_p \frac{{\rm c}_i(M)|_p}{{\rm e}_{S^1}({\rm N}_p)}= \frac{ \sigma_i(\xi_1,\ldots,\xi_n) }
{  \prod_{j=1}^n \xi_j } \, t^{i -n}.$$

We can naturally identify
 ${\rm c}_1(M)|_p$ with an integer 
 ${\rm c}_1(M)(p)$: the sum
 of the weights at $p$.  
 Let $\Lambda_p$ be the product of the weights (with multiplicity) in the isotropy representation 
${\rm T}_p M$ for all $p \in M^{S^1}$.

 \begin{definition}
 The mapping
 $
{\rm c}_1(M) \colon M^{S^1} \to \Z$ defined by $p \mapsto {\rm c}_1(M)(p) \in \Z
 $
is called the \emph{Chern class map} of $M$.
\end{definition}

\begin{prop}[\cite{PeTo}] \label{Azero}
Let $S^1$ act symplectically
on compact symplectic $2n$\--manifold $(M,\omega)$ with isolated fixed points.
If the range of  ${\rm c}_1(M) \colon M \to \Z$ contains at most $n$ elements, 
then 
$$
\sum_{\substack{p \in M^{S^1}\\ 
{\rm c}_1(M)(p)=k}} \! \!  \frac{1}{\,\Lambda_p} = 0.
$$
for every $k \in \Z$.
\end{prop}

\begin{proof}
Let 
$\{{\rm c}_1(M)(p) \, \, | \,\, p \in M^{S^1}\}=
\{k_1,\, \ldots,\, k_{\ell}\}$ and define for $i \in \{1,\dots,\ell\}$,
$A_i:=\sum_{\substack{p \in M^{S^1}\\ 
{\rm c}_1(M)(p)=k_i}} \! \!  \frac{1}{\,\Lambda_p}.
$
Consider the $\ell \times \ell$
matrix $\mathcal{B}$  given by
$
\mathcal{B}_{ij}:= (k_i)^{j-1}$, where $1 \leq i,j \leq \ell$.
Since  $\ell \leq n$ by assumption, 
$
\varint_M {\rm c}_1(M)^j = 0
$
for all $j < \ell$. Applying Theorem~\ref{ABBV} to the elements
$1, \,{\rm c}_1(M), \dots, \,{\rm c}_1(M)^{\ell - 1}$
gives
a homogenous system of linear equations
$
\mathcal{B}\cdot (A_1,\,\ldots,\,A_{\ell})=(0,\,\ldots,\,0).
$
Since $\mathcal{B}$ is a Vandermonde matrix, we have that 
${\rm det}(\mathcal{B}(\ell)) \neq 0.$ 
Thus, it follows  that 
$
A_1= \dots =A_{\ell}\,=\,0.
$
\end{proof}

Let $X$ and $Y$ be sets and let $f\colon X \to Y$ be a map. 
We recall that
$f$ is \emph{somewhere injective} 
if 
there is a point $y\in Y$ such that $f^{-1}(\{y\})$
is the singleton.

\begin{theorem}[\cite{PeTo}]  \label{general}
Let $S^1$ act symplectically
on compact symplectic $2n$\--manifold $(M,\omega)$ with isolated fixed points.
If  the Chern class map is somewhere injective,
then the circle action has at least 
$n + 1$
fixed points.
\end{theorem}

\begin{proof}
Since the Chern class map is somewhere injective
there is $k \in \Z$ such that
$\sum_{\substack{p \in M^{S^1}\\ 
{\rm c}_1(M)(p)=k}} \! \!  \frac{1}{\,\Lambda_p} \neq 0.$
By Proposition~\ref{Azero}, this
implies that the range of the Chern class map contains at least
$n + 1$ elements;
a fortiori,
the action has at least $n + 1$ fixed points.
\end{proof}

\subsection{Further Topics}

There exists an extensive theory of Hamiltonian actions and related topics,
 see for instance the books by Guillemin~\cite{guillemin}, Guillemin\--Sjamaar~\cite{gusj2005}, 
and Ortega\--Ratiu~\cite{ORbook}.

There are many influential works which we do not describe here
for two reasons, brevity being the main one, but also because
they are more advanced and more suitable for a survey
than a succinct invitation to the subject. These works include: 
Sjamaar\--Lerman~\cite{SjLe91} work on stratifications;
Kirwan's convexity theorem~\cite{Ki} (which
generalizes the Atiyah\--Guillemin\--Sternberg theorem to
the non abelian case); and
Lerman's symplectic cutting procedure~\cite{Le95} (a procedure to ``cut" a symplectic manifolds that has many
  applications in equivariant symplectic geometry and completely integrable systems).

\section{Examples} \label{examples2}

\subsection{Symplectic Hamiltonian actions}

The following is an example of a  Hamiltonian symplectic action.

\begin{example} \label{hs}
The first example of a Hamiltonian torus action is
$
(S^2, \, \omega=\textup{d}\theta \wedge\textup{d}h)
$
equipped with the
rotational circle action $\mathbb{R}/\mathbb{Z}$
about the vertical axis of $S^2$ (depicted in Figure~\ref{fg}).
This action has
momentum map  $\mu \colon S^2 \to {\rm Lie}(S^1)={\rm T}_1(S^1) \simeq \mathbb{R}$
equal to the height function 
$\mu(\theta, \,h)= h$, and in this case the 
momentum polytope image is the interval
$\Delta=[-1,\,1].$
Another example (which generalizes this one)  is
the $n$\--dimensional complex projective space equipped
with a $\lambda$\--multiple, $\lambda>0$, of the Fubini\--Study form
$(\mathbb{C}P^n,\, \lambda \cdot \omega_{\rm{FS}})$ and the
rotational $\mathbb{T}^n$\--action induced from the
rotational $\mathbb{T}^n$\--action on the $(2n+1)$\--dimensional
complex plane. This action is Hamiltonian, with
momentum map 
$$
\mu^{\mathbb{C}P^n,\lambda} \colon z=[z_0:z_1:\ldots,z_n] \mapsto 
\Big(\frac{\lambda |z_1|^2}{\sum_{i=0}^n|z_i|^2},\ldots,\frac{\lambda |z_n|^2}{\sum_{i=0}^n|z_i|^2} \Big).
$$
The associated momentum 
polytope is
$
\mu^{\mathbb{C}P^n,\lambda}(\mathbb{C}P^n)=\Delta=\textup{convex hull }\{0,\, \lambda {\rm e}_1,\ldots,\lambda {\rm e}_n\},
$
where ${\rm e}_1=(1,0,\ldots,0),\ldots,{\rm e}_n=(0,\ldots,0,1)$ are the canonical basis vectors of $\R^n$.
\end{example}

The category
of Hamiltonian actions, while large, does not include some simple examples of symplectic actions, for
instance free symplectic actions on compact manifolds, because Hamiltonian
actions on compact manifolds always have fixed points.

\subsection{Symplectic non Hamiltonian actions}

 In the following three examples, there does not exist a momentum map; they are examples
 of what we later call ``maximal symplectic actions'' (discussed in Section~\ref{pss}).

\begin{example}
\normalfont
The $4$\--torus
$
(\R/\Z)^2 \times (\R/\Z)^2
$
endowed with the
standard symplectic form, on which the $2$\--dimensional torus $T:=(\R/\Z)^2$
acts by multiplications on two of the copies of $\R/\Z$ inside
of $(\R/\Z)^4$, is symplectic manifold with symplectic orbits which are $2$\--tori. 
\end{example}

\begin{example}
\normalfont
\label{ex1}
Let $M:=S^2 \times (\R/\Z)^2$ be
endowed with the product symplectic form of the standard area form
on $(\R/\Z)^2$ and the standard area form on $S^2$.  
Let $T:=(\R/\Z)^2$ act on $M$ by translations on the right factor.  This is a free symplectic
action the orbits of which are symplectic $2$\--tori. \end{example}

\begin{example} \normalfont
 \label{ex2}
Let $P:=S^2 \times (\R/\Z)^2$ equipped with the product symplectic form of the standard symplectic (area) form
on  $S^2$ and the standard area form on the sphere $(\R/\Z)^2$. The $2$\--torus $T:=(\R/\Z)^2$ acts freely by translations on the right factor
of $P$. Let the finite group $\Z/2\,\Z$ act on $S^2$ by rotating
each point horizontally by $180$ degrees, and let  $\Z/2\,\Z$ act
on $(\R/\Z)^2$ by the antipodal action on the first circle $\R/\Z$. 
The diagonal action of $\Z/2\,\Z$ on $P$ is free. Therefore, the quotient space 
$
S^2 \times_{\Z/2\,\Z} (\R/\Z)^2
$ 
is a smooth manifold.
Let $M:=S^2 \times_{\Z/2\,\Z} (\R/\Z)^2$ be endowed 
with the symplectic form $\omega$ and $T$\--action inherited from
the ones given in the product $S^2 \times (\R/\Z)^2$, where $T=(\R/\Z)^2$. The action
of $T$ on $M$ is not free, and the $T$\--orbits are symplectic $2$\--dimensional tori.
Notice that the orbit space $M/T$ is $S^2/(\Z/ 2\, \Z)$, which is a 
smooth orbifold with two singular points of order $2$, the South and North poles of $S^2$ (this
orbifold will play an important role in the classification of maximal symplectic actions). 
\end{example}

The following two are examples of  what we later call coisotropic 
actions (discussed in Section~\ref{pss}). 

\begin{example} \normalfont \label{ktexample}
(Kodaira~\cite{kodaira} and Thurston~\cite{Th}) The first example of a (non\--Hamiltonian) symplectic torus action with coisotropic (in fact, Lagrangian) principal orbits is
the Kodaira variety \cite{kodaira} (also known as the Kodaira\--Thurston manifold \cite{Th}), which is a torus bundle over a torus
constructed as follows. Consider the product symplectic manifold
$
(\mathbb{R}^2 \times (\mathbb{R}/\mathbb{Z})^2,\,
 \textup{d}x_1 \wedge \textup{d}y_1 +
\textup{d}x_2 \wedge \textup{d}y_2),
$
where $(x_1,\,y_1) \in \R^2$ and $(x_2,\,y_2) \in (\R/\Z)^2$. Consider the action of 
$(j_1, \,j_2) \in \mathbb{Z}^2$ on $(\mathbb{R}/\mathbb{Z})^2$ by the matrix group
consisting of
$
\left
( \begin{array}{cc}
1 & j_2 \\
0 & 1  \\
\end{array} \right),
$
where $j_2 \in \mathbb{Z}$ (notice that $j_1$ does not appear intentionally in the matrix). The quotient of this symplectic manifold by  the diagonal action of
$\Z^2$ gives rise to a compact, connected, symplectic $4$\--manifold
\begin{eqnarray} \label{ktm}
({\rm KT},\omega):=(\mathbb{R}^2 \times_{\mathbb{Z}^2} (\mathbb{R}/\mathbb{Z})^2, \,\, \,
 \textup{d}x_1 \wedge \textup{d}y_1 +
\textup{d}x_2 \wedge \textup{d}y_2)
\end{eqnarray}
on which the $2$\--torus
$T:=\mathbb{R}/\mathbb{Z} \times \mathbb{R}/\mathbb{Z}$ acts symplectically and freely, where
the first circle acts on the $x_1$\--component, and the second 
circle acts on the $y_2$\--component (one can check that this action is
well defined). Because the $T$\--action is free, all the orbits are principal, and
because the orbits are obtained by keeping the $x_2$\--component and the
$y_1$\--component fixed, ${\rm d}x_2={\rm d}y_1=0,$ so the 
orbits are Lagrangian. 
\end{example}

The symplectic manifold in (\ref{ktm}) fits in  the third case 
in Kodaira \cite[Theorem~19]{kodaira}. Thurston 
rediscovered it \cite{Th}, and observed that there exists no K\"ahler structure on ${\rm KT}$
which is compatible with the symplectic form (by noticing that
the first Betti number $\rm{b}_1({\rm KT})$ is $3$).  
It follows that no other symplectic $2$\--torus action on ${\rm KT}$ is Hamiltonian because in
that case it would be toric and  $\rm{b}_1({\rm KT})$ would vanish since
toric varieties are simply connected (Proposition~\ref{danilov}).

\begin{figure}[htbp]
  \begin{center}
    \includegraphics[height=2.5cm, width=7.5cm]{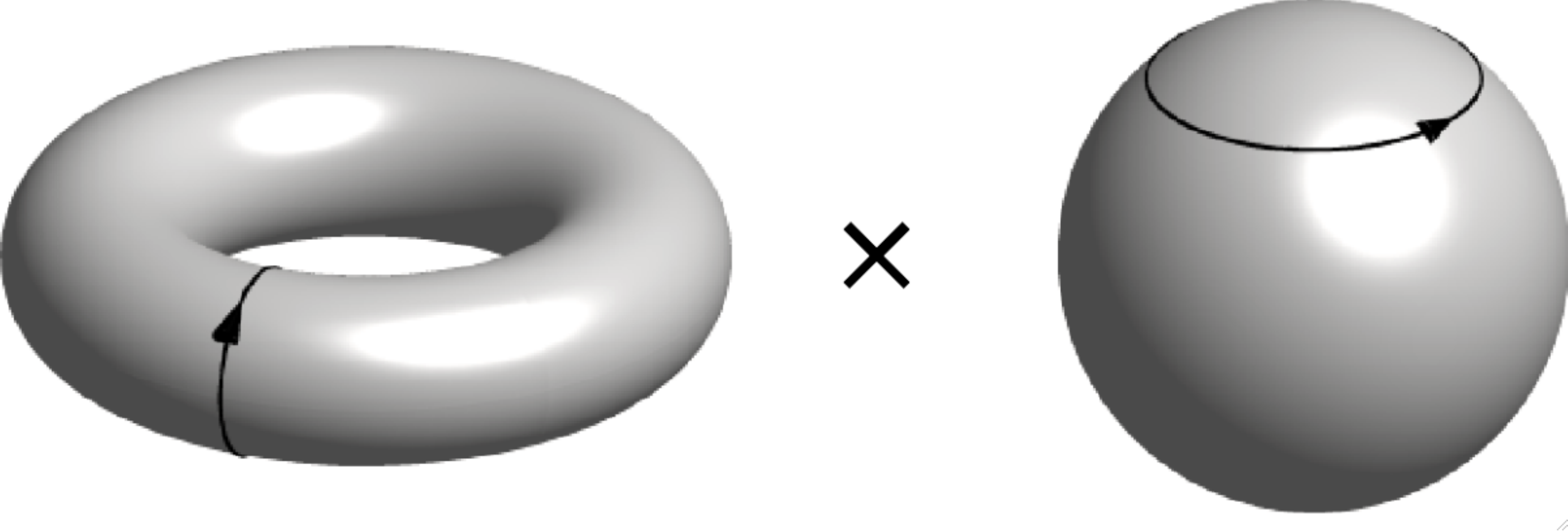}
    \caption{A symplectic  $2$\--torus action on $S^2 \times (\R/\Z)^2$.}
    \label{sph}
    \end{center}
\end{figure}

\begin{example}
\normalfont
This is an example of a  non\--Hamiltonian, non\--free symplectic $2$\--torus action on a compact, connected,
symplectic $4$\--manifold. Consider the compact symplectic $4$\--manifold
$
(M,\omega):=((\mathbb{R}/\mathbb{Z})^2 \times S^2, \, \, \, \textup{d}x  \wedge \textup{d}y 
+\textup{d}\theta  \wedge \textup{d}h ).
$
There is a natural action of the $2$\--torus
$T:=\mathbb{R}/\mathbb{Z} \times \mathbb{R}/\mathbb{Z}$ on expression $(M,\omega)$,
where the first circle of $T$ acts on the first circle $\R/\Z$ of the left factor of $M$,
and the right circle acts on $S^2$ by rotations (about the
vertical axis); see Figure~\ref{sph}. This $T$\--action is symplectic. 
However,  it is not a Hamiltonian action because
it does not have fixed points. It is also not
free, because the stabilizer subgroup of a point  $(p,\,q)$, where 
$q$ is the North or South pole of $S^2$, is a circle
subgroup.   In this case the principal orbits are the
products of the circle orbits of the left factor $(\R/\Z)^2$,
and the circle orbits of the right factor (all orbits of the right factor are circles but the
North and South poles, which are fixed points).  Because
these orbits are obtained by keeping the $y$\--coordinate
on the left factor constant, and the height on the right factor
constant, ${\rm d}y={\rm d}h=0,$ which implies that
the product form vanishes along the principal orbits, which are
Lagrangian, and hence coisotropic.
\end{example}

  \section{Classifications of Hamiltonian actions} \label{hck}
  
 \subsection{Classification of symplectic\--toric manifolds} \label{dp}

\subsubsection{Delzant polytopes} \label{dp}

Let $\Delta$ be an $n$\--dimensional convex polytope in $\mathfrak{t}^*$. 
Let $F$ be the set of all codimension one faces of $\Delta$. Let $V$ be the set of 
 vertices of $\Delta$. For every $v\in V$, we write 
 $
F_v=\{ f\in F\mid v\in f\},
$
that is, $F_v$ is the set of faces of $\Delta$ which contain the vertex $v$. 
Following Guillemin \cite[page~8]{guillemin} we define a very special type of polytope.

\begin{definition} \label{dpol}
We say that $\Delta$ is a {\em Delzant polytope} if: i) for each $f\in F$ there are $X_f\in \mathfrak{t}_{\Z}$ 
and $\lambda _f\in\R$ such that the hyperplane which contains 
$f$ has defining equation $\langle X_f,\,\xi\rangle +\lambda _f=0$, $\xi \in \mathfrak{t}^*$, and 
$\Delta$ is contained in the set of $\xi\in \mathfrak{t}^*$
such that $\langle X_f,\,\xi\rangle +\lambda _f\geq 0$;
ii) for every $v\in V$, $\{X_f \,\, | \,\, f \in F_v\}$ is
a $\Z$\--basis of $\mathfrak{t}_{\Z}$. 
\end{definition}

\medskip

The definition of a Delzant polytope implies the following.

\begin{lemma}
Let $\Delta$ be a Delzant polytope in $\mathfrak{t}^*$. Then for each $f \in F$
there exists $X_f \in \mathfrak{t}_{\mathbb{Z}}$ and $\lambda_f \in \mathbb{R}$
such that $\Delta =\{\xi\in \mathfrak{t}^*\mid 
\langle X_f,\,\xi\rangle +\lambda _f\geq 0\quad \forall f\in F\}.$ 
\end{lemma}

\begin{cor}
For every $v \in V$, $\# (F_v)=n$. 
\end{cor}

\begin{figure}[htbp]
  \begin{center}
    \includegraphics[width=6cm]{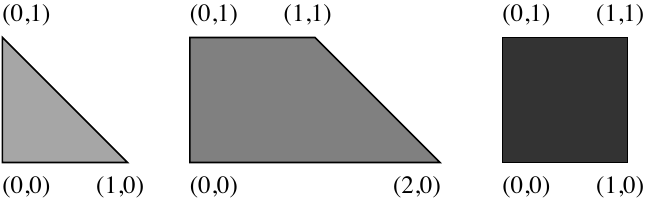}
    \caption{The two left most polygons  are Delzant. The right $3$\--polytope is
    not Delzant (there are four vertices meeting at the frontal vertex).}
    \label{fg}
  \end{center}
\end{figure}

 Delzant (1988, \cite{De}) proved that if the Hamiltonian action in the convexity theorem
 (Theorem~\ref{gs}) is effective and $m=n,$ 
 then $\mu(M)$ is a \emph{Delzant polytope}.

For any $z\in\C ^F$ and $f\in F$ we write $z(f):=z_f,$ 
which we view as the coordinate of the vector $z$ with the index $f$. 
Let $\pi \colon \R ^F \to \mathfrak{t}$ be the linear map
$
\pi (t):=\sum_{f\in F}\, t_f\, X_f.
$
Because, for any vertex 
$v$ of the Delzant polytope $\Delta$, the $X_f$ with $f\in F_v$ form a $\Z$\--basis 
of $\mathfrak{t}_{\Z}$ which is also an $\R$\--basis of $\mathfrak{t}$, 
we have $\pi (\Z ^F)=\mathfrak{t}_{\Z}$ and $\pi (\R ^F)=\mathfrak{t}$. 
It follows that:

\begin{prop}
The map $\pi$ induces a 
surjective Lie group homomorphism  $\pi ' \colon \R ^F/\Z ^F=(\R /\Z)^F \to 
\mathfrak{t}/\mathfrak{t}_{\Z},$ and hence a
surjective homomorphism 
${\rm exp}\circ\pi ' \colon \R ^F/\Z ^F \to T.$
\end{prop}

Write ${\got n}:={\rm ker}\pi$ and 
\begin{eqnarray} \label{ntorus}
N={\rm ker}({\rm exp}\circ\pi '),
\end{eqnarray}
which is a compact abelian 
subgroup of  $\R ^F/\Z ^F$. 
Actually, $N$ is connected (see~\cite[Lemma 3.1]{DuPeTV}), and isomorphic 
to ${\got n}/{\got n}_{\Z}$, where 
${\got n}_{\Z}:={\got n}\cap\Z ^F$ is the integral lattice 
in ${\got n}$ of the torus 
$N$.

\subsubsection{Symplectic\--toric manifolds}

The following notion has been a source of  inspiration to many authors working on symplectic
and Hamiltonian group actions, as well as finite dimensional integrable Hamiltonian systems.

\begin{definition}
A \emph{symplectic toric manifold} is a compact connected 
symplectic manifold $(M, \omega)$ of dimension $2n$ endowed with an effective Hamiltonian 
action of a torus $T$ of dimension $n$. 
\end{definition}

\begin{figure}[htbp]
  \begin{center}
    \includegraphics[width=6cm]{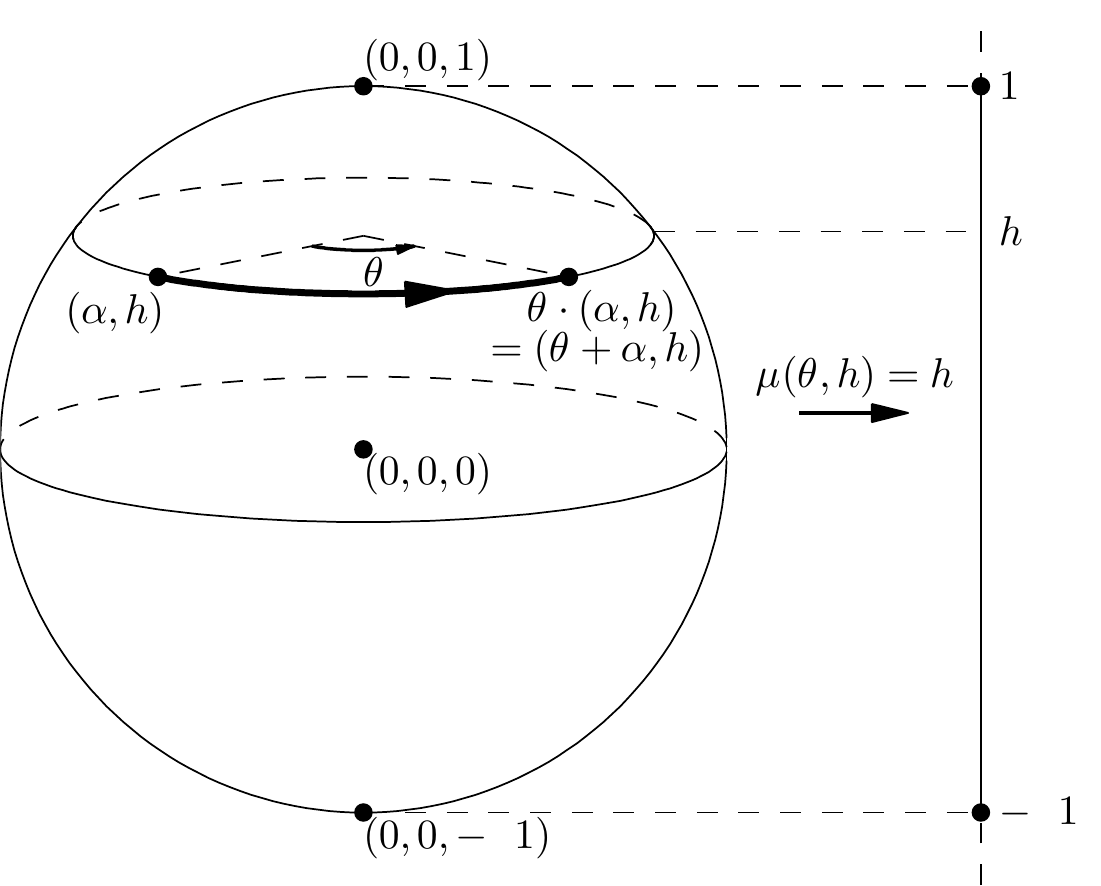}
    \caption{The simplest symplectic\--toric manifold: $S^2$
    endowed with the standard area form and rotational action of $S^1$ about the vertical axis.}
    \label{fg}
  \end{center}
\end{figure}

\medskip

  For instance, the effective $S^1$\--action by rotations about $z=0$ of $S^2$ (Figure~\ref{fg}) is symplectic and Hamiltonian, and
    hence $(S^2,\omega,S^1)$ is a symplectic\--toric manifold, where $\omega={\rm d}\theta \wedge {\rm d}h$ is
    the standard area form in spherical coordinates.  
      If $(x,y,z) \notin \{(0,\,0,\,-1),\, (0,\, 0,\, 1)\}$, then $T_{(x,y,z)}=\{e\}$, and 
    $T_{(0,\,0,\,-1)}=T_{(0,\,0,\,1)}=S^1$.  
     Identifying the dual Lie algebra of $S^1$ with $\mathbb{R}$ (by choosing a basis), 
    the momentum map is, in spherical coordinates, given by $\mu(\theta,\, z)=z$. The image of 
    $\mu$ is the closed interval $[-1,\,1]$, which  is
      convex hull of the
      image of the fixed point set $\{(0,\,0,\,-1),\, (0,\, 0,\, 1)\}$ as in Theorem~\ref{gs}. 
      
      According to 
       the following result by Thomas Delzant, the interval $[-1,1]$ completely
       characterizes the symplectic geometry of $(S^2,\omega,S^1)$ (see Figure~\ref{fg}).

\begin{theorem}[Delzant \cite{De}] \label{delzant}
Given any Delzant polytope $\Delta \subset \mathfrak{t}^*$, there exists a
  compact connected symplectic manifold $(M_{\Delta},\,
  \omega_{\Delta})$ with an effective Hamiltonian $T$\--action with
  momentum map $\mu_{\Delta} \colon M_{\Delta} \to \mathfrak{t}^*$ such that
  $\mu_{\Delta}(M_{\Delta})=\Delta$. Any symplectic toric manifold $(M,\omega)$ 
  is $T$\--equivariantly symplectomorphic to $(M_{\mu(M)},\omega_{\mu(M)})$.
   Two symplectic\--toric manifolds are $T$\--equivariantly
  symplectomorphic if and only if they have the same momentum map image,
  up to translations.  
 \end{theorem}

\begin{proof}
We will prove the existence part~\cite[pages~328, 329]{De} following~\cite{DuPeTV}: 
we are going to prove that for any Delzant polytope $\Delta$ 
there exists a  symplectic\--toric manifold $M_{\Delta}$ such that $\mu _{\Delta}(M_{\Delta})=\Delta$,
and which is obtained as the reduced phase space 
for a linear Hamiltonian action of the torus $N$ in (\ref{ntorus}) on a symplectic vector space 
$E$, at a value $\lambda _N$ of the momentum mapping of the 
Hamiltonian $N$\--action, where $E$, $N$ and $\lambda _N$ 
are determined by $\Delta$. 

On the complex vector space $\C ^F$ of all complex\--valued functions 
on $F$ we have the action of $\R ^F/\Z ^F$, where 
$t\in\R ^F/\Z ^F$ maps $z\in\C ^F$ to the element 
$t\cdot z\in\C ^F$ defined for $f \in F$ by
$
(t\cdot z)_f={\rm e}^{2\pi{\rm i}\, t_f}\, z_f
$ 
The infinitesimal action of $Y\in\R ^F={\rm Lie}(\R ^F/\Z ^F)$ 
is given by 
$
(Y\cdot z)_f=2\pi{\rm i}\, Y_f\,z_f,
$
which is a Hamiltonian vector field defined by the function 
$
z\mapsto\langle Y,\, \mu (z)\rangle 
=\sum_{f\in F} Y_f\, \frac{|z_f|^2}{2}=\sum_{f\in F}\, Y_f\, \frac{{x_f}^2+{y_f}^2}{2},
\label{mu} \nonumber
$
and with respect to the symplectic form 
$
\omega^{\C^F} :=\frac{{\rm i}}{4\pi}\, \sum_{f\in F}\,{\rm d} z_f
\wedge{\rm d}\overline{z}_f
=\frac{1}{2\pi}\,\sum_{f\in F}\,{\rm d} x_f\wedge{\rm d} y_f,
$
if $z_f=x_f+{\rm i}y_f$, with $x_f,\, y_f\in\R$. 
Since the right hand side of (\ref{mu}) depends linearly on $Y$, 
we  view $\mu (z)$ as an element of $(\R ^F)^*\simeq\R ^F$, 
with the coordinates
\begin{equation}
\mu (z)_f=|z_f|^2/2=({x_f}^2+{y_f}^2)/2, \quad f\in F.  
\label{muf}
\end{equation}
In other words, the action of $\R ^F/\Z ^F$ on $\C ^F$ is Hamiltonian
with respect to  $\omega^{\C^F}$ and with 
momentum map $\mu :\C ^F\to ({\rm Lie}(\R ^F/\Z ^F))^*$ 
given by (\ref{muf}).   It follows that the action of $N$  on 
$\C ^F$ is Hamiltonian with momentum map 
$
\mu _N:=\iota _{\got n}^*\circ\mu :\C ^F\to {\got n}^*,
$
where $\iota _{\got n}:{\got n}\to\R ^F$ denotes the identity 
viewed as a linear mapping from ${\got n}\subset\R ^F$ 
to $\R ^F$, and its transposed $\iota _{\got n}^*:(\R ^F)^*\to {\got n}^*$ 
 assigns to a linear form on $\R ^F$ its 
restriction to ${\got n}$.

Let $\lambda$ denote the element of $(\R ^F)^*\simeq\R ^F$ 
with the coordinates $\lambda _f$, $f\in F$.  
Write $\lambda _N=\iota _{\got n}^*(\lambda)$. 
It follows from Guillemin \cite[Theorem 1.6 and Theorem 1.4]{guillemin}  that 
$\lambda _N$ is a regular value of $\mu _N$. Hence 
$Z:={\mu _N}^{-1}(\{ \lambda _N\})$
 is a smooth 
submanifold of $\C ^F$, and that the action of $N$ on $Z$ 
is proper and free. As a consequence the $N$\--orbit space 
$M_{\Delta}:=Z/N$ is a smooth $2n$\--dimensional manifold such that 
the projection $p:Z\to M_{\Delta}$ exhibits $Z$ as a principal $N$\--bundle 
over $M_{\Delta}$ 
for the $N$\--action at $\lambda _N$, see Theorem~\ref{MaWe}). Moreover, there is a unique symplectic form 
$\omega _{\Delta}$ on $M_{\Delta}$ such that 
$p^*\omega _{\Delta}={\iota _Z}^*\omega^{\C^F},
$
where $\iota _Z$ is the identity viewed as a smooth mapping from 
$Z$ to $\C ^F$ 
($M_{\Delta}$ is  the 
Marsden\--Weinstein reduction~\cite[Section~4.3]{am} of  $(\C ^F,\,\omega^{\C^F} )$).  On the $N$\--orbit space $M_{\Delta}$, 
we still have the action of the 
torus $(\R ^F/\Z ^F)/N\simeq T$, with momentum mapping 
$\mu _{\Delta}:M\to \mathfrak{t}^*$ determined by 
$
\pi ^*\circ\mu _{\Delta} \circ p =(\mu -\lambda )|_Z. 
\label{muT}
$
The torus $T$ acts effectively on $M$ and $\mu _{\Delta}(M)=\Delta$, 
see Guillemin \cite[Theorem~1.7]{guillemin}, and therefore we have
constructed the symplectic\--toric manifold
$
(M_{\Delta},\omega_{\Delta})
$
with $T$\--action and momentum map $\mu_{\Delta} \colon M_{\Delta} \to {\got t}^*$ 
such that $\mu_{\Delta}(M)=\Delta$, from $\Delta$.
\end{proof}

There have been  generalizations of Theorem~\ref{delzant},
for instance to multiplicity\--free group actions by 
Woodward~\cite{Wo}, and to symplectic\--toric orbifolds by Lerman and Tolman~\cite{LeTo}.
An extension to noncompact symplectic manifolds was recently given by  Karshon and Lerman~\cite{LeKa15}.

Because any symplectic\--toric manifold is obtained by symplectic
reduction of $\C^F$, it admits a compatible 
$T$\--invariant K\"ahler metric.  Delzant \cite[Section~5]{De} observed that $\Delta$
gives rise to a fan, and that the symplectic toric manifold with Delzant polytope 
$\Delta$ is $T$\--equivariantly 
diffeomorphic to the {\em toric variety} $M^{{\rm toric}}$ 
defined by the fan. Here $M^{{\rm toric}}$ is 
a complex $n$\--dimensional complex analytic manifold, 
and the action of $T$ on $M^{{\rm toric}}$ 
has an extension to a complex analytic action on $M^{{\rm toric}}$ 
of the complexification $T_{\C}$ of $T$.  

A detailed study of the
relation between the symplectic\--toric manifold and $M^{{\rm toric}}$
appears in~\cite{DuPeTV}. 
The following proof illustrates the interplay between the symplectic and algebraic view
points in toric geometry.

\begin{prop}
Every  symplectic\--toric manifold is simply connected. 
\label{danilov}
\end{prop}
\begin{proof}
Every symplectic toric manifold may be provided with the 
structure of a toric variety defined by a complete fan, 
cf. Delzant \cite[Section~5]{De} and Guillemin \cite[Appendix 1]{guillemin}.
On the other hand, Danilov \cite[Theorem~9.1]{danilov} observed that such a 
toric variety is simply connected. The proof goes as follows:
the toric variety has an open cell which is isomorphic to 
the complex space $\C ^n$, whose complement is a complex subvariety 
of complex codimension one. Hence all loops may be deformed 
into the cell and contracted within the cell to a point. 
\end{proof}

\subsection{Log symplectic\--toric manifolds}

Recently there has been a generalization of symplectic\--toric geometry to 
a class of Poisson manifolds, called \emph{log\--symplectic manifolds}. Log\--symplectic manifolds are generically symplectic but degenerate along a normal crossing configuration of smooth hypersurfaces.  

Guillemin, Miranda, Pires and Scott initiated the study of log symplectic\--toric manifolds in 
their article~\cite{Guillemin3}. They proved the analogue of Delzant's theorem (Theorem~\ref{delzant}) in the case where the degeneracy locus  of the associated Poisson structure is a smooth hypersurface.  

Degeneracy loci for Poisson structures are often singular. In \cite{GuLiPeRa} the authors consider 
the mildest possible singularities (normal crossing hypersurfaces) and gave an
analogue of Theorem~\ref{delzant}. Next we informally state this result to give a flavor of the
ingredients involved (being precise would be beyond the scope of this paper). 

The notion of \emph{isomorphism} below generalizes
the classical notion taking into account the log symplectic structure.

\begin{theorem}[\cite{GuLiPeRa}] There is a one-to-one correspondence between isomorphism classes of oriented compact connected log symplectic\--toric $2n$\--manifolds and equivalence classes of pairs $(\Delta, M)$, where $\Delta$ is a compact convex log affine polytope of dimension $n$ satisfying the Delzant condition and $M \to \Delta$ is a principal $n$\--torus bundle over $\Delta$ with vanishing toric log obstruction class.
\end{theorem}

Log\--symplectic geometry and its toric version are an active area of research 
which is  related to tropical geometry and the extended tropicalizations of toric varieties defined by  Kajiwara~\cite{MR2428356} and Payne~\cite{MR2511632}. 
 
    Convexity properties of Hamiltonian torus actions on log\--symplectic manifolds were studied in \cite{GMPS},
    where the authors prove a generalization of Theorem~\ref{gs}.

\subsection{Classification of Hamiltonian $S^1$\--spaces}

In addition to Delzant's classification (Theorem~\ref{delzant}) there have been other classifications
of Hamiltonian $G$\--actions on compact symplectic $2n$\--manifolds. In this section we outline the 
classification when $G=S^1$ and $n=2$ due to Karshon.

\begin{definition}
A \emph{Hamiltonian $S^1$\--space} is a compact connected 
symplectic $4$\--manifold equipped with an effective Hamiltonian $S^1$\--action 
\end{definition}

\medskip

Let $(M,\omega,S^1)$ be a Hamiltonian $S^1$\--space.
We associate it a labelled graph as follows. Let
$\mu \colon M \to \mathbb{R}$ be the momentum map of the $S^1$\--action.
For each component $\Sigma$ of the set of fixed points of the $S^1$\--action 
there is one vertex in
the graph, labelled by  $\mu(\Sigma) \in \mathbb{R}$.

If $\Sigma$ is a surface, the corresponding vertex has two
additional labels, one is the symplectic area of $\Sigma$ and the
other one is the genus of $\Sigma$.

Let $F_k$ be a subgroup of $k$ elements of $S^1$. For  every
connected component $C$ of the set of points fixed by $F_k$ there is an
edge in the graph, labeled by the integer $k>1$. The component $C$ is
a $2$\--sphere, which we call a \emph{$F_k$\--sphere}. The quotient
circle $S^1/F_k$ rotates it while fixing two points, and the two
vertices in the graph corresponding to the two fixed points are
connected in the graph by the edge corresponding to $C$.

  \begin{theorem}[Audin, Ahara, Hattori \cite{AH,  A1, A2}]
  Every Hamiltonian
$S^1$\--space is $S^1$\--equivariantly
diffeomorphic to a complex surface with a holomorphic $S^1$\--action
which is obtained from $\mathbb{CP}^2$, a Hirzebruch surface, or a
$\mathbb{CP}^1$\--bundle over a Riemann surface (with appropriate
circle actions), by a sequence of blow\--ups at the fixed points.
\end{theorem}
 
Let $A$ and $B$ be connected components of the set of fixed
points. The $S^1$\--action extends to a holomorphic action of the
group $\mathbb{C}^{\times}$ of non\--zero complex numbers. Consider
the time flow given by the action of the subgroup $\{{\rm exp}(t)\,\, |\,\,t \in
\mathbb{R}\}$.  

\begin{definition}
We say that \emph{$A$ is greater than $B$} if there is
an orbit of the $\mathbb{C}^{\times}$\--action which at time
$t=\infty$ approaches a point in $A$ and at time $t=-\infty$
approaches a point in $B$. 
\end{definition}

\medskip

Take any of the complex surfaces with $S^1$\--actions considered by
Audin, Ahara and Hattori, and assign a real parameter to every
connected component of the fixed point set such that these
parameters are monotonic with respect to the partial ordering we have
just described. If the manifold contains two fixed surfaces we
assign a positive real number to each of them so that the
difference between the numbers is given by a formula involving the
previously chosen parameters. 

Karshon proved that
for every such a choice of parameters there exists an invariant
symplectic form and a momentum map on the complex surface such that
the values of the momentum map at the fixed points and the symplectic
areas of the fixed surfaces are equal to the chosen
parameters. Moreover, every two symplectic forms with this property
differ by an $S^1$\--equivariant diffeomorphism.  

\begin{theorem}[Karshon \cite{K}] \label{karshon:thm} If two 
  Hamiltonian $S^1$\--spaces have the same graph, then they are
   $S^1$\--equivariantly symplectomorphic. Moreover, every
  compact $4$\--dimensional Hamiltonian $S^1$\--space is $S^1$\--equivarianly symplectomorphic
  to one of the spaces listed in the paragraph above.
\end{theorem}

A generalization of this classification result  to higher dimensions has been recently obtained  by 
Karhson\--Tolman~\cite{kato}. The authors construct all possible 
Hamiltonian symplectic torus actions for which all the nonempty
reduced spaces are two dimensional (and not single points), the manifold is
connected and the momentum map is proper as a map to
a convex set.

The study of symplectic and Hamiltonian circle actions has been an active topic
of current research, see for instance McDuff\--Tolman~\cite{McTo2007}, where
they show many interesting properties, for instance that
 if the weights of a Hamiltonian $S^1$\--action on a compact symplectic symplectic manifold
 $(M,\omega)$ at the points at which the momentum map is a maximum 
 are sufficiently small, then the circle represents a nonzero element of 
 $\pi_1({\rm Ham}(M,\omega))$, where ${\rm Ham}(M,\omega)$ is the group of Hamiltonian
 symplectomorphisms of  $(M,\omega)$.
 
  In \cite{GoSa2015} Godinho and Sabatini  
 construct an algorithm to obtain linear relations among the weights at the fixed points which
 under certain conditions determines a family of vector spaces which contain the admissible 
 lattices of weights.
 
  Concerning symplectic $S^1$\--actions (not necessarily Hamiltonian),  see 
 Godinho's articles~\cite{Go1,Go2}.

\subsection{Hamiltonian $(S^1 \times \mathbb{R)}$\--actions and classification of symplectic\--semitoric manifolds}

Semitoric systems, also called symplectic semitoric manifolds, are a rich
 class of integrable systems which, in the case of compact phase space,
 take place on the Hamiltonian $S^1$\--spaces of the previous section.
  
  Let $(M,\omega)$ be a symplectic $4$\--dimensional manifold.
The \emph{Poisson brackets} of two real valued smooth functions $f$ and $g$ on $M$ 
are defined by
$
\{ f,\, g\} :=\omega (\mathcal{H}_f,\,\mathcal{H}_g).
$

\begin{definition}
An \emph{integrable system with two degrees of freedom} is a smooth map
$F=(f_1,f_2) \colon M \to \mathbb{R}^2$ such that $\{f_1,f_2\}=0$ and the Hamiltonian vector fields
$\mathcal{H}_{f_1},\mathcal{H}_{f_2}$ are linearly independent almost everywhere.
\end{definition}

\medskip

A theorem of Eliasson characterizes  the so called ``non\--degenerate"
 (the term ``non\--degenerate'' is a generalization of 
 ``Morse non\--degenerate" which is more involved to define~\cite{pevn11} here). The following
 is a particular instance of Eliasson's general theorem, of interest to us here.

\begin{theorem}[Eliasson~\cite{eliasson,eliasson-these}] \label{eli}
Let $F:=(f_1,f_2) \colon (M,\omega) \to \mathbb{R}^2$ be an integrable system with two degrees of
freedom all of the singularities
of which are non\--degenerate, and with no hyperbolic blocks. There exist local symplectic coordinates $(x_1,x_2,\xi_1,\xi_2)$
 about every non\--degenerate critical point $m$, in which $m=(0,0,0,0)$, and
$
 (F-F(m))\circ \varphi = g \circ (q_1,q_2),
 $ 
 where $\varphi=(x_1,x_2,\, \xi_1,\xi_2)^{-1}$
 and $g$ is a diffeomorphism from a small neighborhood of the origin
 in $\mathbb{R}^4$ into another such neighborhood, such that $g(0,0,0,0)=(0,0,0,0)$ and 
$(q_1,\,q_2)$ are, depending
on the rank of the critical point, as follows. If $m$ is a critical point of $F$ of rank zero, then
  $q_j$ is one of
  \begin{itemize}
  \item[{\rm (i)}] $q_1 = (x_1^2 + \xi_1^2)/2$ and $q_2 = (x_2^2 +
    \xi_2^2)/2$.
  \item[{\rm (ii)}] $q_1=x_1\xi_2 - x_2\xi_1$ and $q_2 =x_1\xi_1
    +x_2\xi_2$;   \end{itemize}
If $m$ is a critical point of $F$ of rank one, then 
 $q_1 = (x_1^2 + \xi_1^2)/2$ and $q_2 = \xi_2$.
\end{theorem}

The assumption of not having hyperbolic blocks is simply to reduce the complexity of the
statement of the theorem, but is not really needed to understand the discussion which
follows. 

\begin{remark}
The analytic case of Theorem~\ref{eli} is due to
R{\"u}{\ss}mann~\cite{russmann} for two degrees of freedom systems
and to Vey~\cite{vey} in any dimension.
\end{remark}

\begin{definition}
A \emph{semitoric system} $F:=(f_1,f_2) \colon M \to \mathbb{R}^2$ on a connected symplectic
$4$\--manifold $(M,\, \omega)$ is an integrable system with two degrees of freedom such that $f_1$  is the
  momentum map of a Hamiltonian $S^1$\--action,  $f_1$ is a proper map, and 
   the singularities of  $F$ are non\--degenerate, without hyperbolic blocks, and hence they
   are of the form given in Theorem~\ref{eli}.
 \end{definition}

\begin{remark}
In the definition above, $f_1$ gives rise to a Hamiltonian $S^1$\--action on $M$,
and $f_2$ gives rise to a Hamiltonian $\mathbb{R}$\--action on $M$; conceptually,
their flows, one after the other, produce a Hamiltonian $(S^1 \times \mathbb{R})$\--action;
details of the precise relation between $(S^1\times \mathbb{R})$\--actions and semitoric systems 
are spelled out in~\cite[Section~3]{FiPaPe15}.
\end{remark}

\begin{remark}
If $M$ is compact then $(M,\omega)$ endowed with the
  Hamiltonian $S^1$\--action with momentum map $f_1$ is  a Hamiltonian $S^1$\--space.  
 \end{remark}
 
 \begin{definition}
   Suppose that
  $F_1=(f^1_1,f^1_2) \colon (M_1,\omega_1) \to \R^2$ and 
   $F_2=(f^2_1,f^2_2) \colon (M_2,\omega_2) \to \R^2$ are semitoric systems. We say that they
  are
  \emph{isomorphic} if there exists a symplectomorphism $\varphi \colon
  (M_1,\omega_1) \to (M_2,\omega_2),$ and a smooth map $\phi \colon F_1(M_1)\to \mathbb{R}$ with
  $\partial_2 \phi \neq 0$, such that 
 $
\varphi^*f^2_1=  f^1_1$ and $\varphi^*f_2^2 =
  \phi(f^1_1,f^1_2).$
\end{definition}

\medskip

Semitoric systems are classified under the assumption that each singular fiber
contains at most one singular point of type (1.ii); these points are called
\emph{focus\--focus} (or \emph{nodal}, in algebraic geometry). The singular
fiber containing a focus\--focus point is a $2$\--torus pinched precisely
at the focus\--focus point (i.e. topologically a $2$\--sphere with the north
and south poles identified). Semitoric systems satisfying this condition are
called \emph{simple}.
 
\begin{theorem} [\cite{PeVN09,PeVN11}] \label{InAc}
 Simple semitoric systems $(M,\, \omega,\,  F:=(f_1,f_2))$ are determined, up to isomorphisms, by 
a convex polygon endowed with a collection of interior points, each of which is 
labelled by a tuple
$(k \in \mathbb{Z}, \, \sum^{\infty}_{i,j=1} a_{ij}X^iY^j).$
Here $\Delta$ is obtained from $F(M)$ by appropriately unfolding the singular affine structure induced
by $F$, $k$ encodes how twisted the singular Lagrangian fibration $F$ is between consecutive focus\--focus points  arranged according to the first component of their image in $\mathbb{R}^2$, 
and the Taylor series
 $\sum_{i,j=1}^{\infty} a_{ij}X^iY^j$ encodes the singular dynamics of the vector fields 
 $\mathcal{H}_{f_1},\mathcal{H}_{f_2}$. Conversely, given a polygon\footnote{this is really not any polygon, but a polygon
 of so called semitoric type, which generalize the notion of Delzant polygon (which was applicable
 to toric systems) to this more general context.} with interior points $p_1,\ldots,p_n$, and for each $p_{\ell}$ a label
$
(k \in \mathbb{Z}, \,\,\sum_{i,j=1}^{\infty} a_{ij}X^iY^j),
$ 
one can construct  $(M,\, \omega)$ and a semitoric system $F \colon M \to \mathbb{R}^2$ having this data as  invariants.
\end{theorem}

In \cite{Pa} Palmer defined the moduli space of semitoric systems, which is an incomplete metric space,
and constructed its completion. In \cite{KaPaPe} the connectivity properties of this space were studied
using ${\rm SL}_2(\mathbb{Z})$ equations.

Four\--dimensional symplectic\--toric manifolds are a very particular case
of compact semitoric systems (in which the manifold is closed, and the only
invariant is the convex polygon). Every semitoric system takes place
on a Hamiltonian $S^1$\--space, and the relation has been made
completely explicitly recently.
We call \emph{Karshon graph} the labelled directed graph in Theorem~\ref{karshon:thm}.

\begin{theorem}[Hohloch\--Sabatini\--Sepe \cite{HoSaSe}]
Let  $F:=(f_1,f_2) \colon (M,\omega) \to \mathbb{R}^2$ be a simple semitoric system on a compact manifold 
with $m_f$ focus\--focus critical
points and underlying Hamiltonian $S^1$\--space $(M, \omega, f_1)$ with momentum map $f_1$. Then the 
associated convex polygon in Theorem~\ref{InAc} and $m_f$ determines
the Karshon graph, thus classifying $(M, \omega, S^1)$ up to $f_1$\--equivariant
symplectomorphisms.
\end{theorem}

There has been recent work generalizing the convex polygon in Theorem~\ref{InAc} to higher
dimensional semitoric systems by Wacheux~\cite{wa1}. 

The Fomenko school has powerful and far reaching methods to study the topology of
singularities of integrable systems~\cite{bolsinov-fomenko-book}.

\section{Properties of symplectic actions} \label{pss}

\subsection{Fundamental form of a symplectic action} \label{ffg}
 
Let $T$ be a torus with Lie algebra $\mathfrak{t}$. Suppose that $T$ acts symplectically on a connected symplectic manifold $(M,\omega)$.

\begin{prop}
There is a unique antisymmetric bilinear form 
$\omega^{\mathfrak{t}}$ on $\mathfrak{t}$, which we call the ``fundamental form", such that 
\begin{eqnarray} 
\omega^{\mathfrak{t}}(X, Y)=\omega _x(X_M(x),\, Y_M(x)) \label{omegax}
\end{eqnarray}
for every $X,\, Y\in\mathfrak{t}$ and every $x\in M$.  
\label{constlem}
\end{prop}

\begin{proof}
Let $\mathcal{X}$ and $\mathcal{Y}$ be smooth 
vector fields on $M$ satisfying ${\rm L}_{\mathcal{X}}\omega =0$ and ${\rm L}_{\mathcal{Y}}\omega =0$. Then by the 
homotopy identity (\ref{Lv}) 
\begin{eqnarray} 
{\rm i}_{[\mathcal{X},\, \mathcal{Y}]}\omega ={\rm L}_{\mathcal{X}}({\rm i}_{\mathcal{Y}}\omega )
={\rm i}_{\mathcal{X}}({\rm d}({\rm i}_{\mathcal{Y}}\,\omega))+{\rm d}({\rm i}_{\mathcal{X}}({\rm i}_{\mathcal{Y}}\omega))=\, -{\rm d}(\omega (\mathcal{X},\, \mathcal{Y})). \label{XY}
\end{eqnarray} 
Here we have used that ${\rm L}_{\mathcal{X}}\omega =0$ in the first equality, 
the homotopy formula for the Lie derivative in the second equality, 
and the closedness of $\omega$, the homotopy identity and 
${\rm L}_v\omega=0$ in the third equality.  

Now take $\mathcal{X}=X_M,\,\mathcal{Y}=Y_M$ where $X,\, Y\in\mathfrak{t}$. Then from 
(\ref{xmym}) and (\ref{XY}) we have
$\mathcal{H}_{\omega(X_M,\, Y_M)}=0$ 
 and hence
${\rm d}(\omega(X_M,\, Y_M))=0,$ and the connectedness of $M$ implies
that  $x\mapsto\omega_x(X_M(x),\, Y_M(x))$ is constant. 
\end{proof}

\subsection{Benoist\--Ortega\--Ratiu symplectic normal form}
\label{sympltubes}

Let $(M,\,\omega)$ be a symplectic manifold endowed with a proper symplectic
action of a Lie group $G$.

Let $x \in M$, let $H:=G_x$,  let $\mathfrak{l}$ be the kernel of the fundamental
form $\omega^{\mathfrak{t}}$ (Proposition~\ref{constlem}),  let $\mathfrak{g}_M(x):={\rm T}_x(G\cdot x)$, and let 
$\alpha_x$ is as in Proposition~\ref{312}. In addition, let
$
\omega^{G/H}$ be the $G$\--invariant closed $2$\--form $(\alpha _x)^*\omega$
on $G/H$,  and 
let $\omega^W$ be the symplectic form  on 
$
W:=\mathfrak{g}_M(x)^{\omega _x}/(\mathfrak{g}_M(x)^{\omega _x}\cap\mathfrak{g}_M(x)),
$
defined as the restriction to $\mathfrak{g}_M(x)^{\omega _x}$ of 
$\omega _x$. 

The map  
$X+\mathfrak{h}\mapsto X_M(x)$ is a linear isomorphism from 
$\got{l}/\mathfrak{h}$ to $\mathfrak{g}_M(x)^{\omega _x}\cap\mathfrak{g}_M(x)$. 
The linearized action of $H$ on ${\rm T}_x M$ is symplectic and leaves 
$\mathfrak{g}_M(x)\simeq \mathfrak{g}/\mathfrak{h}$ invariant, acting on it 
via the adjoint representation. It also 
leaves $\mathfrak{g}_M(x)^{\omega_x}$ invariant and induces 
an action of $H$ on 
$(W,\,\omega^W)$ by symplectic linear transformations.  Let
$
E:=(\got{l}/\mathfrak{h})^*\times W,
$
on which $h\in H$ acts by sending $(\lambda ,\, w)$ 
to $((({\rm Ad}(h))^*)^{-1}(\lambda ),\, h\cdot w)$. 
Choose ${\rm Ad}H$\--invariant linear complements 
$\got{k}$ and $\got{c}$ of $\mathfrak{h}$ and $\got{l}$ in 
$\mathfrak{g}$, respectively.  
Let $X\mapsto X_{\got{l}}:\mathfrak{g}\to\got{l}$ and 
$X\mapsto X_{\mathfrak{h}} :\mathfrak{g}\to\mathfrak{h}$ denote the 
linear projection from $\mathfrak{g}$ onto $\got{l}$ and $\mathfrak{h}$ 
with kernel equal to $\got{c}$ and $\got{k}$, respectively.  
These projections are ${\rm Ad}H$\--equivariant.  
Define the smooth one\--form $\eta ^{\#}$ 
on $G\times E$ by 
$\eta ^{\#}_{(g,\, (\lambda ,\, w))}(({\rm d}_1{\rm L}_g)(X),\, 
(\delta\lambda ,\,\delta w))
:=\lambda (X_{\got{l}})+\frac{1}{2}\omega^W(w,\,\delta w+X_{\mathfrak{h}}\cdot w) \nonumber
$
for all $g\in G$, $\lambda\in (\got{l}/\mathfrak{h})^*$, $w\in W$, 
and all $X\in\mathfrak{g}$, $\delta\lambda\in 
(\got{l}/\mathfrak{h})^*$,  
$\delta w\in W$. 

Let $G\times _HE$ denote the orbit space of $G\times E$ 
for the proper and free action of $H$ on $G\times E$, 
where $h\in H$ acts on $G\times E$ by sending 
$(g,\, e)$ to $(g\, h^{-1},\, h\cdot e)$. 
The action of $G$ on $G\times _HE$ is induced 
by the translational action of
of $G$ on $G\times E$. 

Let $\pi :G\times _HE\to G/H$ 
be induced by  
$(g,\, e)\mapsto g:G\times E\to G$. Because $H$ acts on $E$ by means 
of linear transformations, this projection 
exhibits $G\times _HE$ as a $G$\--homogeneous vector bundle 
over the homogeneous space $G/H$, which fiber $E$ and 
structure group $H$.

\begin{prop}
If $\pi _H:G\times E\to G\times _HE$ denotes $H$\--orbit mapping, 
then there is a unique smooth one\--form $\eta$ on 
$G\times_HE$, such that $\eta ^{\#}={\pi _H}^*\,\eta$. 
\end{prop}

Endow $G \times_H E$ with the $2$\--form $\pi^*\omega^{G/H}+{\rm d}\eta$. This
$2$\--form is symplectic.

The following is the local normal form
of Benoist \cite[Prop. 1.9]{benoist} and 
Ortega and Ratiu \cite{ortegaratiu} for a general proper 
symplectic Lie group action.

\begin{theorem}[Benoist~\cite{benoist}, Ortega\--Ratiu~\cite{ortegaratiu}]  \label{pds}
There is an open $H$\--invariant 
neighborhood $E_0$ of the origin in $E$, an open $G$\--invariant neighborhood 
$U$ of $x$ in $M$, and a $G$\--equivariant symplectomorphism
$\Phi \colon (G \times_H E,\pi^*\omega^{G/H}+{\rm d}\eta) \to (U,\omega)$ such
that $\Phi(H \cdot (1,0))=x$.
\label{Gthm}
\end{theorem}

For Hamiltonian symplectic actions, these local models 
had been obtained before by Marle \cite{marle} 
and Guillemin and Sternberg \cite[Section~41]{gsst}.

The fundamental form $\omega^{\mathfrak{t}}$ in Proposition~\ref{constlem} is an essential ingredient in the study
of symplectic actions. In the case of Hamiltonian actions, it takes a very particular
form as we see from the next result (but it will also turn out to be essential in the case
of more general symplectic actions).

\begin{theorem}
Let $(M,\omega)$ be a compact connected symplectic manifold endowed with an 
effective symplectic action of an $n$\--dimensional torus $T$ 
on $(M,\,\omega)$. Then the following are equivalent: 
{\rm i)} The action of $T$ has a fixed point in $M$; {\rm ii)} $\omega^{\mathfrak{t}}=0$ and 
$M/T$ is homeomorphic to a convex polytope; 
{\rm iii)} $\omega^{\mathfrak{t}}=0$ and ${\rm H}^1(M/T,\,\R )=0$; 
{\rm iv)} The action of $T$ is Hamiltonian. 
\end{theorem}

The proof of Theorem~\ref{pds} uses Theorem~\ref{gs}, Proposition~\ref{constlem}, 
Theorem~\ref{pds},  and the theory of $V$\--parallel
spaces (\cite[Section~10]{DuPe}),  see \cite[Corollary 3.9]{DuPe}. 

\subsection{Symplectic orbit types}

While Hamiltonian actions of maximal dimension appear as symmetries in 
many integrable systems in mechanics, non\--Hamiltonian actions also occur in physics, 
see eg. Novikov~\cite{novikov}. 

In the remaining of this paper we will give classifications of 
symplectic actions of tori (in the spirit of the Delzant classification Theorem~\ref{delzant}) in two cases: 
maximal symplectic actions, and coisotropic actions.  

We already described Hamiltonian actions 
in the previous sections in the case of $\dim M = 2\dim T$, and these are a special case of coisotropic actions, as we will see.

\begin{definition}
A submanifold 
$C$ of the symplectic manifold  $(M,\omega)$ is \emph{symplectic} if the restriction $\omega|_C:=i^*\omega$ 
of the symplectic form $\omega$ to $C$, where $i \colon C \to M$ is the inclusion map, 
is a symplectic form. \end{definition}
\medskip

\begin{definition}
A \emph{maximal symplectic action} is a symplectic action on compact symplectic manifold
endowed with an effective symplectic  action of a torus $T$ with some symplectic $T$\--orbit 
of maximal dimension $\dim T$.
\end{definition}

\begin{lemma} \label{plo}
Let $(M,\omega)$ be a symplectic manifold endowed with a symplectic $T$\--action.
If there is a symplectic $\dim T$\--orbit then every $T$\--orbit is symplectic and $\dim T$\--dimensional.
\end{lemma}

\begin{proof}
Assume that there is a symplectic $\dim T$\--orbit. Since the fundamental form  
$\omega^{\mathfrak{t}}$ (Proposition~\ref{constlem}) is point independent,
it is non\--degenerate. Hence ${\rm ker}(\omega^{\mathfrak{t}})=0$. 
\end{proof}

\begin{prop}
A maximal symplectic action does not admit a momentum map, and hence it is
not Hamiltonian.
\end{prop}

\begin{proof}
The $T$\--orbits of a Hamiltonian action are isotropic submanifolds, and hence not
symplectic as it is the case for maximal symplectic actions.
\end{proof}

Recall that if
$V$ is a subspace of a symplectic vector space $(W, \, \sigma)$, its
\emph{symplectic orthogonal complement}
$
V^{\sigma}
$
consists of the vectors
$w \in W$ such that $\sigma(w,\,v)=0$ for all $v \in V$. 

\begin{definition}
A submanifold $C$ of a symplectic manifold $(M,\omega)$ is  \emph{coisotropic} if 
for every $x \in C$ we have that $(\textup{T}_xC)^{\omega_x} \subseteq \textup{T}_xC$.
\end{definition}

\begin{prop}
If $C$ is a coisotropic $k$\--dimensional submanifold of a $2n$\--dimensional 
symplectic manifold $(M,\omega)$ then $k \geq n$.
\end{prop}

\begin{proof}
If $C$ is a coisotropic submanifold of dimension $k$, 
then 
$
2n-k=\dim ({\rm T}_x C)^{\omega _x}\leq
\dim ({\rm T}_x C)=k
$
shows that $k\geq n$. 
\end{proof}

The submanifold $C$ has the minimal dimension $n$ 
if and only if $({\rm T}_x C)^{\omega _x}={\rm T}_x C,$
if and only if $C$ is a Lagrangian submanifold of $M$ (Definition~\ref{maslov}).

\medskip

A Lagrangian submanifold is simultaneously a special case of coisotropic and isotropic submanifold. 
In fact, a submanifold is Lagrangian if and only if it is isotropic and coisotropic. Hence a symplectic torus action
for which one can show that it has a Lagrangian orbit
falls into the category of coisotropic actions. Therefore
it also includes Hamiltonian actions of $n$\--dimensional tori 
which are described in Section~\ref{ham:sec}, because the preimage of any point in the
interior of the momentum polytope is a Lagrangian orbit.

\begin{definition}
Let $T$ be a torus. A symplectic $T$\--action on a compact connected symplectic manifold $(M,\omega)$
endowed with a symplectic  $T$\--action with coisotropic orbits is called a \emph{coisotropic action}.  
\end{definition}

\medskip

The following is a consequence of Theorem~\ref{pds}.  In the following recall that $\mathfrak{l}$ denotes the kernel of the fundamental
form $\omega^{\mathfrak{t}}$ (Proposition~\ref{ffg}).

\begin{lemma}
Let $(M,\,\omega )$ be a compact connected symplectic manifold, 
and $T$ a torus which acts effectively and 
symplectically on $(M,\,\omega)$. Suppose that there exists a coisotropic principal $T$\--orbit. 
Then every coisotropic $T$\--orbit is a principal orbit and   
$\dim M=\dim T+ \dim \got{l}$. 
\label{coisotropicorbitlem}
\end{lemma}
\begin{proof}
We use Theorem \ref{Gthm} with $G=T$. Since $T$ is abelian, the adjoint action of 
$H=T_x$ 
on $\got{t}$ is trivial, which implies that the coadjoint 
action of $H$ on the component $(\got{l}/\mathfrak{h})^*$ is trivial.  
Let  $T\cdot x$ be a  coisotropic orbit.
Then $W$  is 
zero. This implies that the action of $H$ on $E=(\got{l}/\mathfrak{h})^*$ 
is trivial, and 
$T\times_HE=T\times_H(\got{l}/\mathfrak{h})^*$ is 
$T$\--equivariantly isomorphic to 
$(T/H)\times (\got{l}/\mathfrak{h})^*$, 
where $T$ acts by left multiplications on
the first factor. It follows that in the model 
all stabilizer subgroups are equal to $H$, and therefore 
$T_y=H$ for all $y$ in the 
$T$\--invariant open neighborhood $U$ of $x$ in $M$. 
Since $M_{\rm reg}$ is dense in $M$, 
there are $y\in U$ such that $T_y=\{ 1\}$, and therefore
$T_x=H=\{ 1\}$, so the orbit $T\cdot x$ is principal. 
The statement $\dim M=\dim T+\dim \got{l}$ also follows. 
\end{proof}

A similar argument to that in the proof of Lemma~\ref{coisotropicorbitlem} using 
Theorem~\ref{Gthm} (this time using 
the formula for the symplectic form in Theorem~\ref{pds}) shows:

\begin{prop}
There exists a coisotropic orbit if and only if
every principal orbit is coisotropic. 
\end{prop}

\begin{prop} The following hold:

{\rm (1)} The fundamental form 
$\omega ^{\mathfrak{t}}$ vanishes identically  
if and only if $\mathfrak{l}:={\rm ker}(\omega^{\mathfrak{t}})=\mathfrak{t}$  
if and only if some $T$\--orbit is isotropic 
if and only if  
every $T$\--orbit is isotropic.

{\rm (2)} Every principal orbit is Lagrangian 
if and only if some principal orbit 
is Lagrangian 
if and only if $\dim M=2\dim T$ and $\omega^\mathfrak{t}=0$. 
\label{isotropicorbitlem}
\end{prop}
\begin{proof}
The equivalence of $\omega ^{\mathfrak{t}}=0$ and 
$\mathfrak{l}=\mathfrak{t}$  is immediate. The
equivalence between $\omega^{\mathfrak{t}}=0$ and the isotropy 
of some $T$\--orbit follows from Proposition~\ref{constlem}.

If $x\in M_{{\rm reg}}$ and $T\cdot x$ is a Lagrangian
submanifold of $(M,\,\omega )$, then 
$\dim M=2\dim (T\cdot x)=2\dim T$, and 
$\omega^\mathfrak{t}=0$ follows in view of the first statement 
in the proposition. Conversely, if ~$\dim M=2\dim T$ and $\omega^\mathfrak{t}=0$, 
then every orbit is isotropic and for every 
$x\in M_{{\rm reg}}$ we have 
$\dim M=2\dim T=2\dim (T\cdot x)$, 
which implies that $T\cdot x$ is a Lagrangian submanifold 
of $(M,\,\omega)$. 
\end{proof}

In Guillemin and Sternberg \cite{multfree} we find the following notion.

\begin{definition}
A symplectic manifold with a 
Hamiltonian action of an arbitrary compact Lie group is called 
a {\em multiplicity\--free space} if the Poisson brackets of any pair of 
invariant smooth functions vanish. 
\end{definition}

\medskip

There is a relationship between coisotropic actions and multiplicity\--free spaces.

\begin{prop}
For a torus $T$ acting
on a closed connected symplectic manifold $(M,\omega)$,  the principal 
orbits are coisotropic if and only if $(M,\omega)$ is a multiplicity free space. 
\end{prop}

\begin{proof}
Let $f,\, g$ be in the set of $T$\--invariant smooth functions 
and let $x\in M_{\rm{reg}}$.   We will use use the notation 
$\mathfrak{t}_M(x):={\rm T}_x(T\cdot x).$
 Since $M_{\rm{reg}}$ is fibered by 
the $T$\--orbits, $\mathfrak{t}_M(x)$ is the common kernel of the 
${\rm d}f(x)$, where $f\in{\rm C}^{\infty}(M)^T$. Because $-{\rm d} f={\rm i}_{\mathcal{H}_f}\omega,$
it follows that $\mathfrak{t}_M(x)^{\omega _x}$ is the 
set of all $\mathcal{H}_f(x)$, $f\in{\rm C}^{\infty}(M)^T$.  So if the principal $T$\--orbits are coisotropic,  we
have that
$\mathcal{H}_f(x),\,\mathcal{H}_g(x) \in \mathfrak{t}_M(x)^{\omega _x}\cap\mathfrak{t}_M(x).$
 It follows that
$$
\{ f,\, g\}(x) :=\omega_x(\underbrace{\mathcal{H}_f(x)}_{\in  \mathfrak{t}_M(x)^{\omega _x}},\,\underbrace{\mathcal{H}_g(x)}_{\in \mathfrak{t}_M(x)})=0
$$
Since the principal orbit type $M_{\rm{reg}}$ is dense 
in $M$, we have that $\{ f,\, g\}\equiv 0$ for all 
$f,\, g\in{\rm C}^{\infty}(M)^T$ if the principal orbits 
are coisotropic. Conversely, if we have that $\{ f,\, g\} \equiv 0$ for all 
$f,\, g\in{\rm C}^{\infty}(M)^T$, then  $\mathfrak{t}_M(x)^{\omega _x}
\subset (\mathfrak{t}_M(x)^{\omega _x})^{\omega _x}=\mathfrak{t}_M(x)$
 for every  $x\in M_{\rm{reg}},$ i.e. $T\cdot x$ is coisotropic. 
\end{proof}

\subsection{Stabilizer subgroups}

 The general problem we want to treat is: 

\begin{problem}
Let $(M,\omega)$ be a closed connected symplectic manifold endowed with an effective
symplectic $T$\--action. Let $x \in M$. Characterize when $T_x$ is connected and when
it is discrete.
\end{problem}

We will prove the following result:

\begin{theorem} \label{abm}
Let $(M,\omega)$ be a symplectic manifold endowed with an effective symplectic $T$\--action and
let $x \in M$.
\begin{itemize}
\item[{\rm (i)}]
If the $T$\--action is coisotropic, then $T_x$ is connected.
\item[{\rm (ii)}]
If the $T$\--action is maximal symplectic, then $T_x$ is finite.
\end{itemize}
\end{theorem}

Let $\mathfrak{t}_x$ denote the Lie algebra of the stabilizer subgroup $T_x$ of the $T$\--action at $x$, 
which consists of  
of the  $X\in\mathfrak{t}$ such that $X_M(x)=0$. That is,
$\mathfrak{t}_x$ is the kernel of the linear mapping
$\alpha _x:X\mapsto X_M(x)$ from $\mathfrak{t}$ to ${\rm T}_x\! M$. 
 
\begin{lemma}[\cite{Pe}]
Let $(M,\omega)$ be a symplectic manifold endowed with a maximal symplectic $T$\--action. 
The stabilizer  $T_x$ is a finite abelian group for every $x \in M$.
\end{lemma}

\begin{proof}
 Since 
$\mathfrak{t}_x \subset \textup{ker}\, \omega^{\mathfrak{t}}$, by
Lemma~\ref{plo} we have that
$\mathfrak{t}_x$ is trivial. 
\end{proof}

Since every $T_x$ is finite,
it follows from  the tube theorem of Koszul (cf. \cite{koszul} or \cite[Theorem~2.4.1]{DuKo})
and the compactness of $M$ that there exists only 
finitely many different  stabilizer subgroups. 

The following  is statement (1) (a) in 
\cite[Lemma 6.7]{benoist}. It is a consequence of Theorem~\ref{pds}.
We use Theorem \ref{Gthm} with $G=T$, with $H$ acting
trivially on the factor $(\got{l}/\mathfrak{h})^*$ in 
$E=(\got{l}/\mathfrak{h})^*\times W$. 
Recall that $t\in T$ acts on $T\times _HE$ by 
sending $H\cdot (t',\, e)$ to $H\cdot (t\, t',\, e)$. 
When $t=h\in H$, 
$
H\cdot (h\, t',\, e)=H\cdot (h\, t'\, h^{-1},\, h\cdot e)=
H\cdot (t',\, h\cdot e)
$
since $T$ is abelian, and the action 
of $H$ on $T\times _HE$ is represented by the 
linear symplectic action 
of $H$ on $W$.  

\begin{lemma}[Benoist~\cite{benoist}]  \label{above2}
Let $(M,\omega)$ be a compact symplectic manifold endowed with a coisotropic
$T$\--action. For every $x\in M$, 
the stabilizer $T_x$ is connected. 
\end{lemma}

\begin{proof}
Since
$\dim M=(\dim T+\dim (\got{l}/\mathfrak{h})+\dim W)-\dim H 
$
and because the assumption that the principal orbits are 
coisotropic implies that $\dim M=\dim T+\dim \got{l}$, 
see Lemma \ref{coisotropicorbitlem}, 
it follows that $\dim W=2 \dim H$. 

Write $m=\dim H$.  
The action of  $H$ 
by means of symplectic linear transformations on 
$(W,\,\omega^W)$ leads to a direct sum decomposition 
of $W$ into $m$ pairwise $\omega ^W$\--orthogonal 
two\--dimensional $H$\--invariant linear subspaces 
$E_j$, $1\leq j\leq m$.  For $h\in H$ and every $1\leq j\leq m$, let $\iota _j(h)$ denote 
the restriction to $E_j\subset W\simeq \{ 0\}\times W
\subset (\got{l}/\mathfrak{h})^*\times W$ of the action of $h$ on $E$. 
Note that ${\rm det} \iota _j(h)=1$, because 
$\iota _j(h)$ preserves the restriction to $E_j\times E_j$ of 
$\omega^W$, which is an area form on $E_j$.  

Averaging any inner product 
in each $E_j$ over $H$, we obtain an $H$\--invariant inner 
product $\beta _j$ on $E_j$, and $\iota _j$ is a homomorphism of Lie 
groups from $H$ to ${\rm SO}(E_j,\,\beta _j)$, the group of 
linear transformations of 
$E_j$ which preserve both $\beta _j$ 
and the orientation. If $h\in H$ and $w\in W_{{\rm reg}}$, then 
$
h\cdot w=\sum _{j=1}^m\,\iota _j(h)\, w_j
$ where $w=\sum_{j=1}^m\, w_j,\quad w_j\in E_j$. Therefore $\iota _j(h)\, w_j=w_j$ for all $1\leq j\leq m$ 
implies that $h=1$. Hence  the map
$
\iota \colon H\to \prod_{j=1}^m\,{\rm SO}(E_j,\,\beta _j)
$
defined by
$
i(h)=(\iota _1(h),\,\ldots ,\,\iota _m(h))
$
is a Lie group isomorphism, so $H$ is connected. 
\end{proof}

It follows from Lemma~\ref{above2} that $T_x$ is a subtorus of $T$.

\section{Classifications of symplectic actions} \label{thecl}

\subsection{Maximal symplectic actions} \label{max}

This section describes the invariants of maximal symplectic actions. 
Using these invariants we will construct a model to which $(M,\omega)$ is $T$\--equivariantly symplectomorphic.  
Let $(M,\omega)$ be a compact connected symplectic manifold endowed with a maximal symplectic
$T$\--action of a torus $T$.

\subsubsection{Orbit space}

We denote by 
$
\pi \colon M \to M/T
$ 
the canonical projection $\pi(x):=T \cdot x$. The orbit space $M/T$ is endowed with the maximal topology for which $\pi$ is continuous (which is a Hausdorff topology). Since $M$ is compact and connected, $M/T$ is compact and
 connected. 
By the tube theorem (see for instance~\cite[Theorem~2.4.1]{DuKo}) if $x \in M$ there is a $T$\--invariant open neighborhood $U_x$ of 
 $T \cdot x$ and a $T$\--equivariant diffeomorphism  $\Phi_x \colon U_x \to T \times_{T_x} D_x,$
 where $D_x$ is an open disk centered at the origin in
 $\C^{k/2}$, $k:=\dim M-\dim T$. 
 In order to form the quotient, $h \in T_x$ acts on $T \times D_x$ by sending $(g,x)$ to $(gh^{-1},h\cdot x)$, where
 $T_x$ acts by linear transformations on $D_x$. The action of $T$ on $T \times_{T_x} D_x$ is induced by the translational action of $T$ on
the left factor of $T \times D_x$. 
The $T$\--equivariant diffeomorphism $\Phi_x$ induces a homeomorphism  
$D_x/T_x \to \pi(U_x)$, which we compose with the projection $D_x \to D_x/T_x$ to
get a map $\phi_x\colon D_x \to \pi(U_x)$. 
The proof of the following is routine.

\begin{prop}  \label{pl}
The collection $\{(\pi(U_x), \, D_x, \, \phi_x, \, T_x)\}_{x \in M}$ 
is an  atlas for  $M/T$.   
\end{prop}

\subsubsection{Flat connection}

Consider the symplectic form on $\C^m$ 
$
\omega ^{\C ^m}
:=\frac{1}{2{\rm i}}\sum_{j=1}^m\,
{\rm d}\overline{z^j}\wedge
{\rm d} z^j. 
$
We identify each tangent space to $T$ 
with $\mathfrak{t}$ and each tangent space of a vector space with the 
vector space itself. 
 The translational
action of $T$ on $T \times \C^m$ descends to an action of $T$ on $T \times_{T_x} \C^m$. 
Since the fundamental form (Proposition~\ref{constlem}) $\omega^{\mathfrak{t}} \colon \mathfrak{t} \times \mathfrak{t} \to \R$
is non\--degenerate, it determines a unique symplectic
form $\omega^T$ on $T$. The product symplectic form on $T \times \C^m$, denoted
by $\omega^{\mathfrak{t}}\oplus \omega^{\mathbb{C}^m}$ is defined pointwise at
$(t, \, z)$ and a pair of vectors $((X,\,u),\, (X',\,u'))$ by
$\omega ^{\mathfrak{t}}(X,\, X')+\omega^{\C ^m}(u,\,u').
$
The  form $\omega^T \oplus \omega^{\C^m}$ descends to a symplectic form on $T \times_{T_x}E_x$. 
Theorem~\ref{pds} gives us the following in the maximal symplectic case.

\begin{lemma} \label{tubetheorem2}
There is an open $\mathbb{T}^m$\--invariant 
neighborhood $E$ of $0$ in $\C ^m$,  
an open $T$\--invariant neighborhood $V_x$ of $x$ in $M$,
and a $T$\--equivariant symplectomorphism 
$
\Lambda_x \colon (T \times_{T_x}E,\omega^{\mathfrak{t}}\oplus \omega^{\mathbb{C}^m}) \to (V_x,\omega)
$
such that
$\Lambda_x([1,\, 0]_{T_x})=x$.
\end{lemma}

Lemma~\ref{tubetheorem2} implies the following essential result.

\begin{prop}
The collection $\Omega:=\{\Omega_x\}_{x \in M}$
where $\Omega_x:=({\rm T}_x(T \cdot x))^{\omega_x},$ is a smooth distribution on $M$ and 
$\pi \colon M \to M/T$ is a smooth principal $T$\--bundle
of which $\Omega$ is a $T$\--invariant flat connection.
\end{prop}

Let $\psi \colon \widetilde{M/T} \to M/T$ be the universal cover of 
$M/T$ based at $p_0=\pi(x_0)$. Let
$\mathcal{I}_x$ be the maximal integral manifold of $\Omega$. The inclusion  $i_x \colon \mathcal{I}_x \to M$
is an injective immersion and $\pi \circ i_x \colon \mathcal{I}_x \to M/T$
is an orbifold covering. Since $\widetilde{M/T}$ covers any covering of $M/T$, 
 it covers $\mathcal{I}_x$, which is a manifold. Because a covering
of a smooth manifold is  a smooth manifold, $\widetilde{M/T}$ is a smooth manifold.
Readers unfamiliar with orbifolds may  consult \cite[Section~9]{Pe}.

\subsubsection{Monodromy}

The universal cover $\widetilde{M/T}$ is a smooth manifold. Let $\pi^{\textup{orb}}_1(M/T, \, p_0)$
be the orbifold fundamental group, based at the same point as the universal cover.
The mapping $\pi^{\textup{orb}}_1(M/T, \, p_0) \times \widetilde{M/T} \to 
\widetilde{M/T}$ given by $([\lambda], \,[\gamma]) \mapsto [\gamma \, \lambda]$
is a smooth action of 
$\pi^{\textup{orb}}_1(M/T, \,p_0)$ on
$\widetilde{M/T}$, 
which is transitive on each fiber $\widetilde{M/T}_p$ of $\psi \colon \widetilde{M/T} \to M/T$. 
For any loop $\gamma \colon [0,\,1]\to M/T$ in  $M/T$
such that $\gamma(0)=p_0$, denote 
by $\lambda_{\gamma} \colon [0,\,1]\to M$ its unique horizontal lift with respect to the 
connection $\Omega$  such that 
$\lambda_{\gamma}(0)=x_0$,
where by horizontal we mean that ${\rm d}\lambda_{\gamma}(t)/{\rm d}t\in 
\Omega_{\lambda_{\gamma(t)}}$ for every $t \in [0,\, 1]$.
$\Omega$ is an orbifold flat connection, 
which means that $\lambda_{\gamma}(1)=\lambda_{\delta}(1)$ if $\delta$ is 
homotopy equivalent 
to $\gamma$ in the space of all orbifold paths in the orbit space
$M/T$ which start at $p_0$ and end at the given end point 
$p=\gamma(1)$. Therefore there exists unique group homomorphism 
$\mu \colon \pi^{\textup{orb}}_1(M/T, \, p_0) \to 
T$ such that $\lambda_{\gamma}(1)=\mu([\gamma]) \cdot x_0$. 

\begin{definition}
The homomorphism
$\mu$ does not depend
on the choice of the base point $x_0 \in M$; we call it the \emph{monodromy homomorphism of $\Omega$.}
\end{definition}

\subsubsection{Case $\dim T=\dim M-2$}

The case $\dim T=\dim M-2$ is the only case in which can give a complete classification,
thanks to the following.

\begin{lemma}[Thurston] \label{TT}
Given a positive integer $g$ and an $n$\--tuple $(o_k)_{k=1}^n$, $o_i \le o_{i+1}$
 of positive integers, 
there exists a compact, connected, boundaryless, orientable smooth orbisurface
$\mathcal{O}$ with underlying topological
space a compact, connected
surface of genus $g$ and $n$ cone points of respective orders $o_1,\dots,o_n$.
Secondly, let $\mathcal{O}, \, \mathcal{O}'$ be compact, connected, boundaryless, orientable smooth
orbisurfaces. Then $\mathcal{O}$ is diffeomorphic to $\mathcal{O}'$ if and only if the genera of their underlying surfaces are the same, and their 
associated increasingly ordered $n$\--tuples of orders of cone points are equal.
\end{lemma}

According to Lemma~\ref{TT},  $\textup{sig}(\mathcal{O}):=(g; \, \vec{o})$ topologically
classifies $\mathcal{O}$. We call  this tuple the \emph{Fuchsian signature of $\mathcal{O}$}.
If $(\mathcal{O},\, \omega)$
is a symplectic orbisurface, $\varint_{\mathcal{O}} \omega$
is the \emph{total symplectic area of $(\mathcal{O}, \,
\sigma)$}. 

It follows from Lemma~\ref{TT} and  the orbifold Moser's theorem \cite[Theorem~3.3]{MW}, 
that in the maximal symplectic case with $\dim  T=\dim  M-2$ the Fuchsian signature and
the symplectic area of $M/T$ determine $M/T$ up to symplectomorphisms.

\medskip

Let $(g; \, \vec{o}) \in \Z^{1+m}$ be the Fuchsian signature
of  $M/T$; let $\{\gamma_k\}_{k=1}^m$ be a basis of small loops around the cone points
$p_1,\ldots,p_n$ of $M/T$, viewed as an orbifold; let $\{\alpha_1, \, \beta_1,\ldots, \alpha_g,\beta_g\}$ be a symplectic basis of a \emph{free} subgroup 
$F$ of  
$
{\rm H}_1^{\textup{orb}}(M/T,
\, \Z)\!=\!\langle 
\{\alpha_i,\, \beta_i\}_{i=1}^g,  \, \{\gamma_k\}_{k=1}^m \, \, | \, \,
\sum_{k=1}^m \gamma_k=0, \, \, o_k \, \gamma_k=0, \, \, 1 \le k \le m
\rangle
$
whose direct sum with the \emph{torsion} subgroup of ${\rm H}_1^{\textup{orb}}(M/T,\mathbb{Z})$ is ${\rm H}_1^{\textup{orb}}(M/T, \, \Z)$.

Let $\mu_{\textup{h}}$ be the homomomorphism induced on homology by
$\mu$.

\begin{definition}
The \emph{monodromy invariant of the triple $(M, \omega, \, T)$} is the  
$\mathcal{G}_{(g; \, \vec{o})}$\--orbit given by
$
 \mathcal{G}_{(g, \, \vec{o})} \cdot ((\mu_{\textup{h}}(\alpha_i), \, \mu_{\textup{h}}(\beta_i))_{i=1}^{g}, \, (\mu_{\textup{h}}(\gamma_k))_{k=1}^m)  \in T^{2g+m}_{(g;\, \vec{o})}/\mathcal{G}_{(g, \, \vec{o})}.
 $ \end{definition}
 
 \medskip
Even though the monodromy invariant depends on choices, one
can show that it is  well\--defined.

\subsubsection{Classification}

The $T$\--action on $\widetilde{M/T} \times_{\Gamma}T$ 
 is the $T$\-- action inherited 
 from the action of $T$ by translations on the right factor of $\widetilde{M/T} \times T$.

One can show that there exists a unique $2$\--form $\nu$ on  $M/T$ such that
$\pi^* \nu|_{\Omega_x}=\omega|_{\Omega_x}$ for every $x \in M$.
Moreover, $\nu$ is a symplectic form. Therefore $(M/T,\, \nu)$
is a compact, connected, symplectic orbifold.  The symplectic form on $\widetilde{M/T}$ is the pullback by the 
covering map $\widetilde{M/T} \to M/T$ of  $\nu$ and the symplectic form on 
$T$ is the unique $T$\--invariant symplectic form determined by $\omega^{\mathfrak{t}}$.
The symplectic form on $\widetilde{M/T} \times T$ is the product symplectic form.

Let $\pi^{\textup{orb}}_1(M/T)$ act on  $\widetilde{M/T} 
\times T$ by the diagonal action $x \, (y,\,t)=(x \star y^{-1},\,\mu(x) \cdot t)$, where 
$\star \colon \pi^{\textup{orb}}_1(M/T) \times \widetilde{M/T} \to \widetilde{M/T}$ denotes the natural
action of $\pi^{\textup{orb}}_1(M/T)$ on $\widetilde{M/T}$. 

The symplectic form on $\widetilde{M/T}\times_{\pi^{\textup{orb}}_1(M/T)}T$ 
is induced on the quotient by the product form. The following is the model of maximal
symplectic actions. The $T$\--action on $\widetilde{M/T} \times_{\pi^{\textup{orb}}_1(M/T)}T$ is inherited from
the $T$\--action on the right factor of $\widetilde{M/T} \times T$.

\begin{theorem}[\cite{Pe}] \label{t3}
Let $(M,\omega)$ be a compact connected symplectic manifold endowed with
 a maximal symplectic $T$\--action. Then 
$M$ is $T$\--equivariantly symplectomorphic to  $\widetilde{M/T} \times_{\pi^{\textup{orb}}_1(M/T)}T$.
\end{theorem}

\begin{proof}
For any homotopy class 
$[\gamma]\in \widetilde{M/T}$ and $t\in T$, define  
$
\Phi([\gamma],\, t) := t\cdot\lambda_{\gamma}(1) \in M. 
$
The assignment $([\gamma],\,t) \mapsto \Phi([\gamma],\,t)$ defines a 
smooth covering $\Phi \colon \widetilde{M/T} \times T\to M$ between smooth manifolds. Let
$[\delta]\in\pi^{\textup{orb}}_1(M/T, \, p_0)$ act on $\widetilde{M/T} \times T$ by sending 
the pair $([\gamma],\, t)$ to
$([\gamma \, \delta^{-1}],\,  \mu([\delta])\, t)$. One can show that this action is free, and hence the associated bundle
$\widetilde{M/T}\times_{\pi^{\textup{orb}}_1(M/T, \, p_0)} T$
is a smooth manifold.  The mapping $\Phi$ induces a diffeomorphism $\phi$ from 
 $\widetilde{M/T}\times_{\pi^{\textup{orb}}_1(M/T, \, p_0)} T$ onto $M$.
By definition, $\phi$ intertwines the action of $T$ by translations on 
the right factor of 
$\widetilde{M/T} \times_{\pi^{\textup{orb}}_1(M/T, \, p_0)} T$ with the action of $T$ on $M$.
It follows
from the definition of the symplectic form 
on $\widetilde{M/T} \times_{\pi^{\textup{orb}}_1(M/T, \, p_0)} T$
that  $\phi$ is a $T$\--equivariant symplectomorphism.
\end{proof}

\begin{theorem}[\cite{Pe}] \label{t4}
Compact connected symplectic $2n$\--dimensional manifolds $(M,\omega)$ endowed
with a maximal symplectic $T$\--action with $\dim T=\dim M-2$ are classified up to $T$\--equivariant symplectomorphisms by: {\rm 1)} fundamental form $\omega^{\mathfrak{t}} \colon \mathfrak{t} \times \mathfrak{t} \to
\mathbb{R}$;  {\rm 2)}  Fuchsian signature $(g;\, \vec{o})$ of $M/T$;  
{\rm 3)}  symplectic area $\lambda$ of $M/T$;   
{\rm 4)}  monodromy of the connection $\Omega$ of  orthocomplements to the $T$\--orbits.

Moreover, for any list {\rm 1)\--4)} there exists a compact, connected
symplectic manifold with an effective symplectic $T$\--action of a torus $T$ of dimension $2n-2$ 
with $(2n-2)$\--dimensional 
symplectic $T$\--orbits 
whose list of invariants is precisely this one.
\end{theorem}

 {Theorem~\ref{t4} is extension of Theorem~\ref{delzant} to a class of symplectic actions
 which are never Hamiltonian. The first part of Theorem~\ref{t4} is uniqueness. The last part is an existence
result for which we have not provided details for simplicity; this would amount to say which
form $\omega^{\mathfrak{t}}$, etc. can appear. We shall say, however, that any antisymmetric
bilinear form can appear, essentially all tuples as in 2) (with very few exceptions), and any
$\lambda>0$ can appear in 3). Similarly for 4). Readers may consult \cite{Pe} for the precise list.

\subsection{Coisotropic actions} \label{prev} 

This section gives  invariants of coisotropic actions. Using these invariants
we construct a model of $(M,\omega)$ and the $T$\--action. Let $(M,\omega)$
be a compact connected symplectic manifold endowed with a coisotropic $T$\--action
of a torus $T$.

\subsubsection{Hamiltonian subaction}

We construct the maximal subtorus $T_{{\rm h}}$ of $T$ which acts
in a Hamiltonian fashion on $(M,\omega)$, and for which accordingly there is an associated
momentum map $\mu$ and an associated polytope $\Delta$ as described in
Section~\ref{ham:sec}. 

The group $T_{{\rm h}}$ is the product of all the different stabilizer subgroups of $T$, which 
is a subtorus of $T$.

Let $x \in M$ and let $m=\dim T_x$. Let  $K$  be a complementary subtorus of 
the subtorus $H:=T_x$ in $T$.  For $t\in T$, let $t_x$ and $t_K$ 
be the unique elements in $T_x$ and $K$, respectively, such that 
$t=t_x\, t_K.$ Recall that
 $\mathfrak{l}=\textup{ker}(\omega^{\mathfrak{t}})$, where $\omega^{\mathfrak{t}}$ is the
 fundamental form in Proposition~\ref{constlem}.   Let $X\mapsto X_{\got{l}}$ be a linear 
projection from $\got{t}$ onto $\got{l}$.

The $H$\--invariant inner product $\beta _j$ on $E_j$, 
introduced in the proof of Lemma \ref{above2},  
is unique, if we also require that the symplectic 
inner product of any orthonormal basis with respect to 
$\omega^W$ is equal to $\pm 1$. In turn this leads to 
the existence of a unique complex structure on 
$E_j$ such that, for any unit vector 
$e_j$ in $(E_j,\,\beta _j)$, we have that 
$e_j$, ${\rm i}\, e_j$ is an orthonormal basis 
in $(E_j,\,\beta _j )$ and $\omega^W(e_j,\,{\rm i}\, e_j)=1$. 
This leads to an identification of $E_j$ with $\C$, 
which is unique up to multiplication by an 
element of $S^1 :=\{ z\in\C\mid |z|=1\}$. 
 
In turn this leads to an identification of $W$ with $\C^ m$, 
with the symplectic form $\omega^W$ defined by 
$
\omega^{\C ^m}
= \frac{1}{2 {\rm i}}\sum_{j=1}^m\,{\rm d}\overline{z^j}\wedge
{\rm d}\! z^j.$  The element $c\in\T ^m$ acts 
on $\C ^m$ component wise
by $(c\cdot z)^j=c^j\, z^j$. There is a unique Lie group isomorphism 
$\iota :H\to\T ^m$ such that 
$h\in H$ acts on $W=\C ^m$ by sending 
$z\in\C ^m$ to $\iota (h)\cdot z$. 
The identification of $W$ with $\C ^m$ is unique up to 
a permutation of the coordinates and the action of an 
element of $\T ^m$. 

Let $T$ act on $K \times (\mathfrak{l}/\mathfrak{t}_x)^* \times \mathbb{C}^m$ 
by 
$
t \cdot (k,\,\lambda ,\, z)=(t_K\, k,\,\lambda ,\, \iota (t_x)\cdot z). \nonumber
$
endowed with the symplectic form given at a point $(k,\,\lambda ,\, z)$ and pair of vectors
$((X,\,\delta\lambda ,\,\delta z),\, (X',\,\delta '\lambda ,\,\delta 'z))$ by
$
\omega^{\got{t}}(X,\, X')
+\delta\lambda (X'_{\got{l}})-\delta '\lambda (X_{\got{l}})
+\omega^{\C ^m}(\delta z,\,\delta 'z). 
$

\begin{lemma} \label{pt}
There is an open $\T ^m$\--invariant 
neighborhood $V$ of $(0,0)$ in
$(\mathfrak{l}/\mathfrak{t}_x)^*\times \C ^m$,  an open $T$\--invariant neighborhood $U_x$ of $x$ in $M$,
and 
a $T$\--equivariant symplectomorphism
\begin{eqnarray} \label{nm}
\Phi \colon K\times V \to U_x
\end{eqnarray}
such that $\Phi (1,\, 0)=x$.\end{lemma}

 \begin{prop} \label{hc}
 The product of  all  stabilizers
is a subtorus of $T$, denoted by $T_{{\rm h}}$, and it acts on $M$ in a Hamiltonian fashion.
Furthermore,any complementary subtorus $T_{\rm{f}}$ to $T_{{\rm h}}$ in $T$ act freely on $M$.
\end{prop}
 
\begin{proof}
As a consequence of Lemma~\ref{pt} the stabilizer subgroup of the $T$\--action on $K\times (\mathfrak{l}/\mathfrak{t}_x)^*\times \C ^m$ at 
$(k,\,\lambda ,\, z)$ is
$
\Big\{t_x \in T_x\,\,|\,\, \iota (t_x)^j=1\,\, \forall j \,\,\, {\rm such\,\, that}\,\, z^j\neq 0\Big\}.
$
In view of  (\ref{nm}) there are  $2^m$ different stabilizer subgroups $T_y$, $y\in U$. 
Since $M$ is compact there 
are only finitely many different stabilizer subgroups of $T$. The product of all the different stabilizer subgroups 
is a subtorus of $T$ because  the product of 
finitely many subtori is a compact and connected subgroup of 
$T$, and therefore it is a subtorus of $T$.  

See~\cite[Corollary~3.11]{DuPe} for the second claim. 

The torus $T_{{\rm f}}$ acts freely 
on $M$, because if $x\in M$, then $T_x\subset T_{{\rm h}}$, hence 
$T_x\cap T_{\rm f}\subset T_{{\rm h}}
\cap T_{{\rm f}}=\{ 1\}$, which proves (iii).
\end{proof}

The 
tube theorem of Koszul~\cite{koszul} or \cite[Theorem~2.4.1]{DuKo} implies that there
is a finite number  of stabilizer subgroups, and hence (i) in Proposition~\ref{hc}, but we wanted to derive this from 
Lemma~\ref{pt} because it is an essential result in the study of coisotropic actions.

\begin{prop} \label{pra}
$T_{{\rm h}}$ is 
the unique maximal stabilizer subgroup of $T$.  
\end{prop}

\begin{proof}
Since any Hamiltonian torus action has fixed points, it follows 
from Proposition~\ref{hc} that there exist $x\in M$ such that 
$T_{{\rm h}}\subset T_x$, hence $T_{{\rm h}}=T_x$ because 
the definition of $T_{{\rm h}}$  implies that 
$T_x\subset T_{{\rm h}}$ 
for every $x\in M$.
\end{proof}
 
Let 
\begin{eqnarray} \label{mom}
\mu :M\to\Delta\subset 
{\mathfrak{t}_{{\rm h}}}^*
\end{eqnarray}
 be  the momentum map of 
$T_{{\rm h}}$\--action. 
Its fixed points are the $x\in M$ such that $\mu (x)$ is a vertex of 
the Delzant polytope $\Delta$.

\subsubsection{Orbit space} \label{361}

{By Leibniz identity
for the Lie derivative one can show that for each $X\in\mathfrak{l}$,
$
\widehat{\omega}(X) := \, -{\rm i}_{X_M}\omega
$ 
is a closed 
basic one\--form on $M$ and $X\mapsto\widehat{\omega} (X)_x$ is an $\mathfrak{l}^*$\--valued 
linear form on ${\rm T}_x M$, which we denote by 
$\widehat{\omega}_x$. Hence $x\mapsto\widehat{\omega}_x$ 
is a basic closed $\mathfrak{l}^*$\--valued one\--form on $M$, which we 
denote by $\widehat{\omega}$ and
$
\widehat\omega _x(v)(X)
=\omega _x(v,\, X_M(x)),
\quad x\in M,\; v\in{\rm T}_x M,\; X\in\mathfrak{l}.
$

In Lemma~\ref{pt}
with $x\in M_{\rm{reg}}$, 
where $\mathfrak{t}_x=\{ 0\}$ and $m=0$, 
at each point  $\widehat{\omega}$ is
$(\delta t,\,\delta\lambda )\mapsto\delta\lambda :
\mathfrak{t}\times\mathfrak{l}^*\to\mathfrak{l}^*$.

  \begin{prop} 
For every $p\in (M/T)_{\rm{reg}}$ the induced  map
$
\widehat{\omega}_p: {\rm T}_p(M/T)_{\rm{reg}}\to\mathfrak{l}^*
$
is a linear isomorphism. 
\end{prop}

Therefore $\zeta \in \mathfrak{l}^*$ acts on $p \in (M/T)_{\rm{reg}}$ by traveling for time
$1$ from $p$ in the direction that $\zeta$ points to. We denote the arrival point
by $p+\zeta$. This action is 
not  defined on $(M/T)\setminus (M/T)_{\rm{reg}}$, it is only defined
in the directions of vectors which as linear forms vanish on the stabilizer subgroup
of the preimage under $\pi \colon M \to M/T$. 

Since $T_{{\rm h}}$ is the maximal stabilizer
subgroup (Proposition~\ref{pra}), and for each $x$, $T_x \subset T_{{\rm h}}$, the additive subgroup
\begin{eqnarray} \label{N}
N:=(\mathfrak{l}/\mathfrak{t}_{{\rm h}})^*
\end{eqnarray}
viewed as the set of linear forms on
$\mathfrak{l}$ which vanish on the Lie algebra $\mathfrak{t}_{{\rm h}}$ of $T_{\rm h}$, is the maximal subgroup
of $\mathfrak{l}^*$ which  acts on $M/T$. 
This turns $M/T$ into  
a {\em $\mathfrak{l}^*$\--parallel space}, intuitively a space modeled on $\mathfrak{l}^*$.
In~\cite[Section ~11]{DuPe} it is proved that they are  isomorphic to  the product of a closed convex set and a torus. 
In the case of $M/T$, the convex polytope is Delzant (Definition~\ref{dpol}), and equal to $\Delta$ in (\ref{mom}). 

\begin{prop}
 If $P$ is the {period lattice} of a $N$\--action 
on  $M/T$, the quotient Lie group $N/P$ is a torus, and $M/T$ is isomorphic to $\Delta \times (N/P)$.
\end{prop}

An analysis of the singularities of $M/T$ allows one to define the structure of
$\mathfrak{l}^*$\--parallel space.   Any $\xi\in\mathfrak{l}^*$ may be viewed as a constant vector 
field on $(M/T)_{\rm{reg}}$.

\subsubsection{Singular connection}

If $\mathfrak{t}_{{\rm f}}$ denotes the Lie algebra of 
$T_{{\rm f}}$ in Propositon~\ref{hc}, then 
$\mathfrak{t}=\mathfrak{t}_{{\rm h}}\oplus\mathfrak{t}_{{\rm f}}$. 
Each linear form on ${\mathfrak{t}_{{\rm h}}}^*$ has a unique extension 
to a linear form on $\mathfrak{l}$ which is  zero 
on $\mathfrak{t}_{{\rm f}}$. This leads to 
an isomorphism of ${\mathfrak{t}_{{\rm h}}}^*$ 
with the subspace 
$(\mathfrak{l}/\mathfrak{l}\cap\mathfrak{t}_{{\rm f}})^*$ of 
$\mathfrak{l}^*$. This isomorphism depends on the choice of 
$T_{{\rm f}}$. Since
$\mathfrak{l}=\mathfrak{t}_{{\rm h}}\oplus 
(\mathfrak{l}\cap\mathfrak{t}_{{\rm f}})$, $
\mathfrak{l}^*=(\mathfrak{l}/\mathfrak{l}\cap\mathfrak{t}_{{\rm f}})^*\oplus 
(\mathfrak{l}/\mathfrak{t}_{{\rm h}})^*$. Let
\begin{equation}
\mu :M\to \Delta\subset 
(\mathfrak{l}/\mathfrak{l}\cap\mathfrak{t}_{{\rm f}})^*
\simeq {\mathfrak{t}_{{\rm h}}}^*
 \label{mumu}
\end{equation}
be
$\pi :M\to M/T$ followed by the projection 
$M/T\simeq\Delta\times (N/P) \to \Delta$. 
The map $\mu :M\to 
{\mathfrak{t}_{{\rm h}}}^*$ is a momentum mapping for the 
Hamiltonian $T_{{\rm h}}$\--action on $M$ in Proposition~\ref{hc} and coincides with (\ref{mom}).
The composite of $\pi :M\to M/T$ with the projection from 
$M/T\simeq\Delta\times (N/P) \to N/P$ is torus\--valued 
generalization of the $S^1$\--momentum map of McDuff~\cite{MD}.

\begin{definition}
 $L_{\zeta}  \in \mathcal{X}^{\infty}(M_{\textup{reg}})$ is a {\em lift of $\zeta$}
if  ${\rm d}_x\pi (L_{\zeta}(x))=\zeta$ for all $x\in M_{\rm{reg}}$.
\end{definition}
 
 \medskip

The word ``lift" is used above the sense $\zeta \in \mathfrak{l}^*$ is as a constant vector field on $(M/T)_{{\rm reg}}$ of which $L_{\zeta}$ is a lift.  
Linear assignments of lifts $\zeta \in \mathfrak{l}^* \mapsto L_{\zeta}$ depending linearly on $\zeta$
and connections for the principal torus bundle $M_{\rm{reg}} \to M_{\rm{reg}}/T$
are equivalent.  
The essential ingredient for construction of the model of $(M,\omega)$ with $T$\--action 
is the existence of the following connection:

\begin{theorem} \label{ppp}
 There is an antisymmetric bilinear map $c:N\times N\to\got{l}$ 
satisfying
$
\zeta (c(\zeta ',\,\zeta ''))
+\zeta '(c(\zeta '',\,\zeta ))
+\zeta ''(c(\zeta ,\,\zeta '))=0
$
for every $\zeta ,\,\zeta ',\,\zeta ''\in 
N$, and  a connection
\begin{eqnarray} \label{con}
\zeta \in \mathfrak{l}^* \mapsto L_{\zeta} \in \mathcal{X}^{\infty}(M_{\textup{reg}}),
\end{eqnarray}
 whose Lie brackets satisfy 
  \begin{eqnarray} \label{chern}
[L_{\zeta}, \,L_{\eta}]=c(\zeta,\,\eta)_M,\,\,\,\,\,\, \forall \zeta,\,\eta \in N,
\end{eqnarray}
and $[L_{\zeta}, \,L_{\eta}]=0$ otherwise,
as well as
$\omega _x(L_{\zeta}(x),\, L_{\eta }(x))
=-\mu (x)(c_{{\rm h}}(\zeta,\eta))\,\,\,\,\,\,\,  \forall \zeta ,\,\eta \in N,\,\,\, \forall x \in M,
$
where $\mu$ is the momentum map of the $T_{\rm h}$ action in {\rm (\ref{mumu})}, 
$c_{{\rm h}}(\zeta ,\eta)$ denote
the $\mathfrak{t}_{{\rm h}}$\--component of $c(\zeta ,\,\eta)$ 
in $\mathfrak{l}=\mathfrak{t}_{{\rm h}}
\oplus (\mathfrak{l}\cap\mathfrak{t}_{{\rm f}})$.  
 and 
$\omega( L_{\zeta},\, L_{\eta})=0$ otherwise.
\end{theorem}

The construction of (\ref{con}) is the most involved part of~\cite{DuPe} (Proposition 5.5 therein). It is a singular
connection which blows up at $M \setminus M_{\textup{reg}}$. The map $c$ has a  geometric interpretation, which we discuss next.

\subsubsection{Chern class}

There is an isomorphism 
$
(M/T)_{\rm{reg}}\simeq 
\Delta ^{\rm{int}}\times (N/P), 
$
induced by the isomorphism  $M/T \simeq \Delta \times (N/P)$.
Any connection for $T$\--bundle  $M_{\rm{reg}} \to M_{\rm{reg}}/T$ has 
a {\em curvature form}, a smooth $\mathfrak{t}$\--valued 
two\--form on $M_{\rm{reg}}/T$. Its cohomology class of 
this curvature form is an element of
${\rm H}^2(M_{\rm{reg}}/T,\,\mathfrak{t})$, which is independent of the 
choice of the connection.  The $N$\--action on $M/T$ leaves
$M_{\rm{reg}}/T\simeq (M/T)_{\rm{reg}}$ 
invariant, with orbits isomorphic to $N/P$. 

The 
pull\--back to the $N$\--orbits defines an isomorphism 
$
{\rm H}^2(M_{\rm{reg}}/T,\,\mathfrak{t}) \to {\rm H}^2(N/P,\,\mathfrak{t}), 
$
which is identified 
with $(\Lambda ^2N^*)\otimes\mathfrak{t}$ (this observation goes back to  \'Elie Cartan).

It follows from the construction of the connection (\ref{con}) 
 that  
$c \colon N \times N \to \mathfrak{l}$, viewed as an element in
$
c\in (\Lambda ^2N^*)\otimes\mathfrak{l}\subset(\Lambda ^2N^*)\otimes\mathfrak{t}
$ 
equals the negative of the pull\--back to an $N$\--orbit of the 
cohomology class of the curvature form. Hence
$c:N\times N\to\mathfrak{l}$ in (\ref{chern}) is independent of  $T_{\rm{f}}$.  
The Chern class $\mathcal{C}$ of the principal $T$\--bundle $\pi :M_{\rm{reg}}\to M_{\rm{reg}}/T$
is an element of ${\rm H}^2(M_{\rm{reg}}/T,\, T_{\Z})$.
It is known that the image 
of $\mathcal{C}$ in ${\rm H}^2(M_{\rm{reg}}/T,\,\mathfrak{t})$ 
under the coefficient homomorphism 
${\rm H}^2(M_{\rm{reg}}/T,\, T_{\Z})
\to{\rm H}^2(M_{\rm{reg}}/T,\,\mathfrak{t})$ 
is equal to the negative of the cohomology 
class of the curvature form of any connection in the principal 
$T$\--bundle, and hence we have the following.

\begin{prop}
The map $c \colon N \times N \to \mathfrak{l}$ represents the Chern class $\mathcal{C}$.
\end{prop}

\subsubsection{Toric foliation} \label{mh}

 Next we
describe a foliation of $M$ by symplectic\--toric manifolds as in Section~\ref{dp}.
Let
$
D_x:=\textup{span}\{\, L_{\eta}(x),\,\,Y_M(x) \,\,\,| \,\,\, Y \in \mathfrak{t}_{\textup{h}}, \,\, \eta \in C\,\},\,\,\, 
x \in M_{\rm{reg}}, 
$
where $C \oplus N = \mathfrak{l}^*$
and let 
$
\mathcal{D}:=\{D_x \,\,\, |\,\,\, x\in M\}.
$

\begin{prop} \label{p}
The distribution $\mathcal{D}:=\{D_x\}_{x \in M}$  is smooth, integrable and $T$\--invariant and the integral manifolds  of $\mathcal{D}$ are $(2\,\dim  T_{{\rm h}})$\--dimensional
symplectic manifolds and $T_{\rm h}$\--equivariantly
symplectomorphic to each other.
\end{prop}

We pick an integral manifold of  Proposition~\ref{p} and call it $M_{{\rm h}}$.  Then  
$\omega$ restricts to a symplectic form $\omega_{{\rm h}}$
on $M_{{\rm h}}$ and $T_{{\rm h}}$ acts Hamiltonianly 
on it.

\subsubsection{Group extensions} \label{Gdef}

 $N$ in (\ref{N}) is the maximal subgroup of $\mathfrak{l}^*$ 
which acts on $M/T$. Denote the flow after time $t\in\R$ of 
$v\in \mathcal{X}^{\infty}(M)$ by ${\rm e}^{t\, v}$. This  defines
a map $v\mapsto {\rm e}^v$, $\mathcal{X}^{\infty}(M) \to {\rm Diff}^{\infty}(M)$, 
analogous to  (\ref{emap}).  

\begin{definition}
The \emph{extension of $N$ by $T$} is the Lie group
$
G:= T \times N 
$
with operation
\begin{eqnarray} \label{nonabgroup}
(t,\,\zeta )\, (t',\,\eta)=
(t\, t'\,{\rm e}^{-c(\zeta ,\,\eta)/2},\,\zeta +\eta).
\end{eqnarray}
\end{definition}

\medskip

\begin{prop}
The Lie group $G$ acts smoothly on 
$M$ by
$
(t,\,\zeta )\mapsto t_M\circ {\rm e}^{L_{\zeta}},
$
where we are using the identification  $G
\simeq (\mathfrak{t}/T_{\Z})\times N$. 
\end{prop}

The projection $\pi :M\to M/T$  intertwines the action of $G$ on $M$ with the 
action of $N$ on $M/T$ and there is an exact sequence 
$1\to T\to G\to N\to 1,$ 
where $G \to N$ corresponds to passing from 
the action of $G$ on $M$ to the action of $N$ on $M/T$, on which the action of 
$T$ is trivial.

\begin{prop}
The Lie algebra of $G$ with {\rm (\ref{nonabgroup})} is the two\--step nilpotent
Lie algebra $\mathfrak{g}=\mathfrak{t}\times N$ with 
$
[(X,\,\zeta ),\, (X',\,\eta)]=\, -(c(\zeta ,\,\eta),\, 0)$.   
The product $\mathfrak{t}\times N$ endowed with the operation
$
(X,\,\zeta )\, (X',\,\eta)=
(X+X'-c(\zeta ,\,\eta)/2,\,\zeta +\eta)$
is a two\--step nilpotent Lie group with Lie algebra ${\got g}$, 
and the identity as the exponential map. 
\end{prop}

\subsubsection{Holonomy}

For $\zeta\in P$ and $p\in M/T$ consider the loop
$\gamma_{\zeta}(t):=p+t\,\zeta$. If 
$p=\pi (x)$, then 
$
\delta (t)={\rm e}^{t\, L_{\zeta}}(x), 
$
$0\leq t\leq 1$ is  the \emph{horizontal lift of $\gamma _{\zeta}$} 
which starts at $x$ because $\delta (0)=x$ and 
$\delta '(t)=L_{\zeta}(\delta (t))$ is a horizontal 
tangent vector mapped by ${\rm d}_{\delta (t)}\!\pi$ 
to  $\zeta$. Hence
$\pi (\delta (t))=\gamma _{\zeta}(t)$ for all $0\leq t\leq 1$.

\begin{definition}
The element of $T$ which maps 
$\delta (0)=x$ to 
$\delta (1)$ is called the {\em holonomy} 
$\tau _{\zeta}(x)$ of the 
loop $\gamma _{\zeta}$ at $x$
with respect to the 
connection (\ref{con}).
\end{definition}

\medskip

 Because $\delta (1)
={\rm e}^{L_{\zeta}}(x)$, we have that
$\tau _{\zeta}(x)\cdot x={\rm e}^{L_{\zeta}}(x).$
The  element $\tau _{\zeta}(x)$  does depend
on the point $x\in M$, on the period $\zeta\in P$,
and on the choice of connection (\ref{con}). Let $H=
\{(t,\,\zeta )\in G \,\,\, |\,\,\, \zeta\in P,\,\,\, t\,\tau _{\zeta}\in T_{{\rm h}} \}. 
$
The elements $\tau_{\zeta} \in T$, $\zeta \in P$, encode the holonomy of (\ref{con}).
So the holonomy is an element of the set
 ${\rm Hom}_c(P,\, T)$  of maps $\tau \colon P \to T$, denoted by
$\zeta \mapsto \tau_{\zeta}$, such that 
$
\tau_{\zeta}\tau_{\eta}=\tau_{\zeta+\eta}
{\rm e}^{c(\eta,\,\zeta)/2}.
$ 
There is a Lie subgroup  
$
B \leq {\rm Hom}_c(P,\, T)
$
which eliminates the dependance on the choice of connection and base point,
so the true holonomy invariant of $(M,\omega)$ is an element of $ {\rm Hom}_c(P,\, T)/B$.
The precise definition of $B$ is technical and appeared in \cite{DuPe}.

\subsubsection{Nilmanifolds}

The quotient $G/H$ is with respect to the non standard group structure in expression (\ref{nonabgroup}). 
On $G/H$ we still have the free action of the torus $T/T_{\rm h}$, 
which exhibits $G/H$ as a principal $T/T_{\rm h}$\--bundle 
over the torus $(G/H)/T\simeq N/P$. Palais and Stewart \cite{ps} 
showed that every principal torus bundle over a torus is diffeomorphic to 
a nilmanifold for a two\--step nilpotent Lie group. 
When the nilpotent Lie group is not abelian, then the manifold 
$M$ does not admit  a K\"ahler structure, cf. Benson and Gordon \cite{bg}.  
Next we  give a description of $G/H$. 

\begin{prop}
The $G$\--space $G/H$ is isomorphic 
to the quotient of the simply connected two\--step 
nilpotent Lie group $(\mathfrak{t}/\mathfrak{t}_{\rm h})\times N$ 
by the discrete subgroup of elements $(Z,\,\zeta )$ such that 
${\rm e}^Z\,\tau _{\zeta}\in T_{\rm h}$.
\end{prop}

\begin{proof}
The identity component $H^o=T_{\rm h}\times\{ 0\}$ 
of $H$ is a closed normal Lie subgroup 
of both $G$ and $H$. The mapping $(G/H^o)/(H/H^o) \to G/H$
given by $(g\, H^o)\, (H/H^o)\mapsto g\, H$ is a $G$\--equivariant 
diffeomorphism. The structure in $G/H^o=(T/T_{\rm h})\times N$ is 
$
(t,\,\zeta)\, (t',\,\eta)
=(t\, t'\,{\rm e}^{-c_{\mathfrak{l}/\mathfrak{t}_{\rm h}}
(\zeta ,\,\eta)/2},\,\zeta +\eta), 
\quad t,\, t'\in T/T_{\rm h},\quad\zeta ,\,\eta\in N, 
$
where $c_{\mathfrak{l}/\mathfrak{t}_{\rm h}}
:N\times N\to\mathfrak{l}/\mathfrak{t}_{\rm h}$ 
is the composite of $c:N\times N\to\mathfrak{l}$ and
the projection $\mathfrak{l}\to\mathfrak{l}/\mathfrak{t}_{\rm h}$. 
Hence $G/H^o$ is a two\--step nilpotent Lie group 
with universal covering $(\mathfrak{t}/\mathfrak{t}_{\rm h})\times N$ 
and covering group $(T/T_{\rm h})_{\Z}\simeq 
T_{\Z}/(T_{\rm h})_{\Z}$. Also 
$P \to H/H^o$ given by  $\iota :\zeta\mapsto ({\tau _{\zeta}}^{-1},\,\zeta )\, H^o$ 
is an isomorphism. 
\end{proof}

\subsubsection{Classification} \label{t1sec}

Let $h\in H$ act on $G\times M_{\rm h}$ by 
$(g,\, x) \mapsto (g\, h^{-1},\, h\cdot x)$ and consider 
$G \times_H M_{\rm h}$.

The $T$\--action by translations on the left factor of $G$ passes to an action on 
$G \times_{H} M_{{\rm h}}$. 
Each of the fibers of $G \times_H M_{\rm h}$ is identified with the 
symplectic\--toric manifold $(M_{\rm h},\omega_{\rm h}, T_{\rm h})$. Any
complementary subtorus $T_{\rm f}$ permutes the fibers of $G \times_H M_{\rm h} \to G/H$,
each of which is identified with  $(M_{\rm h},\omega_{\rm h}, T_{\rm h})$.

Next let us explicitly construct a symplectic form on $G \times_{H} M_{{\rm h}}$. This construction
uses in an essential way Lemma~\ref{pt} but for simplicity here we skip the details
as the general formula may be given directly. 
Let $\delta a=((\delta t,\,\delta\zeta ),\,\delta x)$,
and $\delta' a=((\delta 't,\,\delta '\zeta ),\,\delta ' x)$ 
be tangent vectors to the product $G\times M_{{\rm h}}$ at 
the point $a=((t,\,\zeta ),\, x)$, 
where we identify each tangent space of  $T$ with 
$\mathfrak{t}$. Write 
$X=\delta t+c(\delta\zeta ,\,\zeta )/2$
and 
$
X'=\delta 't+c(\delta '\zeta ,\,\zeta )/2.
$
Let $X_{{\rm h}}$ be
$\mathfrak{t}_{{\rm h}}$\--component of $X\in\mathfrak{t}$ in $\mathfrak{t}_{{\rm h}}\oplus 
\mathfrak{t}_{\rm{f}}$, and  similarly for $X_{\mathfrak{l}}$ and define
\begin{eqnarray}
\Omega _a(\delta a,\,\delta 'a)
&=&\omega ^{\mathfrak{t}}(\delta t,\,\delta 't)
+\delta\zeta ({X'}_{\mathfrak{l}})
-\delta '\zeta (X_{\mathfrak{l}})
-\mu (x)(c_{{\rm h}}(\delta\zeta ,\,\delta '\zeta ))
\nonumber\\
&&+\, (\omega _{{\rm h}})_x(\delta x,\, 
({X'}_{{\rm h}})_{M_{{\rm h}}}(x))
-(\omega _{{\rm h}})_x(\delta 'x,\, 
(X_{{\rm h}})_{M_{{\rm h}}}(x))
+\, (\omega _{{\rm h}})_x(\delta x,\,\delta 'x).
\label{A*sigma} 
\end{eqnarray}

If $\pi _{M}$ is the projection 
$G\times M_{{\rm h}} \to G\times _HM_{{\rm h}}$, 
 the $T$\--invariant symplectic form 
on $G\times _HM_{{\rm h}}$
is the unique two\--form $\beta$ on 
$G\times _HM_{{\rm h}}$ such that 
$\Omega={\pi _{M}}^*\,\beta$.

We are ready to state the model theorem.

\begin{theorem}[\cite{DuPe}] \label{t1}
Let $(M,\omega)$ be a compact connected symplectic manifold endowed with a
coisotropic $T$\--action. Then $(M,\omega)$ is $T$\--equivariantly symplectomorphic to the total space  $G \times_H M_{{\rm h}}$ of the symplectic fibration 
$
(M_{{\rm h}},\omega_{\rm h}, T_{\rm h}) \hookrightarrow (G \times_H M_{{\rm h}}, \Omega, T) \to G/H
$
with base  $G/H$ being a torus bundle
over a torus, and symplectic toric manifolds $(M_{{\rm h}},\omega_{\rm h}, T_{\rm h})$ as fibers.  The $T$\--action on $G \times_H M_{\rm h}$ is the symplectic action by translations on the $T$\--factor of $G$, 
and the symplectic form $\Omega$ is given pointwise by formula {\rm (\ref{A*sigma})}.
\end{theorem}

\begin{proof}[Sketch of proof]
The map from $F \colon G \times_H M_{\textup{h}}$ to $M$ given by
$
((t,\,\xi),\,x) \mapsto t \cdot \textup{e}^{L_{\xi}}(x) 
$
is a $T$\--equivariant symplectomorphism. 
\end{proof}

In view of Lemma~\ref{pt}, Theorem~\ref{ppp}, and Proposition~\ref{p},  it is not difficult to verify that
$F$ is a $T$\--equivariant diffeomorphism and $F^*\omega=\Omega$ since the way we have
arrived at the model $(G \times_H M_{{\rm h}}, \Omega, T)$ of $(M,\omega,T)$ is constructive. However careful
checking is still fairly technical and not necessarily illuminating on a first reading; we refer
interested readers to \cite{DuPe} for a full proof.

\begin{figure}[htbp]
  \begin{center}
    \includegraphics[height=8cm, width=6cm]{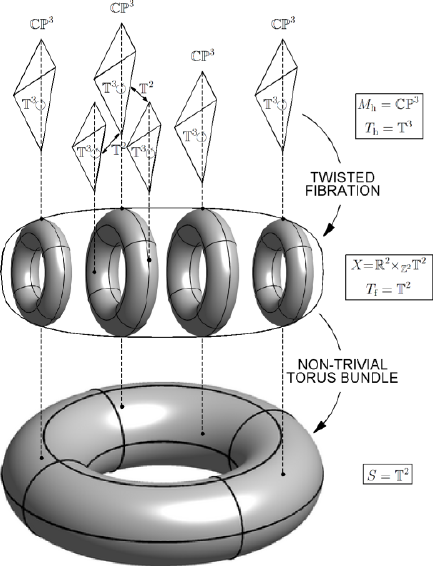}
    \caption{A $10$\--dimensional symplectic manifold with a torus
      action with Lagrangian orbits. The fiber is 
      $(\mathbb{CP}^3,\mathbb{T}^3)$. The base is a $\R^2 \times_{\Z^2} \mathbb{T}^2$.}
    \label{LO}
  \end{center}
\end{figure}

Notice that: 

(i) {if the $T$\--action is free}, then the Hamiltonian subtorus $T_{{\rm h}}$
is trivial, and hence $M$ is itself a torus bundle over a torus. Concretely,
$M$ is of the form $G/H$. The Kodaira variety (Example~\ref{ktexample})
is one of these spaces. Since $M$ is a principal torus 
bundle over a torus, it is a nilmanifold for a two\--step 
nilpotent Lie group as explained in  Palais\--Stewart \cite{ps}.  In the
case when this nilpotent Lie group is not abelian,  $M$
does not admit a K\"ahler structure, see Benson\--Gordon \cite{bg}.

(ii) In the case of  $4$\--dimensional manifolds $M$, item (i) corresponds to the third case 
in Kodaira's description \cite[Theorem~19]{kodaira} of the 
compact complex analytic surfaces which 
have a  holomorphic $(2,\, 0)$\--form that is nowhere 
vanishing, see \cite{DuPecois}\footnote{In~\cite{DuPecois} the authors show
 that  a compact connected symplectic 
$4$\--manifold with a symplectic $2$\--torus action admits an invariant complex structure and give
an identification of those that do not admit a K\"ahler structure with Kodaira's class of complex surfaces
which admit a nowhere vanishing holomorphic $(2, 0)$\--form, but are not a torus or a K3 surface.}. 
As mentioned, these were rediscovered by Thurston \cite{Th}
as the first examples of compact connected symplectic manifolds 
without K\"ahler structure.

(iii) {If on the other hand the $T$\--action is Hamiltonian}, then
$T_{{\rm h}}=T$, and in this case $M$ is itself a symplectic toric manifold an hence a 
toric variety (see~\cite{De, guillemin, DuPeTV} for
the relations between symplectic toric manifolds and toric varieties). 

(iv)
Henceforth,
{we may view the coisotropic orbit case as a twisted mixture of the Hamiltonian case,
and of the free symplectic case.

\begin{theorem}[\cite{DuPe}] \label{t2}
Compact connected symplectic manifolds $(M, \omega)$ with a coisotropic $T$\--action 
are determined up to $T$\--equivariant symplectomorphisms by: {\rm 1)} fundamental form
$\omega^{\mathfrak{t}} \colon \mathfrak{t} \times \mathfrak{t} \to
\mathbb{R}$;  {\rm 2)} Hamiltonian torus $T_{{\rm h}}$ and its associated 
polytope $\Delta$;  {\rm 3)}  period lattice $P$ of $N=(\mathfrak{l}/\mathfrak{t}_{{\rm h}})^*$;  
{\rm 4)}  Chern class $c \colon N \times N \to \mathfrak{l}$ of
$M_{\rm{reg}} \to M_{\rm{reg}}/T$; {\rm 5)} holonomy $[\tau \colon P \to T]_B 
\in {\rm Hom}_c(P,\, T)/B$. 

Moreover, for any list {\rm 1)\--5)} there exists a compact connected
symplectic manifold with a coisotropic $T$\--action whose list of invariants is precisely this one.
\end{theorem}

 {Theorem~\ref{t2} is analogous to Theorem~\ref{delzant}.
 
 The first part of Theorem~\ref{t2} is a uniqueness theorem. The last part is an existence
theorem for which we have not provided details for simplicity. Nonetheless we shall say that, for example, any antisymetric bilinear
form can appear as invariant 1), and any subtorus $S \subset T$ and Delzant polytope
can appear as ingredient 2) etc. This is explained in \cite{DuPe}.

\begin{example}
In the case of the Kodaira variety
$M=\R^2 \times_{\Z^2} (\mathbb{R}/\mathbb{Z})^2$ in Example~\ref{ktexample}, $T=(\R/\Z)^2$, 
$\mathfrak{t}\simeq \R^2$ and
its invariants are:  1) fundamental form: $\omega^{\mathfrak{t}}=0$; 2) Hamiltonian torus: $T_{{\rm h}}=\{[0,\,0]\}$; Delzant polytope: 
$\Delta=\{(0,\,0)\}$; 3) period lattice is $P=\mathbb{Z}^2$;  4) Chern class 
$
c \colon \mathbb{R}^2 \times \mathbb{R}^2 \to \mathbb{R}^2$,
defined by $c(e_1,\,e_2)=e_1$; 5) The holonomy  is the class of 
$\tau$
given by
$
\tau_{e_1}=\tau_{e_2}=[0,\,0]$. In this case
$G=(\mathbb{R}/\mathbb{Z})^2 \times \mathbb{R}^2$,
$M_{{\rm h}}=\{p\}$, and
$H=\{[0,\,0]\} \times \mathbb{Z}^2$. The model of $M$ is 
$
G \times_H M_{{\rm h}} \,\, \simeq \, \,
G/H \,\, \simeq \, \mathbb{R}^2 \times_{\mathbb{Z}^2} (\mathbb{R}/\mathbb{Z})^2.
$
\end{example}

\subsection{Symplectic $2$\--torus actions on $4$\--manifolds} \label{four:sec}

Consider on  $(\R/\Z)^2 \times S^2$  the product symplectic form.
The action of the $2$\--torus is: one circle acts on the first circle of $(\R/\Z)^2$ by translations,
while the other circle acts on $S^2$ by rotations about the vertical axis. 

If $T$ is a $2$\--dimensional
torus, consider the product  $T \times \mathfrak{t}^*$ with the standard cotangent bundle form 
and the standard $T$\--action on left factor of $T \times \mathfrak{t}^*$. If the symplectic\--toric
manifold $M_{\rm h}$ is trivial,  then the model for a symplectic $T$\--action with coisotropic orbits
simplifies greatly, and it splits into cases (1), (2) and (3) below.  

The following is a simplified version of the main result of \cite{Pe}; readers may consult \cite[Theorem~8.2.1]{Pe} for the 
complete version of the statement.

\begin{theorem}[\cite{Pe}] \label{main:thm}
Let $(M,\omega)$ be a compact, connected, symplectic $4$\--manifold equipped with an effective symplectic action of
a $2$\--torus  $T$. If the symplectic $T$\--action is Hamiltonian, then {\rm (1)} $(M, \, \omega)$ is a symplectic toric $4$\--manifold, and hence classified up to $T$\--equivariant
symplectomorphisms by the image $\Delta$ of the momentum map $\mu \colon M \to \mathfrak{t}^*$ of the $T$\--action.

If the symplectic $T$\--action is not Hamiltonian, then one and only one of the following cases occurs:
\begin{itemize}
\item[{\rm (2)}]
$(M, \, \omega)$ is equivariantly symplectomorphic to $(\R/\Z)^2 \times S^2$.
\item[{\rm (3)}]
$(M, \, \omega)$ is equivariantly symplectomorphic to $(T \times \mathfrak{t}^*)/Q$
with the induced form and $T$\--action,  where
$Q \le T \times \mathfrak{t}^*$ is a discrete cocompact subgroup for
the  group structure {\rm (\ref{nonabgroup})} on $T\times \mathfrak{t}^*$.
\item[{\rm (4)}]
$(M, \, \omega)$ is equivariantly symplectomorphic to a symplectic orbifold bundle
$\widetilde{\Sigma} \times_{\pi^{\textup{orb}}_1(\Sigma, \, p_0)}T
$
over a good orbisurface $\Sigma$, where the symplectic form and $T$\--action are induced by the product ones, 
and $\pi^{\textup{orb}}_1(\Sigma, \, p_0)$ acts on $\widetilde{\Sigma} \times T$ diagonally,
and on $T$ is by means
of any homomorphism $\mu \colon \pi_1^{\rm{orb}}(\Sigma) \to T$.
\end{itemize}
\end{theorem}

The proof of Theorem~\ref{main:thm} uses as stepping stones the maximal symplectic and coisotropic cases. 

\begin{proof}[Idea of proof]
The  fundamental observation to use the results in order to prove Theorem~\ref{main:thm} is that under the 
assumptions of Theorem~\ref{main:thm}, 
there are (by linear algebra of $\omega$) precisely two possibilities: (a)
 the $T$\--orbits are symplectic $2$\--tori, so the fundamental
form $\omega^{\mathfrak{t}}$ is non\--degenerate and hence $\mathfrak{l}$ is trivial. This 
corresponds to case 4); (b)  the $2$\--dimensional $T$\--orbits are Lagrangian $2$\--tori, and hence $\mathfrak{l}=\mathfrak{t}$.
  This corresponds to cases 1), 2), and 3). Case 3) is derived from Theorems~\ref{t1} and \ref{t2}.  Case 4) is derived from Theorems~\ref{t3} and \ref{t4}. 

A significant part of 
the proof of Theorem~\ref{main:thm} 
consists of unfolding item b) above into items 1), 2), 3) in the statement
of Theorem~\ref{main:thm}. 

Notice that item 1)  
is classified in terms of the Delzant polytope in view of Theorem~\ref{delzant}, which is the only invariant in the Hamiltonian case.
\end{proof}

Case 3) corresponds to the third case 
in the description of Kodaira \cite[Th. 19]{kodaira} of the 
compact complex analytic surfaces which carry a nowhere 
vanishing holomorphic $(2,\, 0)$\--form. These were rediscovered 
as the first examples of compact symplectic manifolds 
without K\"ahler structure by Thurston \cite{Th}.

The article \cite{DuPesymp} shows that the first Betti number of $M/T$ is equal to the first
Betti number of $M$ minus the dimension of $T$.

\begin{example}
The invariants of $M=S^2 \times_{\Z/2\,\Z} (\R/\Z)^2$  are: the non\--degenerate antisymmetric bilinear form
$
\omega^{\mathbb{R}^2}=\left( \begin{array}{cc}
0 & 1  \\
-1 & 0 
\end{array} \right)
$;  2) The Fuchsian signature $(g;\, \vec{o})=(0;\,2,\,2)$ of the orbit space $M/\T^2$; 3)
The symplectic area of  $S^2/(\mathbb{Z}/2\,\mathbb{Z})$: $1$ (half of the area of $S^2$); 4)
The monodromy invariant: 
$
 \mathcal{G}_{(0; \, 2, \,2)} \cdot (\mu_{\textup{h}}(\gamma_1), \, \mu_{\textup{h}}(\gamma_2))=
  \Big\{ \left( \begin{array}{cc}
1 & 0  \\
0 & 1 
\end{array} \right), \, 
  \left( \begin{array}{cc}
0 & 1  \\
1 & 0 
\end{array} \right)
                 \Big \} \cdot ([1/2, \, 0],\, [1/2, \, 0]).$
Here the $\gamma_1,\gamma_2$ are small loops around the poles of $S^2$. Then $M/T=S^2 / (\Z/2\,\Z)$,
$\pi^{\textup{orb}}_1(M/T,\,p_0)=\langle \gamma_1\, |\, \gamma_1^2=1 \rangle \simeq \Z/2\,\Z$, and
$\mu \colon \langle \gamma_1\, |\, \gamma_1^2=1 \rangle  \to T= (\R/\Z)^2$ is 
$
\mu(\gamma_1)=[1/2, \,0].
$
We have a $T$\--equivariant symplectomorphism
$
\widetilde{M/T} \times_{\pi^{\textup{orb}}_1(M/T,\,p_0)} T 
=\widetilde{S^2 / (\Z/2\,\Z)} \times_{\pi^{\textup{orb}}_1(S^2 / (\Z/2\,\Z),\,p_0)} (\R/\Z)^2 
\simeq M.
$
\end{example}

\section{Final remarks} \label{remarks}

In this paper we have covered symplectic {Hamiltonian} actions as contained in the works of 
Audin, Ahara, Hattori, Delzant, Duistermaat, Heckman
 Kostant, Atiyah, Guillemin, Karshon, Sternberg, Tolman, Weitsman~\cite{ AH,  A1, A2, DuHe, K, kostant, atiyah, gs, De} among others, 
and more general symplectic actions as in the works of Benoist, Duistermaat, Frankel, McDuff, Ortega, Ratiu, and the author \cite{benoist, benoistcorr, DuPe, ortegaratiu, Pe} among others. 

We have described classifications (on compact manifolds) in four cases:
(i) \emph{maximal Hamiltonian case}:  Hamiltonian $T$\--action, $\dim M=2\dim T$;
(ii)
 \emph{$S^1$\--Hamiltonian case}: Hamiltonian $T$\--action, $\dim M=4$,  $\dim T=1$;
(iii)
\emph{four\--dimensional case}: $\dim M=4$ and $\dim T=2$; 
(iv)
\emph{maximal symplectic case}: there is a $\dim T$\--orbit symplectic orbit;
(v)
\emph{coisotropic case}: there is a coisotropic orbit. We have outlined the connections of these works with
\emph{complex algebraic geometry}, in particular Kodaira's classification of complex analytic 
surfaces \cite{kodaira}), the theory of toric varieties and K\"ahler manifolds~\cite{DuPeTV}, 
and toric log symplectic\--geometry~\cite{GuLiPeRa}; \emph{geometric topology}, in particular the work of Palais\--Stewart~\cite{ps} and Benson\--Gordon \cite{bg} on torus bundles over tori and nilpotent  Lie groups; also with orbifold theory (for instance Thurston's classification of compact $2$\--dimensional orbifolds); and \emph{integrable systems}, in particular the work of Guillemin\--Sternberg on multiplicity\--free spaces \cite{multfree}
and semitoric systems~\cite{HoSaSe, PeVN09, PeVN11}.

Some of the techniques to study Hamiltonian torus actions (see for instance 
the books by Guillemin~\cite{guillemin}, Guillemin\--Sjamaar~\cite{gusj2005}, 
and  Ortega\--Ratiu~\cite{ORbook}) are useful in the study of 
non\--Hamiltonian symplectic torus actions (since many non\--Hamiltonian actions exhibit proper
subgroups which act Hamiltonianly). 

 In the study of Hamiltonian actions, one tool that  is often used is Morse theory for the (components of the) momentum map of the action. Since there is no momentum map in the classical sense for
a general symplectic action, Morse theory does not appear as a natural tool in the  non\--Hamiltonian case. 

There is an  analogue, however, ``{circle valued\--Morse theory}" (since any symplectic circle action 
admits a circle\--valued momentum map, see McDuff~\cite{MD} and \cite{PeRa}, which is also Morse in a sense) but  it is less 
immediately useful in our setting; for instance a more complicated form of the Morse inequalities holds 
(see Pajitnov~\cite[Chapter 11, Proposition 2.4]{Pajitnov2006} and Farber~\cite[Theorem 2.4]{Farber2004}),
and the theory appears more difficult to apply, at least in the context of non\--Hamiltonian symplectic actions; see \cite[Remark~6]{PeRa} for further discussion in this direction. This could be one reason that non\--Hamiltonian symplectic actions have been studied less in the literature than their Hamiltonian counterparts. 

The moduli space of coisotropic actions 
includes as a particular case Hamiltonian actions of maximal dimension (see \cite{PePiRaSa}
for the description of the moduli space of Hamiltonian actions of maximal dimension on $4$\--manifolds), classified
in Delzant's  article~\cite{De}.

We conclude with a general problem for further research:

\begin{problem} \label{lastp}
Let $T$ be an $m$\--dimensional torus (or even more generally, a compact Lie group). 
Give a classification of effective symplectic $T$\--actions on compact connected symplectic 
$2n$\--dimensional manifolds $(M,\omega)$. For instance, where $2n=4$ or $2n=6$.
\end{problem}

In this paper we have given an answer to this question under the additional assumptions in the 
cases (i)\--(v) above. Theorem~\ref{main:thm} give the complete answer when $m=2$ and $2n=4$ (using 
Theorem~\ref{delzant} for case 1) therein), under no additional assumptions. 

Current techniques are dependent on the additional assumptions (i.e. being Hamiltonian,
having some orbit of a certain type etc.) and solving Problem~\ref{lastp} in further cases
poses a challenge.

\section*{Dedication} This paper is dedicated to J.J. Duistermaat (1942--2010).

The memorial article \cite{GuPeVNWe} edited by V. Guillemin, \'A. Pelayo, S. V\~u Ng\d oc, and A. Weinstein outlines 
some of Duistermaat's most  influential contributions (see also \cite[Section 2.4]{pevn11}). Here is a brief extraction from that article: ``We are honored to pay tribute to Johannes (Hans) J. Duistermaat (1942-2010),
 a world leading figure in geometric analysis and one of the foremost Dutch mathematicians of the XX  century, by presenting a collection of contributions by some of Hans' colleagues, collaborators and students. Duistermaat's first striking contribution was his article ``Fourier integral operators II" with H\"ormander (published 
in Acta Mathematica), a work which he did after his doctoral dissertation. Several influential results in analysis and geometry have the name Duistermaat attached to them, for instance the Duistermaat-Guillemin trace formula
(1975), Duistermaat's global action-angle Theorem (1980), the Duistermaat\--Heckman Theorem (1982) and
the Duistermaat\--Grunbaum Bi\--spectral Theorem (1986). Duistermaat's papers offer an unusual display of
originality and technical mastery."

\medskip
\noindent
\'Alvaro Pelayo\\
Department of Mathematics\\
University of California, San Diego\\
9500 Gilman Drive  $\#$ 0112\\
La Jolla, CA 92093-0112, USA\\
E\--mail: alpelayo@math.ucsd.edu

\end{document}